\documentclass[11pt,oneside,reqno]{amsart}

\usepackage{graphicx}
\usepackage{pict2e}
\usepackage{amssymb}
\usepackage{amsthm}                
\usepackage[margin=1in]{geometry}  
\usepackage{graphics}
\usepackage{epstopdf}
\usepackage{amsmath}
\usepackage{graphicx}
\usepackage{subfigure}
\usepackage{latexsym}
\usepackage{amsmath}
\usepackage{amsfonts}
\usepackage{amssymb}
\usepackage{tikz}
\usepackage{mathdots}
\usepackage{mathtools}
\usepackage{hyperref}
\usepackage{comment}
\usepackage{mathtools}
\usepackage{float}
\usepackage[mathscr]{euscript}
\usepackage{enumerate}
\usepackage{epsfig}
\usepackage { graphicx }
\usepackage { hyperref }
\usepackage { graphics }
\usepackage { graphicx }

\theoremstyle{definition}
\newtheorem{theorem}{Theorem}[section]
\newtheorem{lemma}[theorem]{Lemma}
\newtheorem{corollary}[theorem]{Corollary}
\newtheorem{proposition}[theorem]{Proposition}

\theoremstyle{definition}
\newtheorem{definition}[theorem]{Definition}
\newtheorem{example}[theorem]{Example}
\theoremstyle{remark}
\newtheorem{remark}[theorem]{Remark}
\numberwithin{equation}{section}

\numberwithin{equation}{section}

\begin{document}

\title{The tail of a quantum spin network}

\author{Mustafa Hajij}
\address{Department of Mathematics, Louisiana State University, 
Baton Rouge, LA 70803 USA}
\email{mhajij1@math.lsu.edu}





\begin{abstract} The tail of a sequence $\{P_n(q)\}_{n \in \mathbb{N}}$ of formal power series in $\mathbb{Z}[[q]]$ is the formal power series whose first $n$ coefficients agree up to a common sign with the first $n$ coefficients of $P_n$. This paper studies the tail of a sequence of admissible trivalent graphs with edges colored $n$ or $2n$. We use local skein relations to understand and compute the tail of these graphs. We also give product formulas for the tail of such trivalent graphs. Furthermore, we show that our skein theoretic techniques naturally lead to a proof for the Andrews-Gordon identities for the two variable Ramanujan theta function as well to corresponding identities for the false theta function.
\end{abstract}

 \maketitle
 
 \tableofcontents

\section{Introduction}
The colored Jones polynomial is a link invariant that produces a sequence of Laurent polynomials in one variable with integer coefficients. In \cite{Dasbach} Dasbach and Lin showed that for an alternating link the absolute values of the first three and the last three leading coefficients of the colored Jones polynomial are independent of $n$, provided $n$ is sufficiently large. Let $\mathcal{P}=\{P_n(q)\}_{n \in \mathbb{N}}$ be a sequence of formal power series in $\mathbb{Z}[[q]]$. The tail of the sequence $\mathcal{P}$- if it exists - is the formal power series in $\mathbb{Z}[[q]]$ whose first $n$ coefficients agree up to a common sign with the first $n$ coefficients of $P_n$. In \cite{Cody} Armond and Dasbach introduced the head and the tail of the colored Jones polynomial of an alternating link. These two link invariants which take the form of formal $q$-series with integer coefficients. The existence of these two power series was conjectured by Dasbach and Lin \cite{Dasbach} and was proven by Armond in \cite{Cody2}. Higher stability of the coefficients of the colored Jones polynomial of an alternating link was shown by Garoufalidis and Lˆe in \cite{klb}. Calculations of the tail of the colored Jones polynomial were done by a number of authors, see Armond and Dasbach \cite{Cody}, Garoufalidis and Lˆe \cite{klb}, and Hajij \cite{Hajij}. Recently, Garoufalidis and Vuong  \cite{klb2} have given an algorithm to compute the tails of the colored Jones polynomial of alternating links.\\

Skein theoretic techniques have been used in \cite{Cody} and \cite{Cody2} to understand the head and tail of an alternating link. Write $\mathcal{S}(S^{3})$ to denote the Kauffman Bracket Skein Module of $S^{3}$. Let $L$ be an alternating link and let $D$ be a reduced link diagram of $L$. It was proven in \cite{Cody} that for an adequate link $L$ the first $(n+1)$ coefficients of $n^{th}$ unreduced colored Jones polynomial coincide with the first $(n+1)$ coefficients of the evaluation in $\mathcal{S}(S^{3})$ of a certain quantum spin network obtained from the link diagram $D$. Hence, studying the tail of the colored Jones polynomial can be reduced to studying the tail of these quantum spin networks.\\

A quantum spin network is a banded trivalent graph with edges labeled by non-negative integers, also called the colors of the edges, and the three edges meeting at a vertex satisfy some admissibility conditions. The main purpose of this paper is to understand the tail of a sequence of planner quantum spin network with edges colored $n$ or $2n$. Our method to study the tail of such graphs relies mainly on adapting various skein theoretic identities to new ones that can in turn be used to compute and understand the tail of such graphs. Studying the tail of these graphs via local skein relations does not only give an intuitive method to compute the tails but also demonstrates certain equivalence between the tails of different quantum spin networks as well as the existence of the tail of graphs that are not necessarily derived from alternating links.\\ 

The $q$-series obtained from knots in this way appear to be connected to classical number theoretic identities. Hikami \cite{hikami} realized that that Rogers-Ramanujan identities appear in the study of the
colored Jones polynomial of torus knots. In \cite{Cody2} Armond and Dasbach calculate the head and the tail of the colored Jones polynomial via multiple methods and use these computations to prove number theoretic identities. In this paper we show that the skein theoretic techniques we developed herein can be also used to prove classical identities in number theory. In particular we use skein theory to prove the Andrews-Gordon identities for the two variable Ramanujan theta function, as well as corresponding identities for the false theta function.

\subsection{Plan of the paper} In \ref{NT}, we give the  definitions for the number theoretic functions that we will appear and list some of their properties. In \ref{ST}, we give the necessary skein Theory background. In section \ref{3}, we discuss the existence of the tail of a sequence of skein elements. In section \ref{section4}, we study the tail of a sequence of quantum spin networks via various local skein relations. In section \ref{section5}, we give product formulas for the tail. In section \ref{section6}, we apply our results to study the tail of the colored Jones polynomial and we show that linear skein theory can be used to prove both the Andrews-Gordon identity for the theta function and a corresponding identity for the false theta function.
\subsection{Acknowledgments}I would like to express my gratitude to Oliver Dasbach for his advice and patience. I am also grateful to to Pat Gilmer for teaching me skein theory. I also appreciate the help of Moshe Cohen, Hany Hawasly, Kyle Isvan, Robert Osburn, Eyad Said, and Anastasiia Tsvietkova for offering a number of helpful comments.
\section{Background}
\label{section2}
\subsection{Number Theory}
\label{NT}
In this section we give the definitions of the general Ramanujan theta function and false theta functions and we list some of their properties.\\

Recall that $q$-Pochhammer symbol, denoted $(a;q)_n$, is defined by 
\begin{equation*}
(a;q)_n=\prod\limits_{j=0}^{n-1}(1-aq^j).
\end{equation*}
The $q$-binomial coefficients are given by
\begin{equation*}
\label{qb}
{n \brack i}_{q}=\frac{(q;q)_n}{(q;q)_i(q;q)_{n-i}}.
\end{equation*}
\begin{enumerate}
\item
The general two variable Ramanujan theta function, see \cite{andd2}, is defined by :
\begin{equation}
\label{oneone}
f(a,b)=\sum\limits_{i=0}^{\infty}a^{i(i+1)/2}b^{i(i-1)/2}+\sum\limits_{i=1}^{\infty}a^{i(i-1)/2}b^{i(i+1)/2}
\end{equation}
The definition of $f(a,b)$ implies
 \begin{eqnarray*}
 f(a,b)=f(b,a).
  \end{eqnarray*}
The Jacobi triple product identity of $f(a,b)$ is given by
\begin{equation*}
f(a,b)=(-a;ab)_{\infty}(-b,ab)_{\infty}(ab,ab)_{\infty},
\end{equation*}
It follows immediately from the Jacobi triple product identity that
\begin{eqnarray*}
 f(-q^2,-q)=(q;q)_{\infty}.
 \end{eqnarray*}
The function $f(a,b)$ specializes to (\ref{oneone})
\begin{equation}
f(-q^{2k},-q)=\sum\limits_{i=0}^{\infty}(-1)^i q^{k(i^2+i)}q^{i(i-1)/2}+\sum\limits_{i=1}^{\infty}(-1)^i q^{k(i^2-i)}q^{i(i+1)/2}.
\end{equation}
The Andrews-Gordon identity for the Ramanujan theta function is given by
\begin{eqnarray}
\label{and1}
f(-q^{2k},-q)=(q,q)_{\infty}\sum\limits_{l_1=0}^{\infty}\sum\limits_{l_{2}=0}^{\infty}...\sum\limits_{l_{k-1}=0}^{\infty}\frac{q^{\sum\limits_{j=1}^{k-1}(i_j(i_j+1))}}{\prod\limits_{j=1}^{k-1}(q,q)_{l_j}}
\end{eqnarray}
where $i_j=\sum\limits_{s=j}^{k-1}l_s$. This identity is a generalization of the second Rogers-Ramanujan identity
\begin{equation}
f(-q^{4},-q)=(q,q)_{\infty}\sum\limits_{i=0}^{\infty}\frac{q^{i^2+i}}{(q,q)_{i}}
\end{equation}
\item
The general two variable Ramanujan false theta function is given by (e.g. \cite{andd1}):
\begin{equation}
\Psi(a,b)=\sum\limits_{i=0}^{\infty}a^{i(i+1)/2}b^{i(i-1)/2}-\sum\limits_{i=1}^{\infty}a^{i(i-1)/2}b^{i(i+1)/2}
\end{equation}
In particular
\begin{equation}
\Psi(q^{2k-1},q)=\sum\limits_{i=0}^{\infty} q^{ki^2+(k-1)i}-\sum\limits_{i=1}^{\infty} q^{k(i^2-i)+i}
\end{equation}
We will show that the Andrews-Gordon identities (\ref{and1}) have corresponding identities for the false Ramanujan theta function:
\begin{eqnarray}
\label{fock2}
\Psi(q^{2k-1},q)=(q,q)_{\infty}\sum\limits_{l_1=0}^{\infty}\sum\limits_{l_{2}=0}^{\infty}...\sum\limits_{l_{k-1}=0}^{\infty}\frac{q^{\sum\limits_{j=1}^{k-1}(i_j(i_j+1))}}{(q,q)^2_{l_{k-1}}\prod\limits_{j=1}^{k-2}(q,q)_{l_j}}
\end{eqnarray}
where $i_j=\sum\limits_{s=j}^{k-1}l_s$. The latter identity is a generalization of the following identity (Ramanujan's notebook, Part III, Entry $9$ \cite{rama})
\begin{eqnarray}\Psi(q^{3},q)=(q,q)_{\infty}\sum\limits_{i=0}^{\infty}\frac{q^{i^2+i}}{(q;q)_i^2}
\end{eqnarray}
\end{enumerate}
Using skein theory, we recover and prove the identities (\ref{and1}) and (\ref{fock2}) in Theorem \ref{usingskein}.
\begin{remark}
The identities (\ref{and1}) and (\ref{fock2}) can be derived using $q$-series techniques. See \cite{robert}.  Nonetheless, we will give later a proof that we will give later utilizes completely different tools.
\end{remark}
\subsection{Skein Theory}
\label{ST}
Consider the $3$-manifold $F\times [0,1]$, where $F$ is a connected oriented surface. Denote by $ \mathbb{Q}(A)$ the field generated by the indeterminate $A$ over the rational numbers. Furthermore, set $q=A^4$.
\begin{definition}(J. Przytycki \cite{Przytycki} and V. Turaev \cite{tur})
 The \textit{Kauffman Bracket Skein Module} of $M=F\times [0,1]$, denoted by $\mathcal{S}(F)$, is the $ \mathbb{Q}(A)$ module generated by the set of isotopy classes of framed links in $M$(including the empty knot) modulo the submodule generated by the Kauffman relations \cite{Kauffman1}:
\begin{eqnarray*}(1)\hspace{3 mm}
  \begin{minipage}[h]{0.06\linewidth}
        \vspace{0pt}
        \scalebox{0.04}{\includegraphics{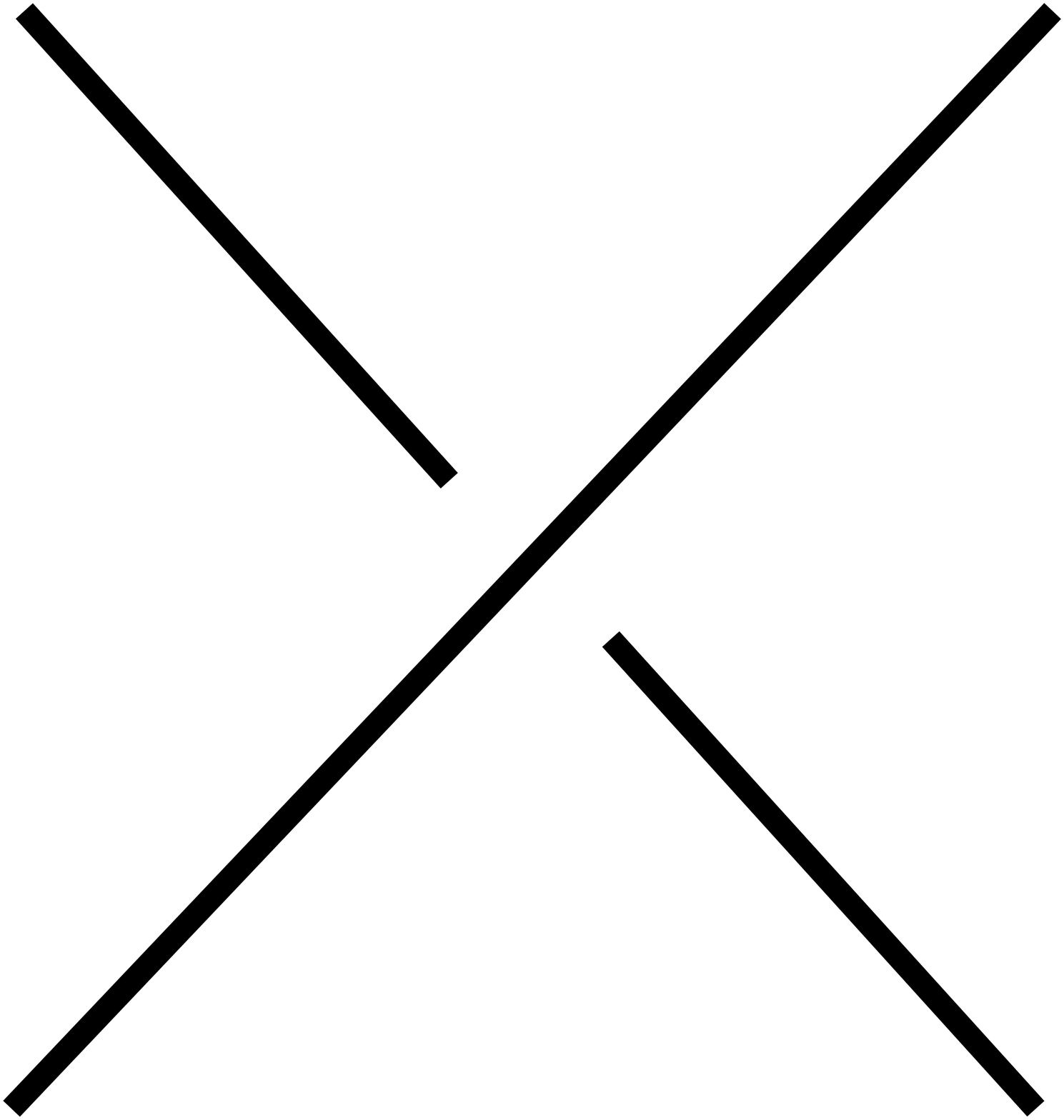}}
   \end{minipage}
   -
     A 
  \begin{minipage}[h]{0.06\linewidth}
        \vspace{0pt}
        \scalebox{0.04}{\includegraphics{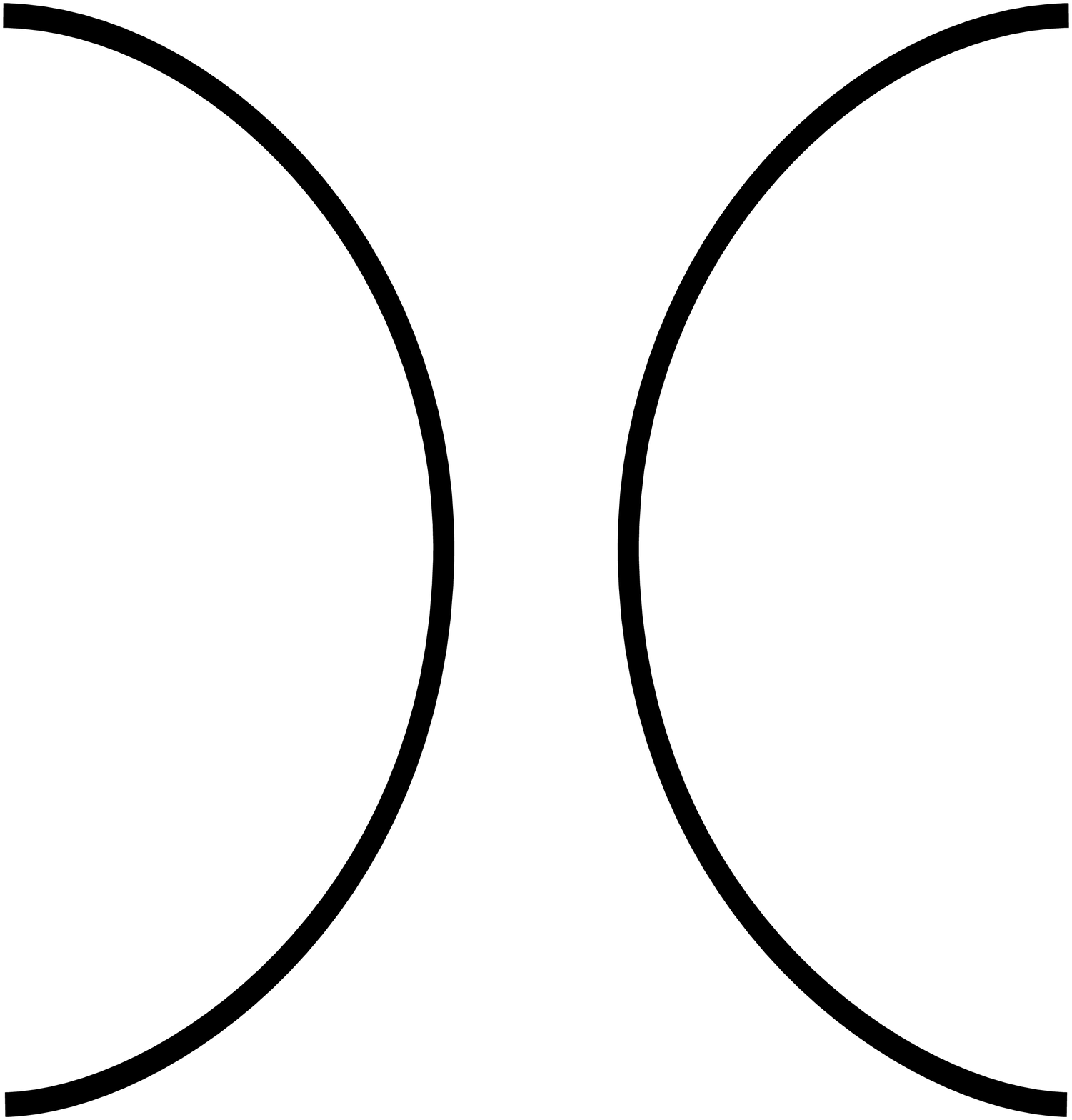}}
   \end{minipage}
   -
  A^{-1} 
  \begin{minipage}[h]{0.06\linewidth}
        \vspace{0pt}
        \scalebox{0.04}{\includegraphics{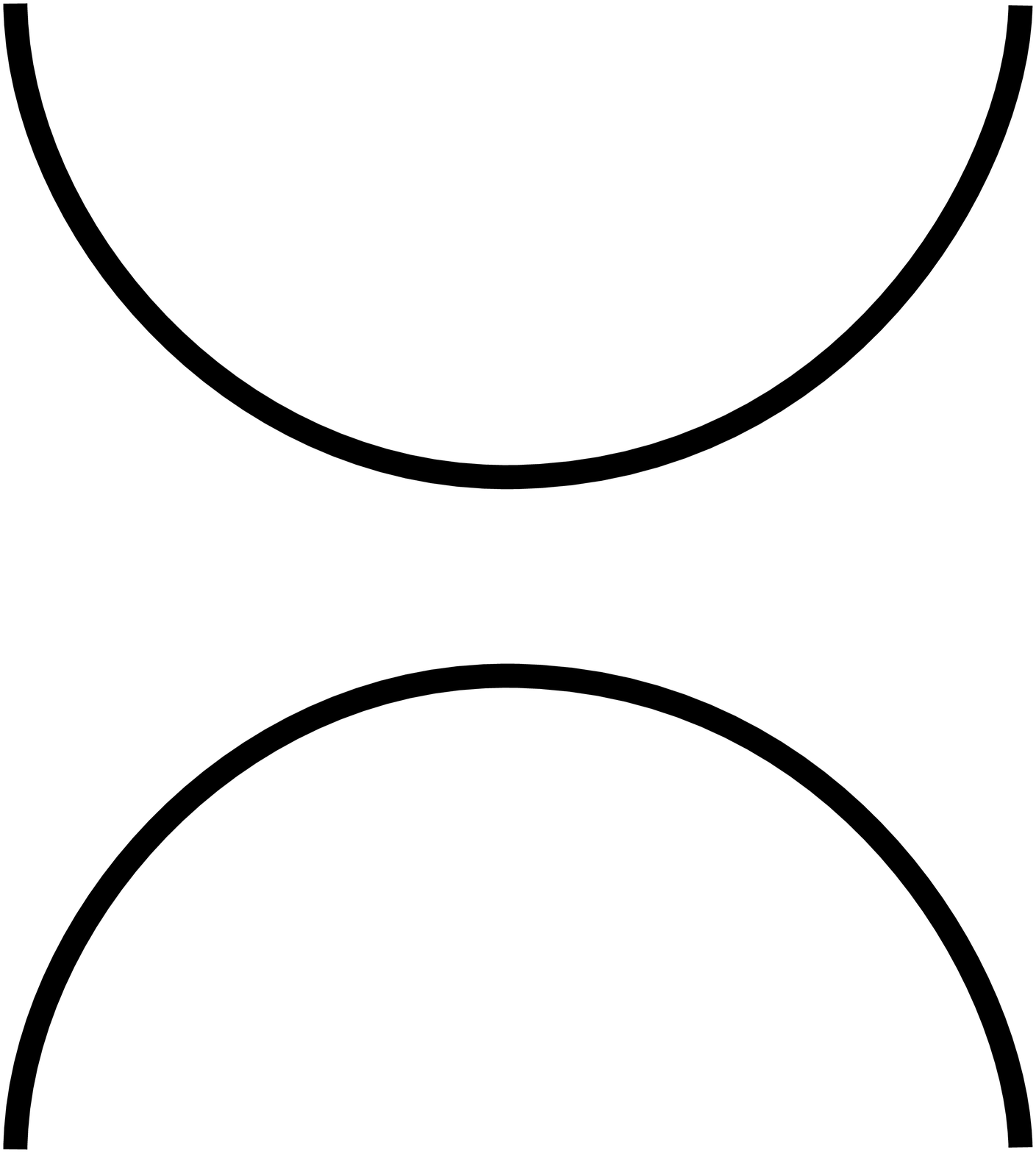}}
   \end{minipage}
, \hspace{20 mm}
  (2)\hspace{3 mm} L\sqcup
   \begin{minipage}[h]{0.05\linewidth}
        \vspace{0pt}
        \scalebox{0.02}{\includegraphics{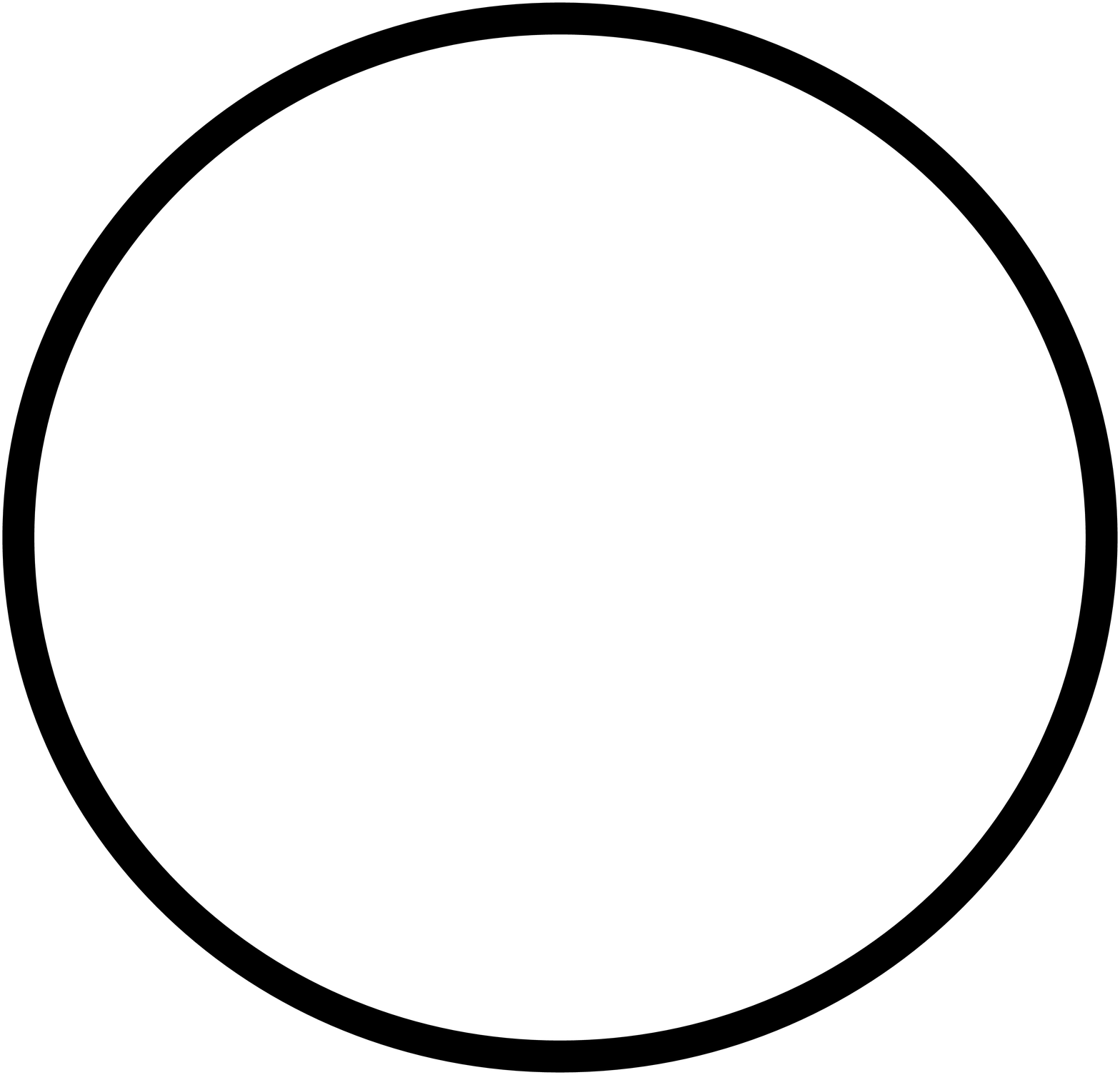}}
   \end{minipage}
  +
  (A^{2}+A^{-2})L. 
  \end{eqnarray*}
where $L\sqcup$ \begin{minipage}[h]{0.05\linewidth}
        \vspace{0pt}
        \scalebox{0.02}{\includegraphics{simple-circle}}
   \end{minipage}  consists of a framed link $L$ in $M$ and the zero-framed knot 
    \begin{minipage}[h]{0.05\linewidth}
        \vspace{0pt}
        \scalebox{0.02}{\includegraphics{simple-circle}}
   \end{minipage} that bounds a disk in $M$.
 \end{definition}    
Since we are considering $F\times [0,1]$ we may considers an appropriate version of link diagrams in $F$ to represent the classes of $\mathcal{S}(F)$ rather than framed links in $M$. A relative version of the Kauffman bracket skein module can be defined if $M$ has a boundary. In this case we specify a finite (possibly empty) set of points on the boundary of $M$. The relative module is the module generated by the set of isotopy classes of framed links together with arcs joining the designated boundary points of $M$ modulo the Kauffman relations specified above.\\

We will work with the skein module of the sphere $\mathcal{S}(S^{2})$. Let $D$ be any diagram in $\mathcal{S}(S^{2})$. The Kauffman relations imply that every crossing in $D$ and every unknot can be eliminated. Thus one can write $D =\langle D \rangle\phi$ in $\mathcal{S}(S^{2})$, where $\langle D \rangle\in \mathbb{Q}(A)$. The coefficient $\langle D \rangle$ is usually called the Kauffman bracket evaluation of $D$ in $\mathcal{S}(S^{2})$. In this case the Kauffman bracket provides an isomorphism $\mathcal{S}(S^{2})\cong\mathbb{Q}(A)$, induced by sending $D$ to $\langle D \rangle$. In particular it send the empty link to $1$. We will also work with the relative skein module $\mathcal{S}(D^2,2n)$, where the rectangular disk $D^2$ has $n$ designated points on the top edge and $n$ designated points on the bottom edge. This relative skein module can be thought of as the $Q(A)$-module generated by the crossingless matching diagrams of the $2n$ points of the disk $D^2$. The module $\mathcal{S}(D^2,2n)$ admits a multiplication given by vertical juxtaposition of two diagrams in $\mathcal{S}(D^2,2n)$. With this multiplication, $\mathcal{S}(D^2,2n)$ is an associative algebra over $Q(A)$ known as the \textit{$n^{th}$ Temperley-Lieb algebra} $TL_n$.\\

The $n^{th}$ \textit{Jones-Wenzl idempotent} (projector), denoted $f^{(n)}$, is a special element in $TL_n$ first discovered by Jones \cite{Jones}. This element and its properties will be used extensively in this paper and we shall adapt a common graphical presentation for this element which is due Lickorish \cite{Lickorish3}. In this graphical notation one thinks of $f^{(n)}$ as an empty box with $n$ strands entering and $n$ strands leaving the opposite side. See Figure \ref{JW}.
\begin{figure}[H]
  \centering
   {\includegraphics[scale=0.13]{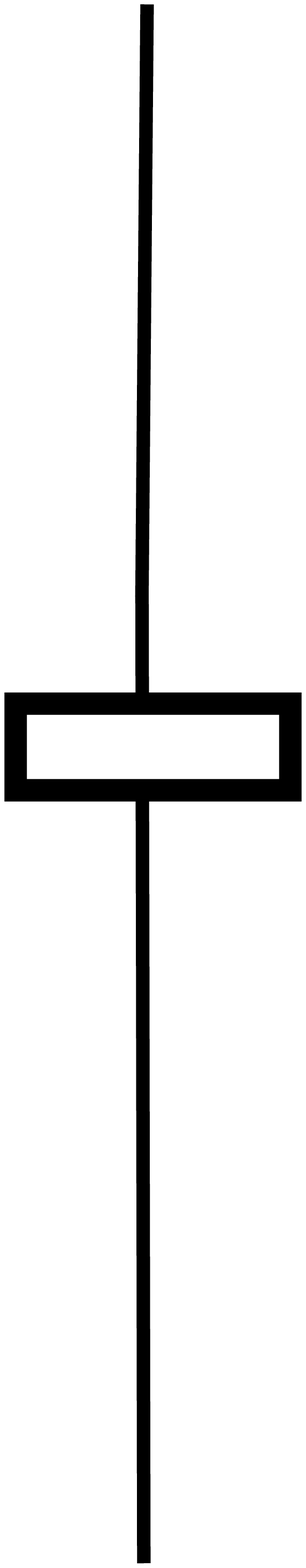}
   \put(-20,+70){\footnotesize{$n$}}
    \caption{The $n^{th}$ Jones-Wenzl idempotent}
  \label{JW}}
\end{figure}
The defining properties of $f^{(n)}$, given in the graphical notation mentioned earlier, are:
\begin{eqnarray}
\label{properties}
\hspace{0 mm}
    \begin{minipage}[h]{0.21\linewidth}
        \vspace{0pt}
        \scalebox{0.115}{\includegraphics{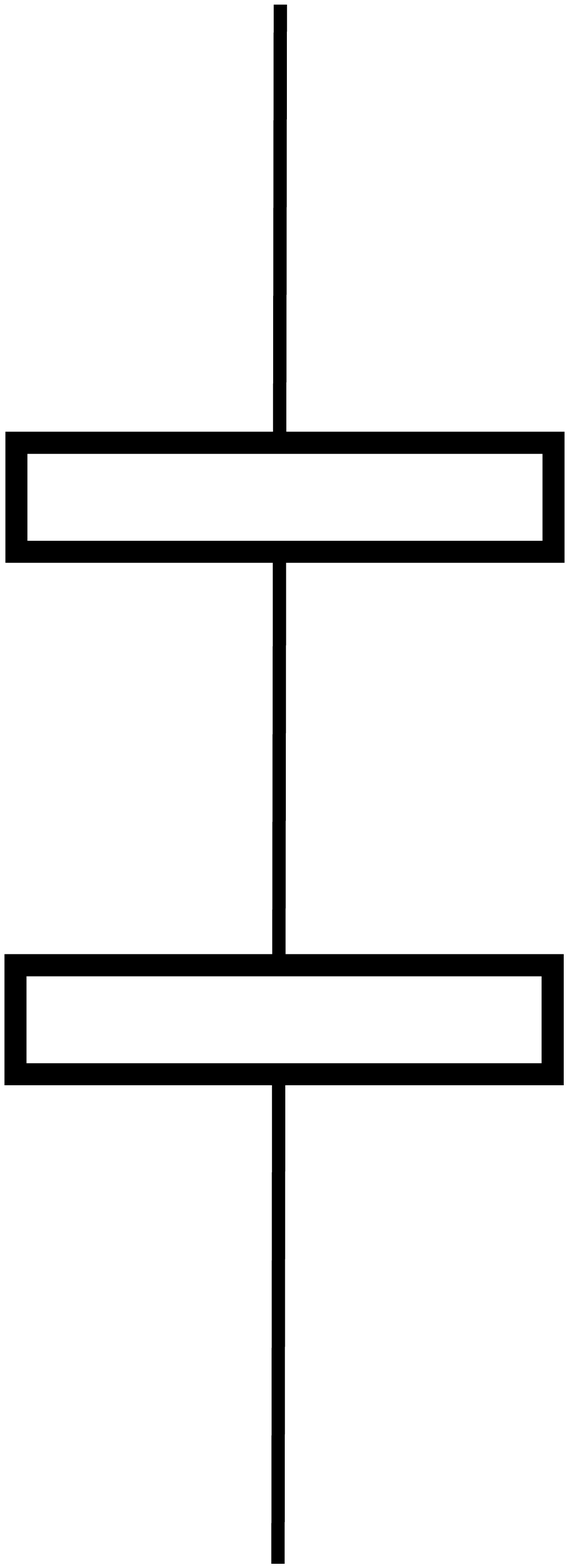}}
        \put(0,+80){\footnotesize{$n$}}
       
   \end{minipage}
  = \hspace{5pt}
     \begin{minipage}[h]{0.1\linewidth}
        \vspace{0pt}
         \hspace{100pt}
        \scalebox{0.115}{\includegraphics{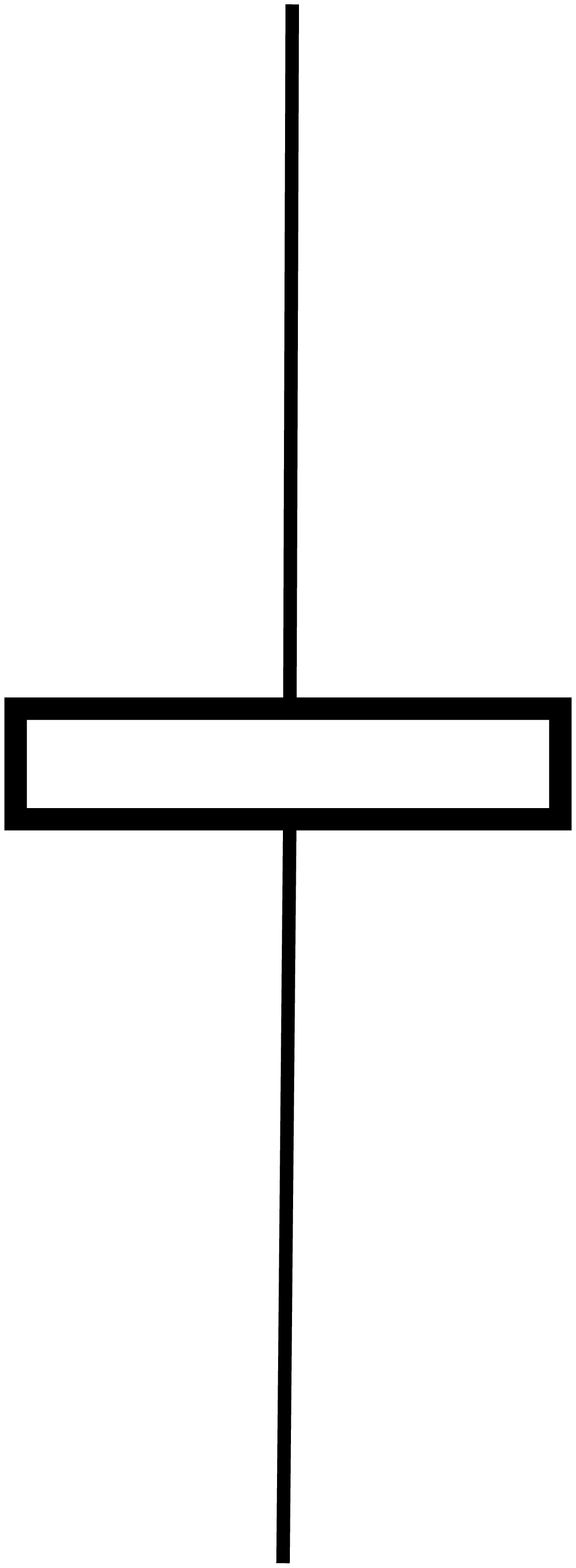}}
        \put(-60,80){\footnotesize{$n$}}
   \end{minipage}
    , \hspace{35 mm}
    \begin{minipage}[h]{0.09\linewidth}
        \vspace{0pt}
        \scalebox{0.115}{\includegraphics{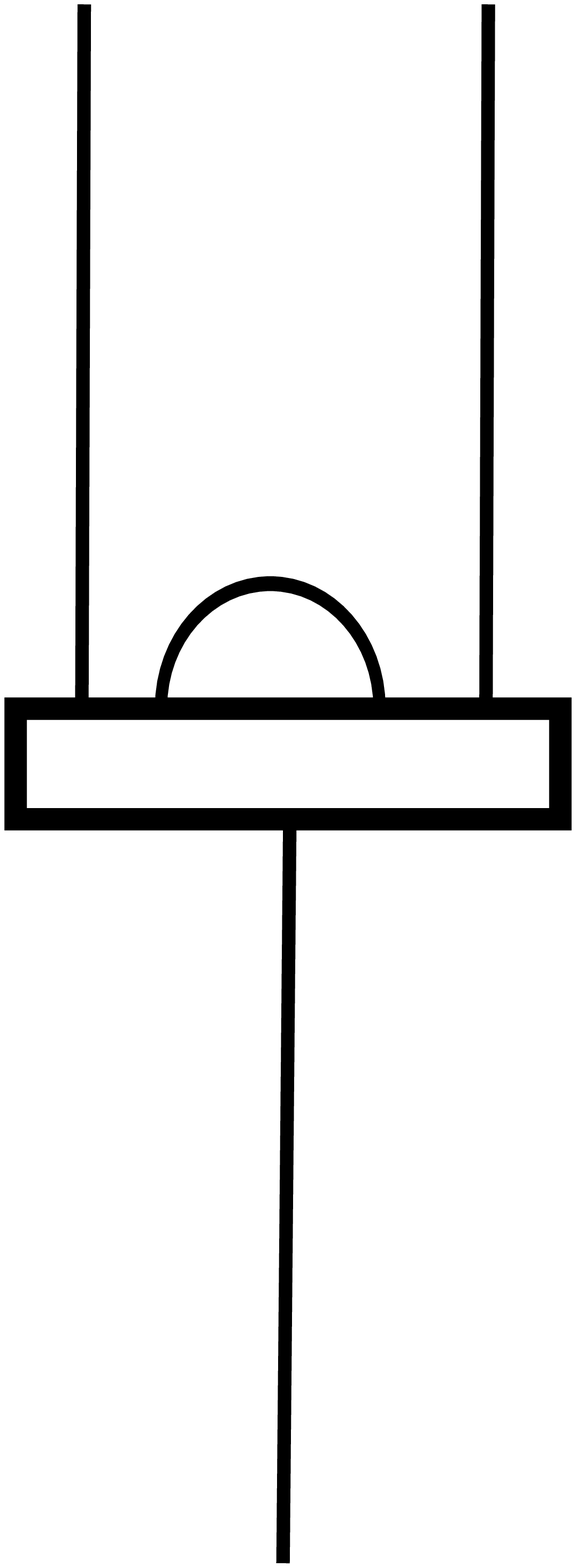}}
         \put(-70,+82){\footnotesize{$n-i-2$}}
         \put(-20,+64){\footnotesize{$1$}}
        \put(-2,+82){\footnotesize{$i$}}
        \put(-28,20){\footnotesize{$n$}}
   \end{minipage}
   =0.
   \label{AX}
  \end{eqnarray}
The second equation of \ref{AX} holds for $1\leq i\leq n-1$ and it is usually called \textit{the annihilation axiom}. The $n^{th}$ Jones-Wenzl idempotent has a recursive formula due to Wenzl \cite{Wenzl} and this formula is stated graphically as follows:   
\begin{align}
\label{recursive}
  \begin{minipage}[h]{0.05\linewidth}
        \vspace{0pt}
        \scalebox{0.12}{\includegraphics{nth-jones-wenzl-projector}}
         \put(-20,+70){\footnotesize{$n$}}
   \end{minipage}
   =
  \begin{minipage}[h]{0.08\linewidth}
        \hspace{8pt}
        \scalebox{0.12}{\includegraphics{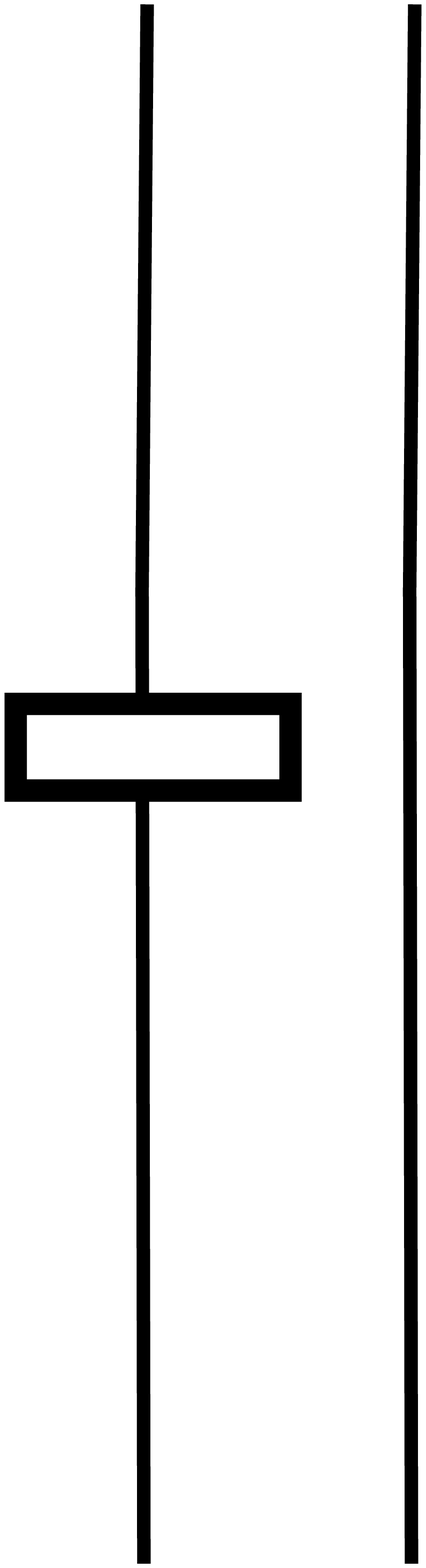}}
        \put(-42,+70){\footnotesize{$n-1$}}
        \put(-8,+70){\footnotesize{$1$}}
   \end{minipage}
   \hspace{9pt}
   -
 \Big( \frac{\Delta_{n-2}}{\Delta_{n-1}}\Big)
  \hspace{9pt}
  \begin{minipage}[h]{0.10\linewidth}
        \vspace{0pt}
        \scalebox{0.12}{\includegraphics{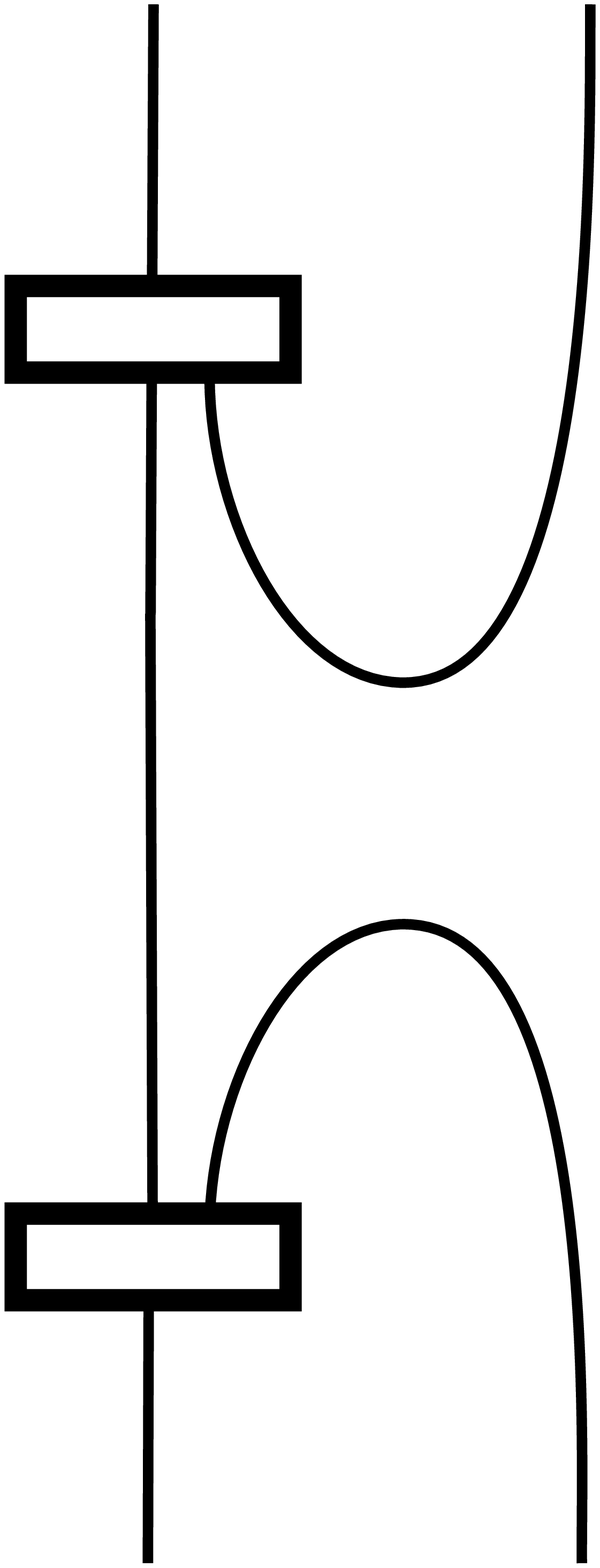}}
         \put(2,+85){\footnotesize{$1$}}
         \put(-52,+87){\footnotesize{$n-1$}}
         \put(-25,+47){\footnotesize{$n-2$}}
         \put(2,+10){\footnotesize{$1$}}
         \put(-52,+5){\footnotesize{$n-1$}}
   \end{minipage}
  , \hspace{20 mm}
    \begin{minipage}[h]{0.05\linewidth}
        \vspace{0pt}
        \scalebox{0.12}{\includegraphics{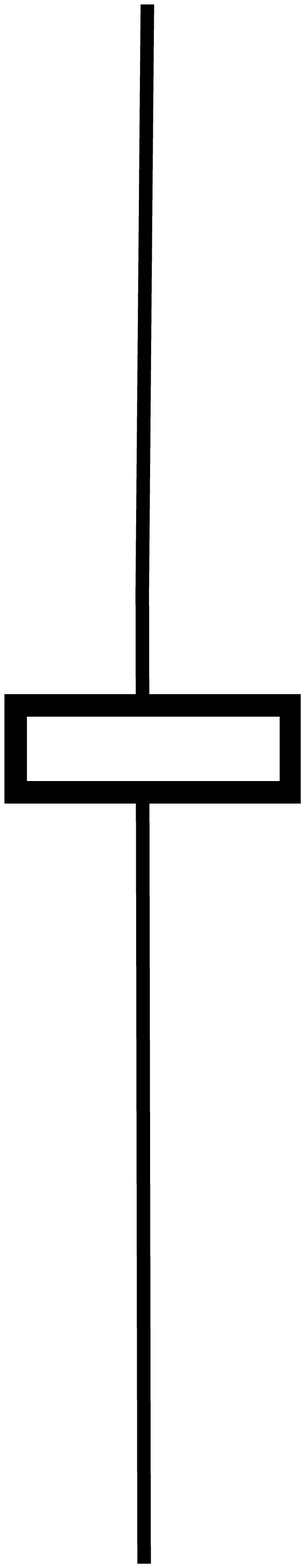}}
        \put(-20,+70){\footnotesize{$1$}}
   \end{minipage}
  =
  \begin{minipage}[h]{0.05\linewidth}
        \vspace{0pt}
        \scalebox{0.12}{\includegraphics{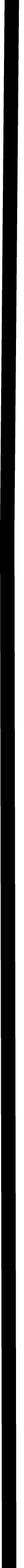}}
   \end{minipage}   
  \end{align}
  where 
\begin{equation*}
 \Delta_{n}=(-1)^{n}\frac{q^{(n+1)/2}-q^{-(n+1)/2}}{q^{1/2}-q^{-1/2}}.
\end{equation*}
The polynomial $\Delta_{n}$ is related to $[n+1]$, the $(n+1)^{th}$ quantum integer, by $\Delta _{n}=(-1)^{n}[n+1]$.
We assume that $f^{(0)}$ is the empty tangle. A proof of Wenzl's formula can be found in \cite{Lickorish1} and \cite{Wenzl}. The Jones-Wenzl idempotent satisfy the following identities
\begin{eqnarray}
\label{properties}
\hspace{0 mm}
\Delta_{n}=
  \begin{minipage}[h]{0.1\linewidth}
        \vspace{0pt}
        \scalebox{0.12}{\includegraphics{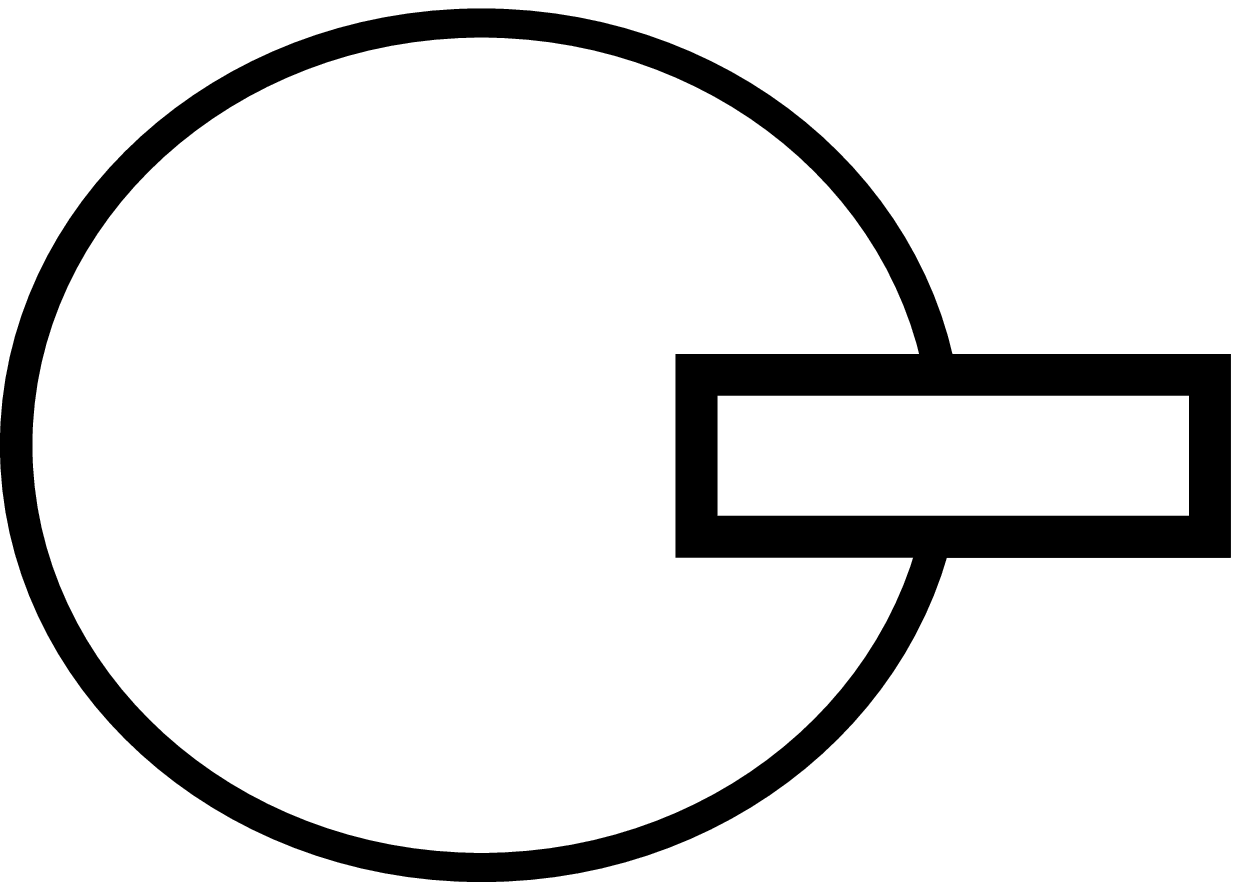}}
        \put(-29,+34){\footnotesize{$n$}}
   \end{minipage}
   , \hspace{19 mm}
     \begin{minipage}[h]{0.08\linewidth}
        \vspace{0pt}
        \scalebox{0.115}{\includegraphics{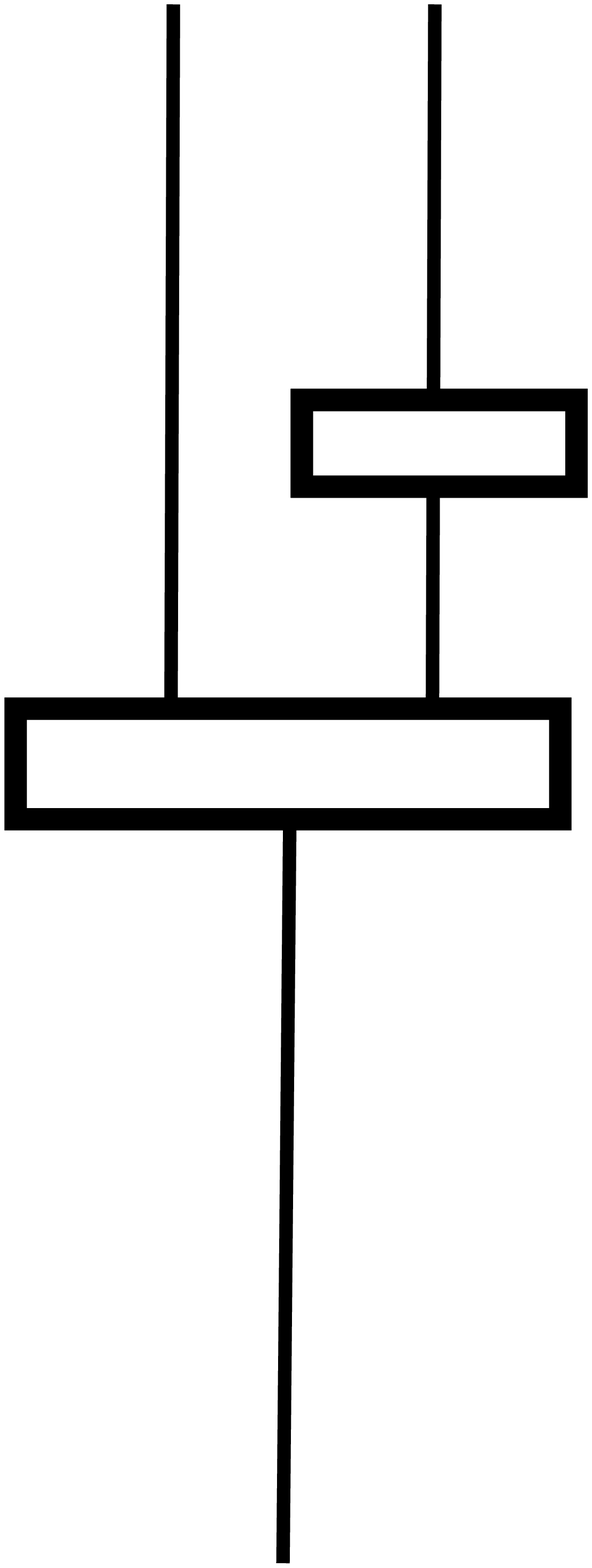}}
        \put(-34,+82){\footnotesize{$n$}}
        \put(-19,+82){\footnotesize{$m$}}
        \put(-46,20){\footnotesize{$m+n$}}
   \end{minipage}
  =\hspace{13pt}
     \begin{minipage}[h]{0.1\linewidth}
        \hspace{30pt}
        \scalebox{0.115}{\includegraphics{idempotent2}}
        \put(-9,82){\footnotesize{$m+n$}}
   \end{minipage}
  \end{eqnarray}
  and
  \begin{eqnarray}
\label{properties2}
    \begin{minipage}[h]{0.09\linewidth}
        \vspace{0pt}
        \scalebox{0.115}{\includegraphics{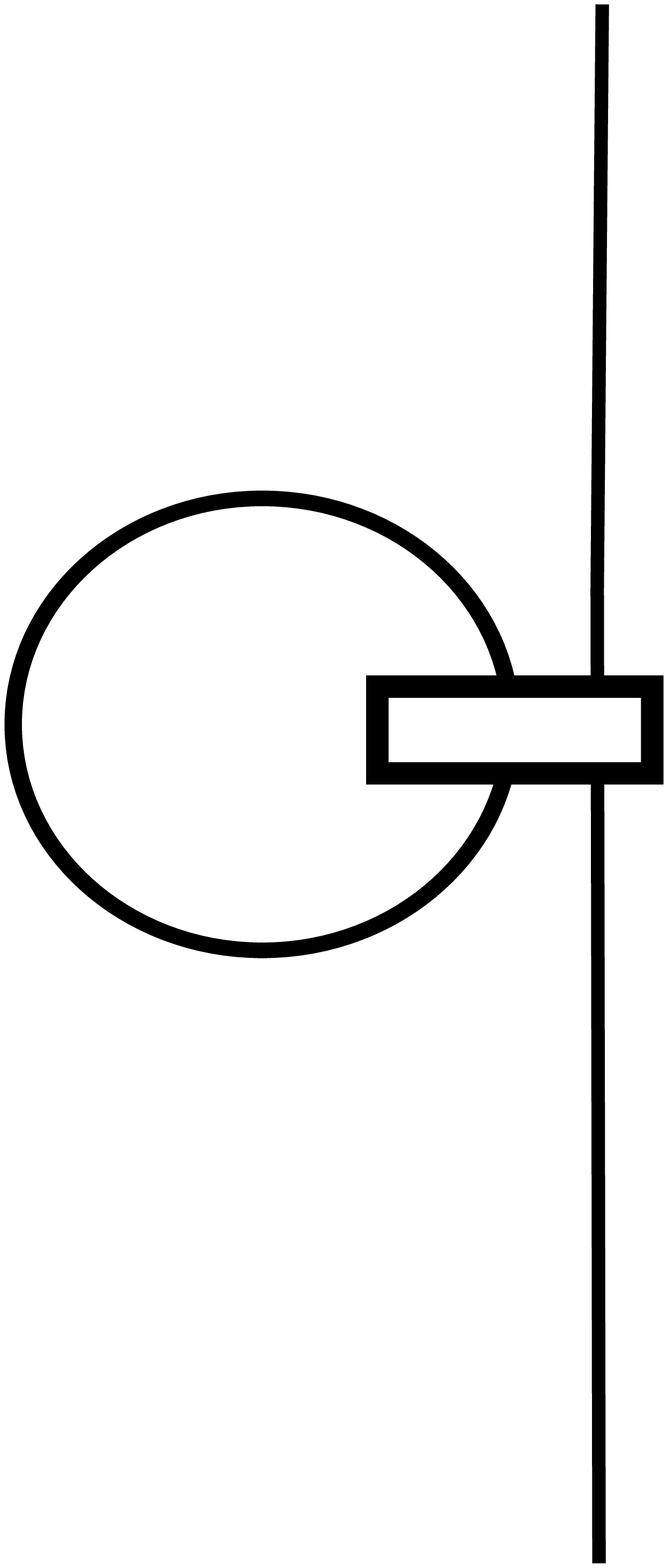}}
          \put(-13,+82){\footnotesize{$n$}}
        \put(-29,+70){\footnotesize{$m$}}
   \end{minipage}
   =\frac{\Delta_{m+n}}{\Delta_{n}}\hspace{1 mm}  
    \begin{minipage}[h]{0.06\linewidth}
        \vspace{0pt}
        \scalebox{0.115}{\includegraphics{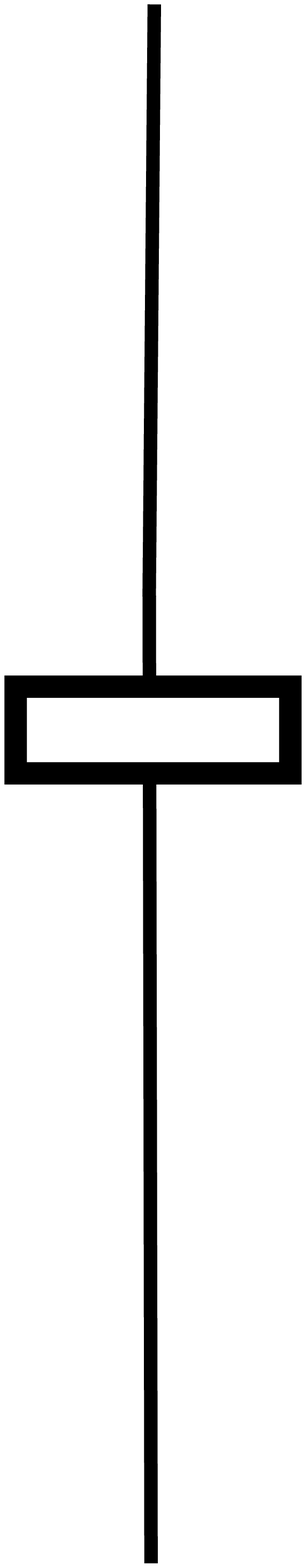}}
           \put(-19,+82){\footnotesize{$n$}}
   \end{minipage}
   \end{eqnarray}
Consider the relative skein module of the disk
with $a_{1}+...+a_{m}$\ \ marked points on the boundary. We are interested in a submodule of this module constructed as follows. Partition the set
of the $a_{1}+...+a_{m}$ points into $m$ sets of $a_{1},..,a_{m-1}$ and $a_{m}$ 
points respectively. At each cluster of these points we place an
appropriate idempotent, i.e. the one whose color matches the cardinality
of this cluster. We will denote this relative skein module by $T_{a_{1},...,a_{m}}.$ Hence an element in  $T_{a_{1},...,a_{m}}$ is obtained by taking an element in the module of the disk with $a_{1}+...+a_{m}$ marked points and then adding the idempotents  $f^{(a_1)},...,f^{(a_m)}$ on the outside of the disk. See Figure \ref{example of new element}. 

\begin{figure}[H]
  \centering
   {\includegraphics[scale=0.2]{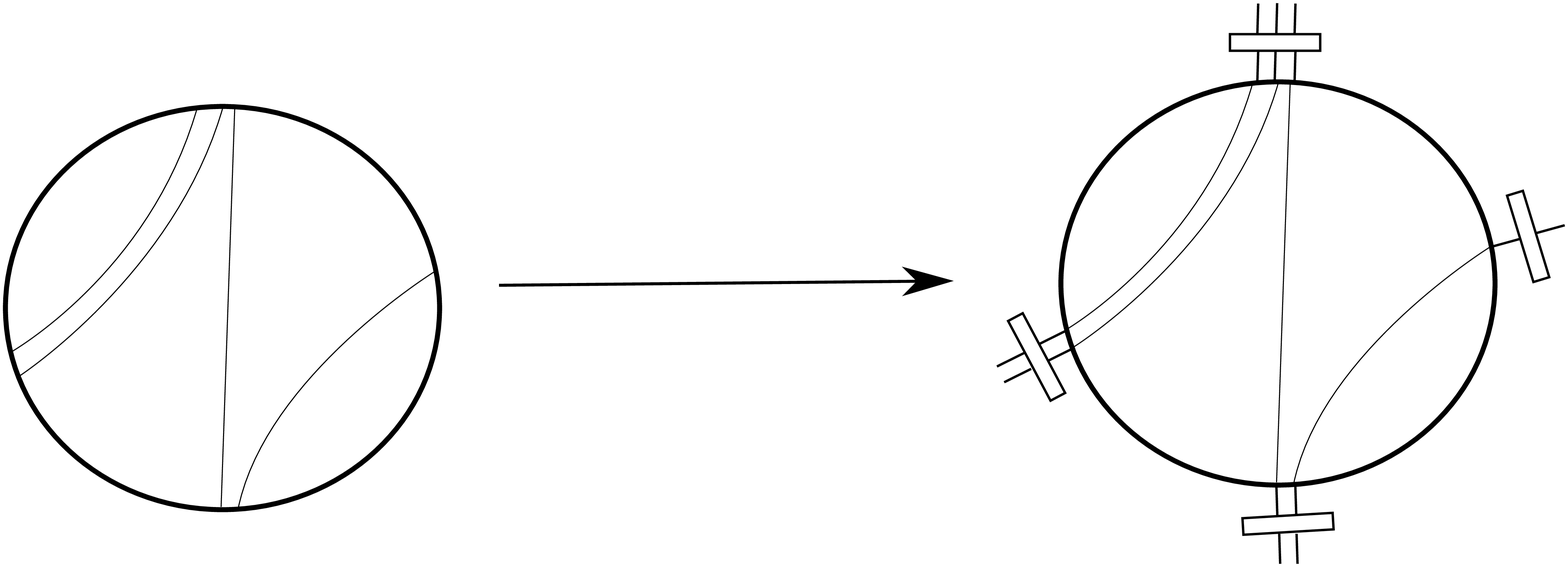}
     \caption{ An element in the Kauffman bracket skein module of the disk with $3+2+2+1$ marked points on the boundary and the corresponding element in the space $T_{3,2,2,1}.$ }
  \label{example of new element}}
\end{figure}

Of particular interest are the spaces $T_{a,b\text{ }}$,  $T_{a,b,c\text{ }}$ and $T_{a,b,c,d\text{ }}$. The properties of the Jones-Wenzl projector imply that the space $T_{a,b\text{ }}$ is zero dimensional when $%
a\neq b$ and one dimensional when $a=b$, spanned by $f^{(a)}$. Similarly, $T_{a,b,c\text{ }}$ is either zero dimensional or one
dimensional. The space $T_{a,b,c\text{ }}$ is one dimensional if and only if the element $%
\tau _{a,b,c}$ shown in Figure \ref{taw} exists. This occurs when one has
non-negative integers $x,y$ and $z$ such that the following three equations are satisfied
\begin{equation}
\label{threeequations}
a=x+y,\hspace{6pt} b=x+z,\hspace{6pt} c=y+z.
\end{equation}

\begin{figure}[H]
  \centering
   {\includegraphics[scale=0.25]{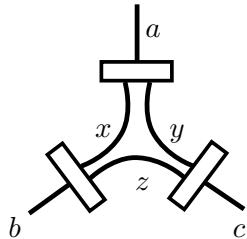}
    \put(-90,-8){$b$}
         \put(-38,68){$a$}
          \put(-57,30){$x$}
          \put(-29,30){$y$}
          \put(-42,10){$z$}
         \put(-5,-8){$c$}
     \caption{The skein element $\tau_{a,b,c}$ in the space $T_{a,b,c\text{ }}$ }
  \label{taw}}
\end{figure}
When $%
\tau _{a,b,c}$ exists we will refer to the outside colors of $%
\tau _{a,b,c}$ by the colors $a,b$ and $c$ and to the inside colors of $%
\tau _{a,b,c}$ by the colors $x,y$ and $z$. The following definition characterizes the existence of the element $\tau _{a,b,c}$ in terms of the outside colors. 
\begin{definition}
\label{admi}
A triple of colors $(a,b,c)$ is \textit{admissible} if $a+b+c\equiv0(\mod 2)$ and $a+b\geq
c\geq |a-b|.$
\end{definition}
Note that if the triple $(a,b,c)$ is admissible, then writing 
$x=(a+b-c)/2$,
$y=(a+c-b)/2$, and 
$z=(b+c-a)/2$ we have that $x$,$y$ and $z$ satisfy the equations \ref{threeequations}. If the triple $(a,b,c)$ is not admissible then the space $T_{a,b,c}$ is zero dimensional.  The fact that the inside colors are determined by the outside colors allows
us to replace $\tau _{a,b,c}$ by a trivalent graph as follows:
\begin{figure}[H]

  \centering
   {\includegraphics[scale=0.25]{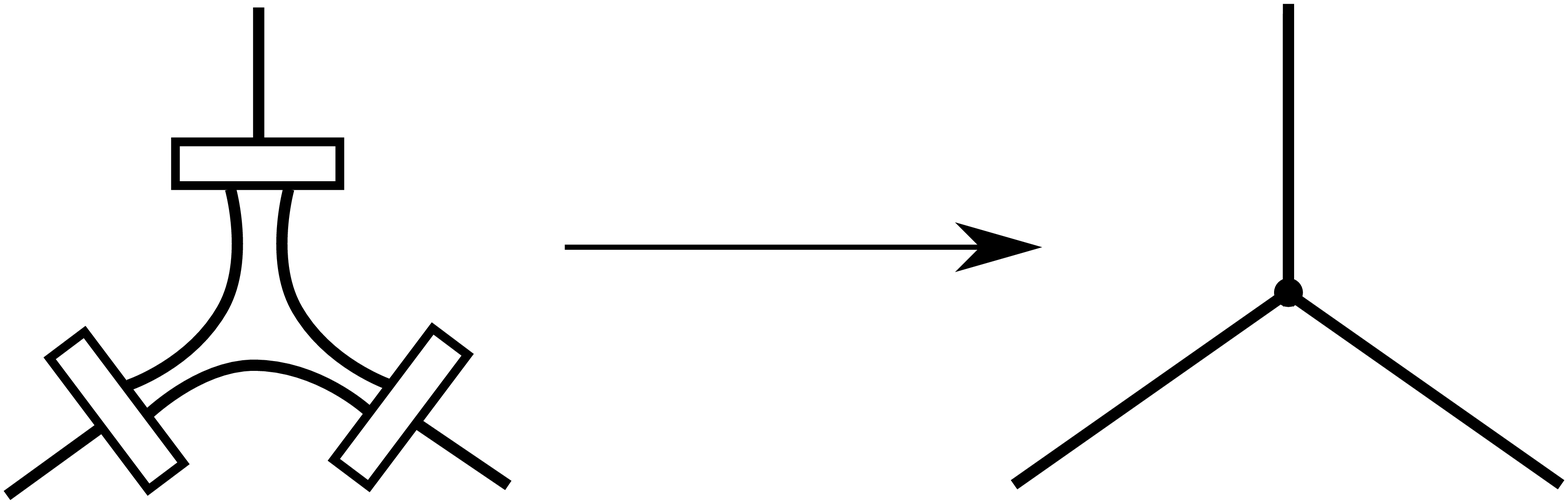}
    \put(-90,-8){$b$}
         \put(-38,68){$a$}
          \put(-227,30){$x$}
          \put(-199,30){$y$}
          \put(-212,10){$z$}
         \put(-5,-8){$c$}
            \put(-260,-8){$b$}
         \put(-205,68){$a$}
               \put(-169,-8){$c$}
     \caption{}
  \label{taw1}}
\end{figure}
A \textit{quantum spin network} is a planer trivalent graph with edges labeled by non-negative integers. A quantum spin network is called \textit{admissible} if each three edges meeting at a vertex satisfy the admissibility condition \ref{admi}. If $D$ is an admissible quantum spin network then the Kauffman bracket evaluation of $D$ is defined to be the evaluation of $D$ as an element in $\mathcal{S}(S^{2})$ after replacing each edge colored $n$ by the projector $f^{(n)}$ and each admissible vertex colored $(a,b,c)$ by the skein element $\tau_{a,b,c}$, as in Figure \ref{taw1}. \\

The skein space $T_{m,n,m^{\prime},m^{\prime}}$ with four clusters of points $m$, $n$, $m^{\prime}$ and $n^{\prime}$ is also of interest of us. In particular we are interested in the \textit{bubble skein element} in $T_{m,n,m^{\prime},m^{\prime}}$ shown in Figure \ref{bubble}.
\begin{figure}[H]
  \centering
   {\includegraphics[scale=0.13]{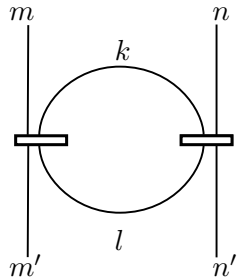}
   \put(-85,+90){$m$}
          \put(-8,+90){$n$}
          \put(-85,-7){$m^{\prime}$}
          \put(-8,-7){$n^{\prime}$}
           \put(-45,+75){$k$}
         \put(-45,+3){$l$}
  \caption{The bubble skein element $\mathcal{B}_{m^{\prime},n^{\prime}}^{m,n}(k,l)$}
  \label{bubble}}
\end{figure}
 We will denote this element by $\mathcal{B}_{m^{\prime},n^{\prime}}^{m,n}(k,l)$. In \cite{Hajij} we expanded this element in terms of some linearly independent elements in $T_{m,n,m^{\prime},m^{\prime}}$. We restate this result here.
\begin{theorem}(The bubble expansion formula \cite{Hajij})
\label{main}
Let $m,n,m^{\prime},n^{\prime} \geq 0$, and $k\geq l$,  $k,l \geq 1$. Then\\
{\small
\begin{eqnarray}
\label{bubble expansion formula1}
  \begin{minipage}[h]{0.15\linewidth}
        \vspace{0pt}
        \scalebox{0.11}{\includegraphics{second-lemma-main-bubble}}
        \put(-68,+80){$m$}
          \put(-8,+80){$n$}
          \put(-68,-7){$m^{\prime}$}
          \put(-8,-7){$n^{\prime}$}
           \put(-37,+67){$k$}
         \put(-37,+3){$l$}
   \end{minipage}
   =\displaystyle\sum\limits_{i=0}^{\min(m,n,l)}
   \left\lceil 
\begin{array}{cc}
m & n \\ 
k & l%
\end{array}%
\right\rceil _{i}
    \begin{minipage}[h]{0.15\linewidth}
        \vspace{0pt}
        \scalebox{0.11}{\includegraphics{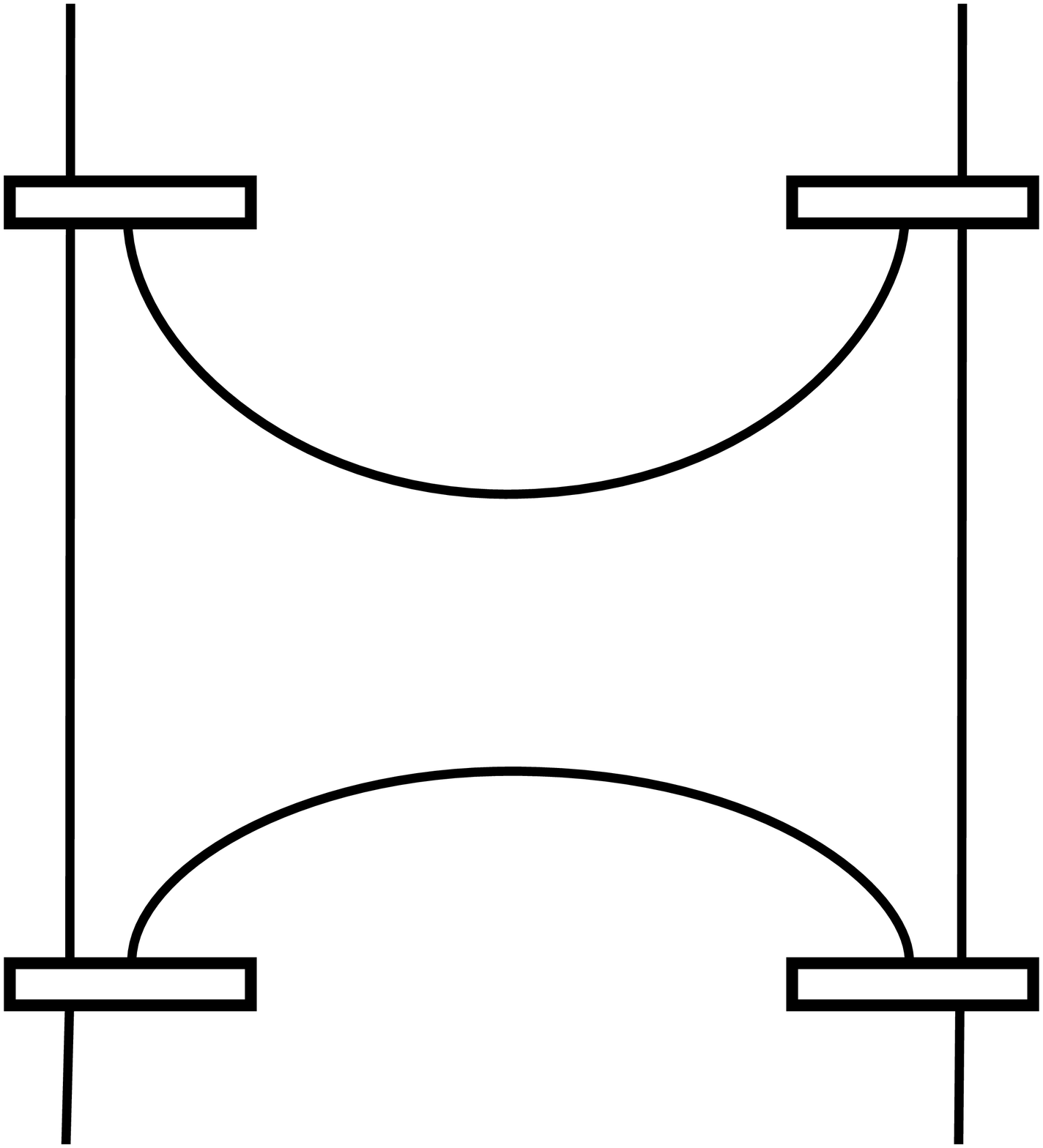}}
        \put(-68,+80){$m$}
          \put(-8,+80){$n$}
          \put(-68,-7){$m^{\prime}$}
          \put(-8,-7){$n^{\prime}$}
           \put(-37,+50){$i$}
         \put(-52,28){$k-l+i$}
   \end{minipage}
  \end{eqnarray}
}
where 
\begin{equation*}
\left\lceil 
\begin{array}{cc}
m & n \\ 
k & l%
\end{array}%
\right\rceil _{i}=(-1)^{i+l}q^{i(i-l)/2}\frac{\displaystyle\prod_{j=0}^{l-i-1}[k-j]\prod_{s=0}^{i-1}[n-s][m-s]}{\displaystyle\prod_{t=0}^{l-1}[n+k-t][m+k-t]}
{l \brack i}_{q}\prod_{j=0}^{l-i-1}[m+n+k-i-j+1]
\end{equation*}
\end{theorem}
\label{sectionmain}
\section{Existence of the tail of an adequate skein element}
\label{3}
In \cite{Cody} C. Armond proves that the tail of the colored Jones polynomial of alternating links exist. This was done by proving that the tail of the colored Jones polynomial of an alternating link $L$ is equal to the tail a sequence of certain skein elements in $\mathcal{S}(S^2)$ obtained from an alternating link diagram of $L$. In fact, Armond proved this for a larger class of links, called Adequate links. Following \cite{Cody}, we briefly recall the proof of existence of the tail the colored Jones polynomial and we illustrate how this can be applied to our study.\\

If $P_1(q)$ and $P_2(q)$ are Laurent series, we write $P_1(q)\doteq_n P_2(q)$ if their first $n$ coefficients agree up to a sign.
\begin{definition}
Let $\mathcal{P}=\{P_n(q)\}_{n \in \mathbb{N}}$ be a sequence of formal power series in $\mathbb{Z}[[q]]$. The tail of the sequence $\mathcal{P}$- if it exists - is the formal power series $T_{\mathcal{P}}(q) $ in $\mathbb{Z}[[q]]$ that satisfies
\begin{equation*}
\label{main defintion}
T_{\mathcal{P}}(q)\doteq_{n}P_n(q), \text{ for all } n \in \mathbb{N}.
\end{equation*}
\end{definition}
Observe that the tail of the sequence $\mathcal{P}=\{P_n(q)\}_{n \in \mathbb{N}}$ exists if and if only if $P_{n}(q)\doteq_{n}P_{n+1}(q)$ for all $n$. 
\begin{remark}
\label{main remark} 
Consider the sequence $\{f_n(q)\}_{n \in \mathbb{N}}$ where $f_n(q)$ is a rational function  in $\mathbb{Q}(q)$. After multiplying by $q^{\pm s}$ for some $s$, the rational function $f_n(q)$ can be viewed as a formal power series in $\mathbb{Z}[[q]]$. Using this convention one can study the tail of the sequence $\{f_n(q)\}_{n \in \mathbb{N}}$.
\end{remark}
Let $\mathcal{D}=\{D_n(q)\}_{n \in \mathbb{N}}$ be a sequence of skein elements in $\mathcal{S}(S^2)$. The evaluation of $D_n(q)$ gives in general a rational function. Using the observation in remark \ref{main remark}, one could study the tail of the of the sequence $\mathcal{D}$.\\

 Consider a crossingless skein element $D$ in $\mathcal{S}(S^2)$ consisting of arcs connecting Jones-Wenzl idempotents of various colors. Let $\bar{D}$ be the diagram obtained from $D$ by replacing each idempotent $f^{(n)}$ with the identity $id_{n}$ in $TL_n$. The diagram $\bar{D}$ thus consists of non-intersecting circles. We say that $D$ is adequate if each circle in $\bar{D}$ passes at most once through any given region where we replaced the idempotents in $D$. See Figure \ref{cody_1} for a local picture of an adequate skein element and note that the circle indicated in the figure bounds a disk. See also Figure \ref{adequate example}. 
 
 \begin{figure}[H]
  \centering
   {\includegraphics[scale=0.094]{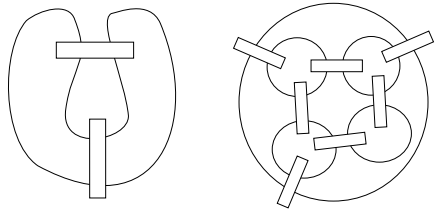}
    \caption{An inadequate skein element on the left and an adequate skein element on the right. All arcs are colored $n$.}
  \label{adequate example}}
\end{figure}

 Every alternating link diagram $L$ induces a family of adequate skein elements in the following way. For any crossing in $L$ there are two ways to smooth this crossing, the $A$-smoothing and the $B$-smoothing. See Figure \ref{smoothings}.
\begin{figure}[H]
  \centering
   {\includegraphics[scale=0.14]{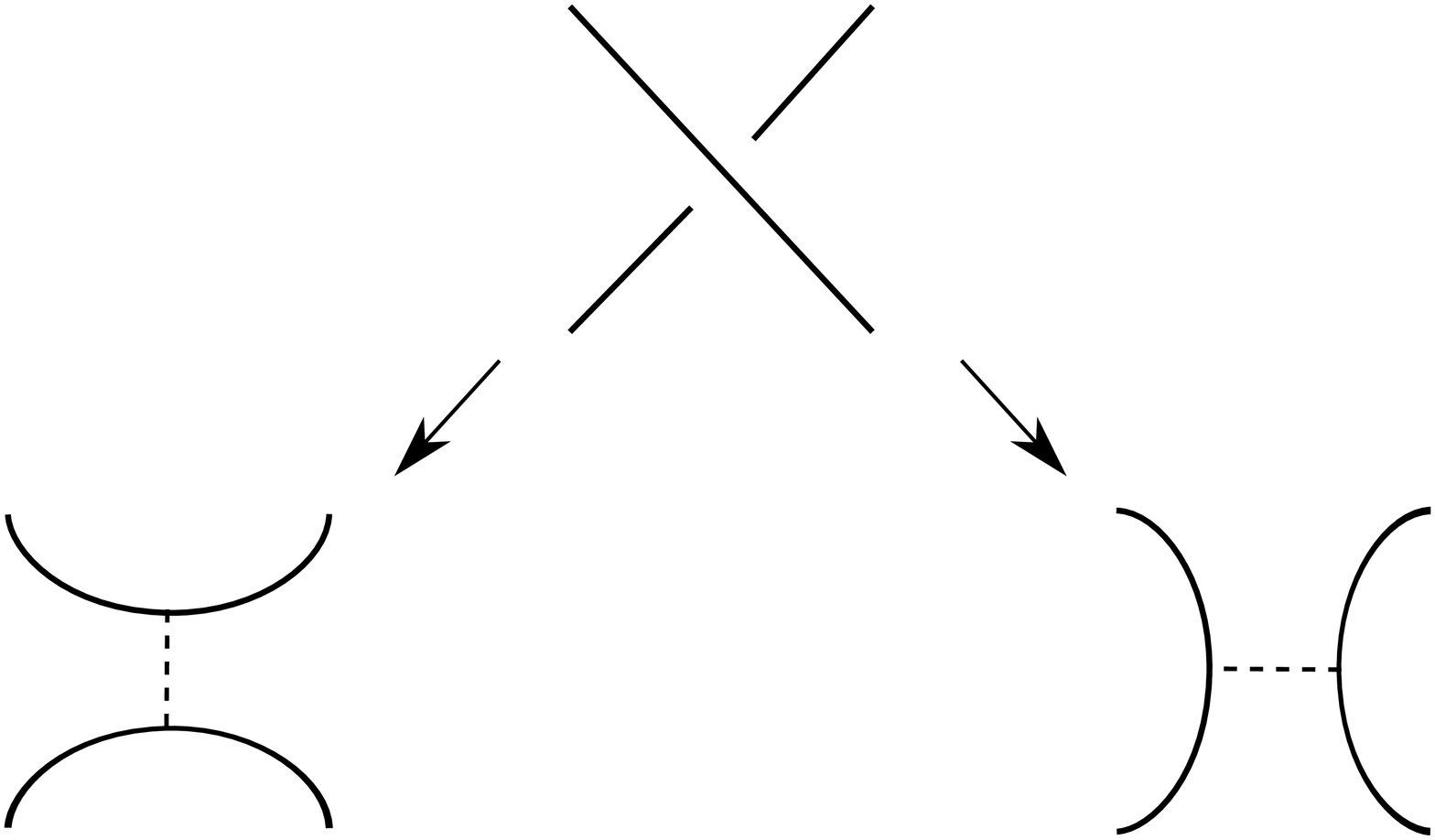}
    \put(-88,33){$A$}
          \put(-30,33){$B$}
    \caption{A and B smoothings}
  \label{smoothings}}
\end{figure}

We replace a crossing with a smoothing together with a dashed line joining the two arcs. After applying a smoothing to each crossing in $L$ we obtain diagram of a collection of disjoint circles in the plane called a \textit{Kauffman state} for the diagram $L$. The all-$A$ smoothing (all-$B$ smoothing) state of $L$ is the state obtained by replacing each crossing by an $A$ smoothing ($B$ smoothing). Write $S_A(D)$ and $S_B(L)$ to denote the all-$A$ smoothing and all-$B$ smoothing states of $L$ respectively. Now consider the skein element obtained from $S_B(L)$ by decorating each circle in $S_B(L)$ with the $n^{th}$ Jones-Wenzl idempotent and replacing each dashed line in $S_B(D)$ with the $(2n)^{th}$ Jones-Wenzl idempotent. See Figure \ref{allB1} for an example. Write $S_B^{(n)}(L)$ to denote this skein element.
\begin{figure}[H]
  \centering
   {\includegraphics[scale=0.08]{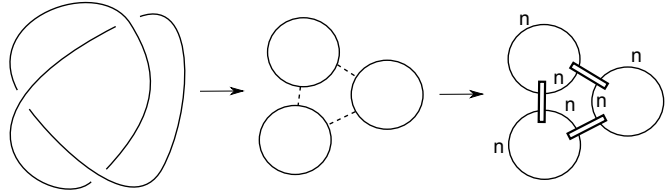}
    \caption{Obtaining $S_B^{(n)}(L)$ from a link diagram $L$}
    \label{allB1}
 }
\end{figure} 
Note that for any alternating link diagram $L$, the skein element $S_B^{(n)}(L)$ is adequate. C. Armond proves that the tail of family $\{S_B^{(n)}(L)\}_{n \in \mathbb{N}}$ exists by showing that $S_B^{(n+1)}(L)(q)\doteq_{n+1}S_B^{(n)}(L)(q)$ using three basic steps:

\begin{enumerate}
\item Since the link diagram $L$ is alternating one can observe that $S_B^{(n+1)}(L)$ is an adequate skein element and it actually looks locally like Figure \ref{cody_1}.
\begin{figure}[H]
  \centering
   {\includegraphics[scale=0.1]{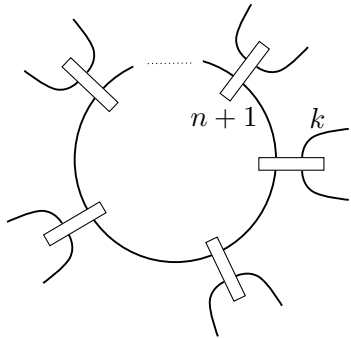}
    \put(-60,80){$n+1$}
          \put(-15,79){$k$}
    \caption{A local picture for $S_B^{(n+1)}(L)$}
    \label{cody_1}
 }
\end{figure} 

 Furthermore, we have the following equality 
\small{
\begin{eqnarray*}
   \begin{minipage}[h]{0.3\linewidth}
         \vspace{0pt}
         \scalebox{0.1}{\includegraphics{cody_1}}
         \put(-60,80){$n+1$}
          \put(-15,79){$k$}
   \end{minipage}&\doteq_{n+1}&
  \begin{minipage}[h]{0.16\linewidth}
        \vspace{0pt}
        \scalebox{0.1}{\includegraphics{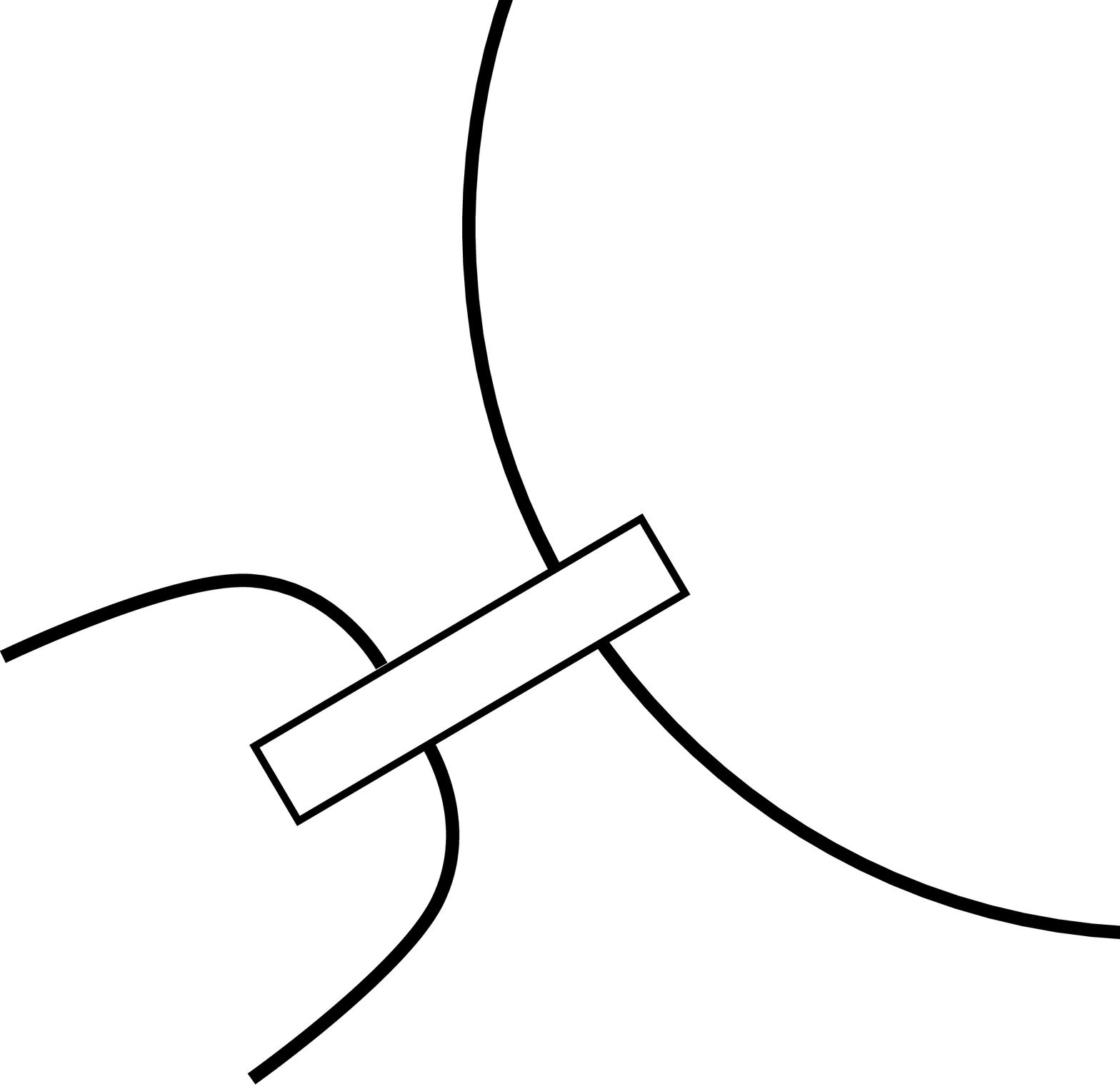}}
         \put(-48,62){$1$}
         \put(-27,77){$n$}
          \put(-15,79){$k$}
           \end{minipage}
  \end{eqnarray*}}
This is done by using the recursive definition of the Jones-Wenzl idempotent and showing that all of the other terms resulting from applying the recursive definition of the idempotent do not contribute to the first $n+1$ coefficients of $S_B^{(n+1)}(L)$.     
\item Step one can be applied around the circle until we reach  the final idempotent: 
\small{
\begin{eqnarray*}
   \begin{minipage}[h]{0.27\linewidth}
         \vspace{0pt}
         \scalebox{0.09}{\includegraphics{cody_2}}
          \put(-41,62){$1$}
         \put(-27,77){$n$}
          \put(-15,71){$k$}
   \end{minipage}\doteq_{n+1}
  \begin{minipage}[h]{0.27\linewidth}
        \vspace{0pt}
        \scalebox{0.09}{\includegraphics{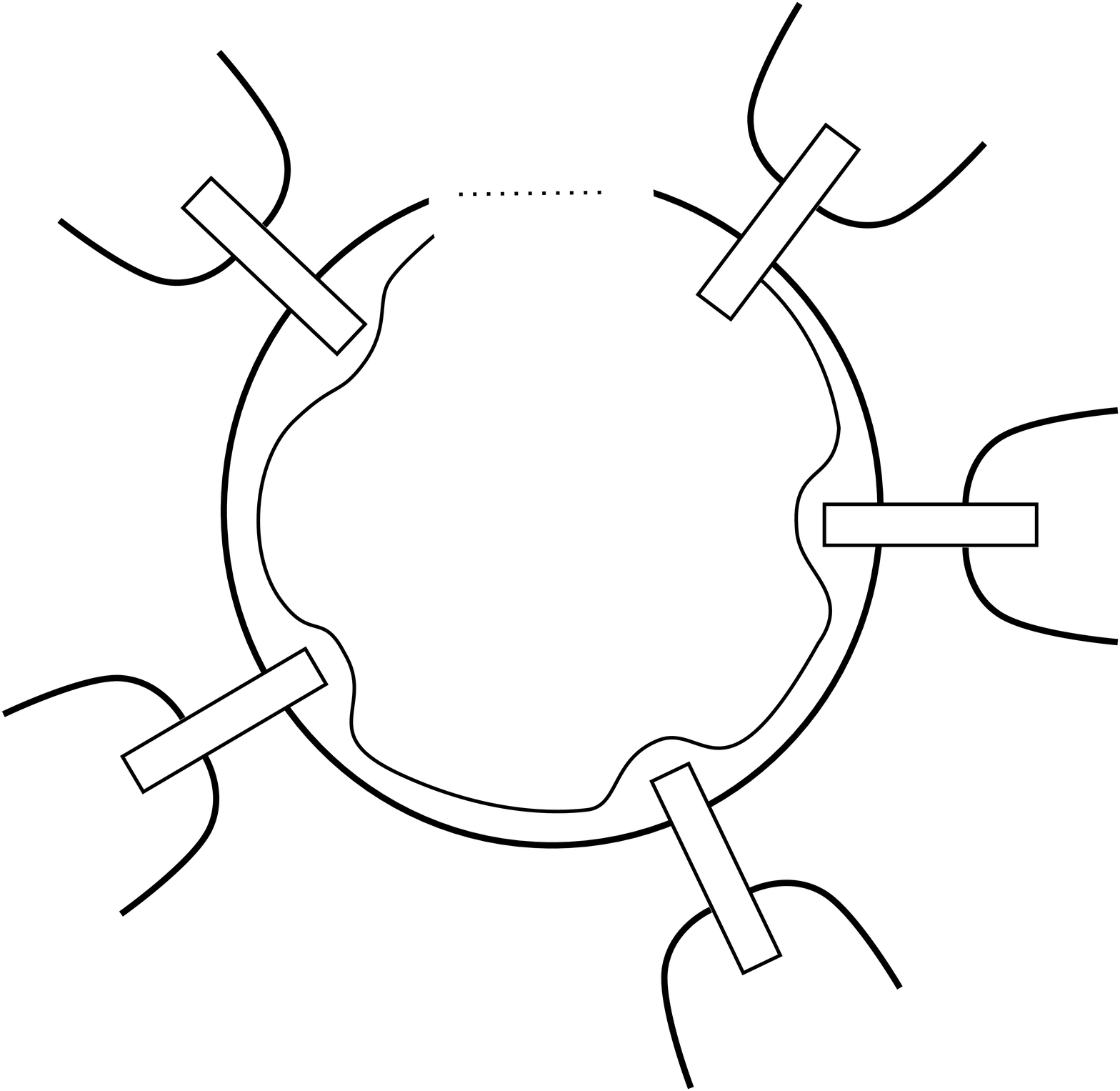}}
         \put(-41,62){$1$}
         \put(-27,77){$n$}
          \put(-15,71){$k$}
           \end{minipage}\doteq_{n+1}
  \begin{minipage}[h]{0.16\linewidth}
        \vspace{0pt}
        \scalebox{0.09}{\includegraphics{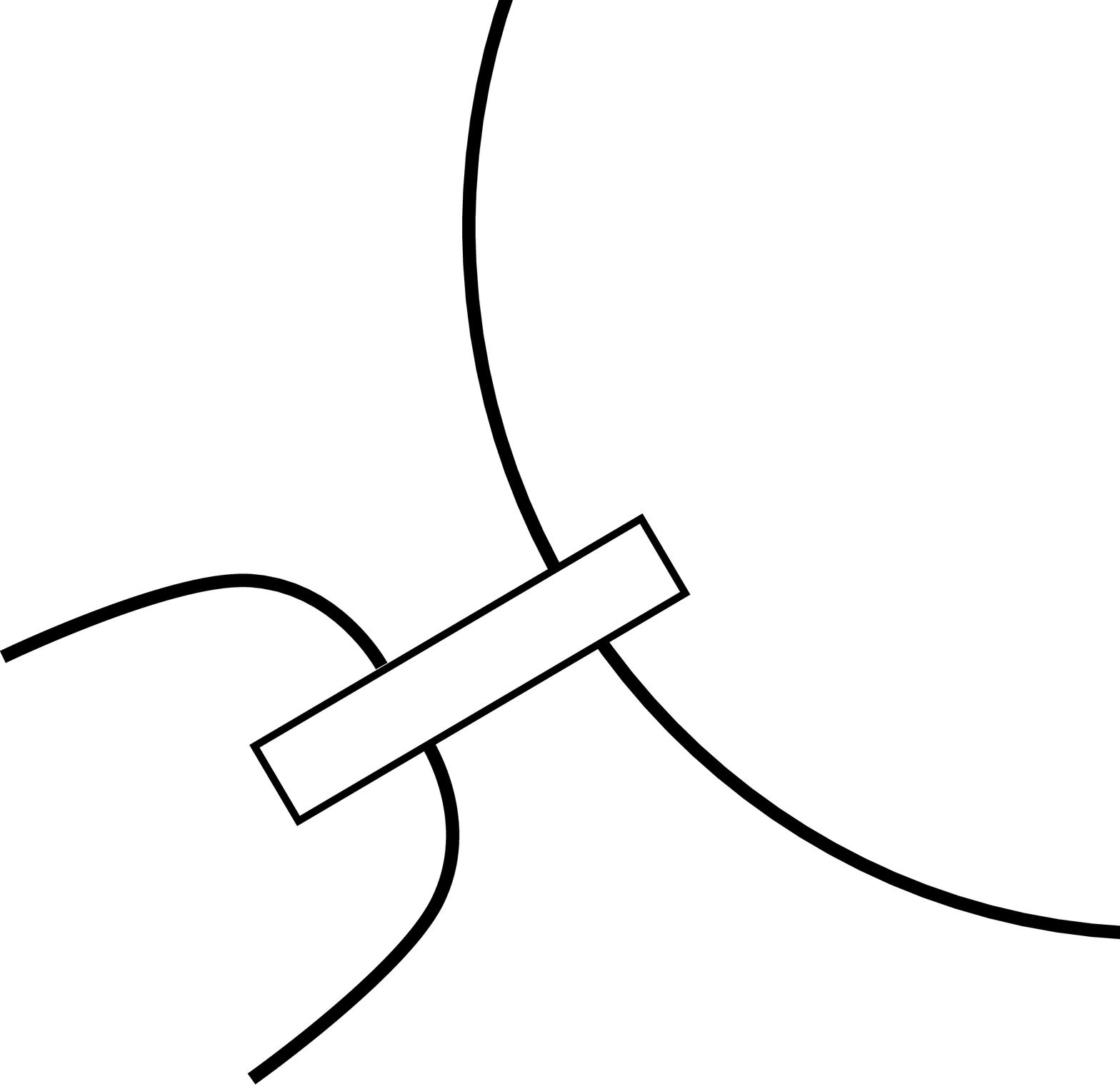}}
         \put(-48,61){$1$}
         \put(-27,77){$n$}
          \put(-15,71){$k$}
           \end{minipage}
  \end{eqnarray*}}
and finally one can show that removing the circle colored $1$ do not affect the first $n+1$ coefficients of $S_B^{(n+1)}(L)$.
\item Step one and two can be applied on every circle in  $S_B^{(n+1)}(L)$ and eventually we reduce $S_B^{(n+1)}(L)$ to $S_B^{(n)}(L)$.
\end{enumerate}

Now let $\mathcal{D}=\{D_n(q)\}_{n \in \mathbb{N}}$ be a sequence of skein elements in $\mathcal{S}(S^2)$. The previous discussion implies that the tail of the sequence $\mathcal{D}$ exists whenever the skein elements $D_n(q)$ are adequate. On the other hand, we know that the tail of the sequence $\mathcal{D}$ exists if and only if $D_{n+1}(q)\doteq_{n+1} D_{n}(q)$. This condition and the previous proof suggest that adequateness of every diagram in the family $\mathcal{D}$ may be a necessary condition for the tail of $\mathcal{D}$ to exist. This is not true however, and in the next section we give infinite family of sequences of inadequate skein elements whose tail exist. See Example \ref{inadequate}. 
\section{Computing the tail of a quantum spin network via local skein relations}
\label{section4}
Let $D$ be a banded trivalent
graph. Recall that an admissible coloring of $D$ is an assignment of colors to the edges of $D$ so that at each vertex, the three colors meeting there form an admissible triple. Consider a sequence of admissible quantum spin networks $\{D_n\}_{n \in \mathbb{N}}$ obtained from $D$ by labeling each edge by $n$ or $2n$. Recall that the evaluation of the quantum spin network $D_n$ in the skein module $\mathcal{S}(S^2)$ gives in general a rational function. Using definition \ref{main defintion} and remark \ref{main remark} one could study the tail of the sequence $\{D_n\}_{n \in \mathbb{N}}$. In this section we will study the tail of such skein elements. We start this section with a simple calculation for a certain coefficient of a crossingless matching diagram in the expansion of the Jones-Wenzl projector and we use this coefficient to derive our first local skein relation. We then use the bubble skein relation to compute more complicated local skein relations.
\begin{remark}
Since we will be working closely with identities such as the bubble expansion equation it will be easier to work with Jones-Wenzl projectors than to work with trivalent graphs. For this reason we will not state our results in terms of trivalent graphs notation.
\end{remark}
One can regard any skein element $\Gamma$ in the linear skein space $T_{a_1,a_2,..,a_m}$ as an element of the dual space $T^*_{a_1,a_2,..,a_m}$. This is done by embedding the space $T_{a_1,a_2,..,a_m}$ in $S^2$ and wiring the outside in someway to obtain a skein element in $\mathcal{S} (S^2)$. Let $\Gamma$ be an element of the skein space $T_{a_1,a_2,..,a_m}$ and let $x$ be a wiring in the disk in $S^2$ that is complementary to $T_{a_1,a_2,..,a_m}$ with the same specified boundary points. We denote by $\Gamma^*$ to the element in $T^*_{a_1,a_2,..,a_m}$ induced by the skein element $\Gamma$. We call the skein element $\Gamma^*(x) \in  \mathcal{S} (S^2) $ a \textit{closure} of $\Gamma$. In the following definition we assume that $\alpha_n$ and $\beta_n$ are admissible trivalent graphs with edges labeled $n$ or $2n$ in the skein space $T_{a_1,a_2,..,a_m}$, where $a_i {\in \{n,2n}\}$,  with  $\alpha^*_n$ and $\beta^*_n$ are the corresponding dual elements.
\begin{definition}
Let $\alpha_n$,$\beta_n$,$\alpha^*_n$ and $\beta^*_n$ be as above. Let $S$ be a subset of $T_{a_1,a_2,..,a_m}$. We say that
\begin{equation*}
\label{closure}
\alpha_n\doteq_n \beta_n 
\end{equation*}
on $S$ if 
\begin{equation*}
\alpha^*_n(x)\doteq_n\beta^*_n (x)
\end{equation*} for all $x$ in $S$.
\end{definition}

\begin{remark}
\label{important2}
The set $S$ mentioned in the definition can be chosen to be the set of all wiring $x$ such that the skein elements  $\alpha^*_n(x)$ and $\beta^*_n(x)$ are adequate. However, adequateness seems to be unnecessary in some cases and one could loosen this condition on the set $S$ further. We will give examples of such cases in this paper. Ideally, the set $S$ is supposed to be the set of all wiring $x$ such that the tail of the skein elements $\alpha^*_n(x)$ and $\beta^*_n(x)$ exist. It is not known to the author what is the largest set for which this condition holds.
\end{remark}

\begin{remark}
\label{important}
If we are working with $T_{a,b,c}$, where $(a,b,c)\in\{(2n,2n,2n),(n,n,2n)\}$, then for any skein element $\alpha_n$ in  $T_{a,b,c}$ we can write
\begin{equation}
\label{easy}
\alpha_n=P_n(q)\tau_{a,b,c}
\end{equation}
for some rational function $P_n(q)$. 
Hence if $x$ is an element in $T_{a,b,c}$ then the tail of the sequence $\{\alpha^*_n(x)\}_{n \in \mathbb{N}}$ exists if and only if the tails of the sequences $\{\tau^*_{a,b,c}(x)\}_{n \in \mathbb{N}}$ and $\{P_n(q)\}_{n \in \mathbb{N}}$ exist. In particular, \ref{easy} also implies that if the tail of $\{P_n(q)\}_{n \in \mathbb{N}}$ exists and $x$ is a wiring in $T_{a,b,c}$ such that $\tau^*_{a,b,c}(x)$ is an adequate skein element, then the tail of the sequence $\{\alpha^*_n(x)\}_{n \in \mathbb{N}}$ exists. 
Note that for every such $x$ one has $P_n(q)\doteq_n\alpha^*_n(x)/\tau^*_{a,b,c}(x)$.
\end{remark}

Following Morrison \cite{Morrison}, write $\underset{\in f^{(n)}}{ \text{coeff}}(D)$ to denote the coefficient of the crossingless matching diagram $D$ appearing in the $n^{th}$ Jones-Wenzl projector. We will use Morrison's recursive formula to calculate certain coefficients of the Jones-Wenzl idempotent. The recursive formula is explained very well in \cite{Morrison}, see Proposition 4.1 and the examples within,  and we shall not repeat it here.
\begin{lemma}
{\small
\begin{eqnarray}
\label{morrison}
\underset{\in f^{(2n)}}{ \text{coeff}}\left(
\hspace{1pt}
  \begin{minipage}[h]{0.105\linewidth}
        \vspace{5pt}
        \scalebox{0.2}{\includegraphics{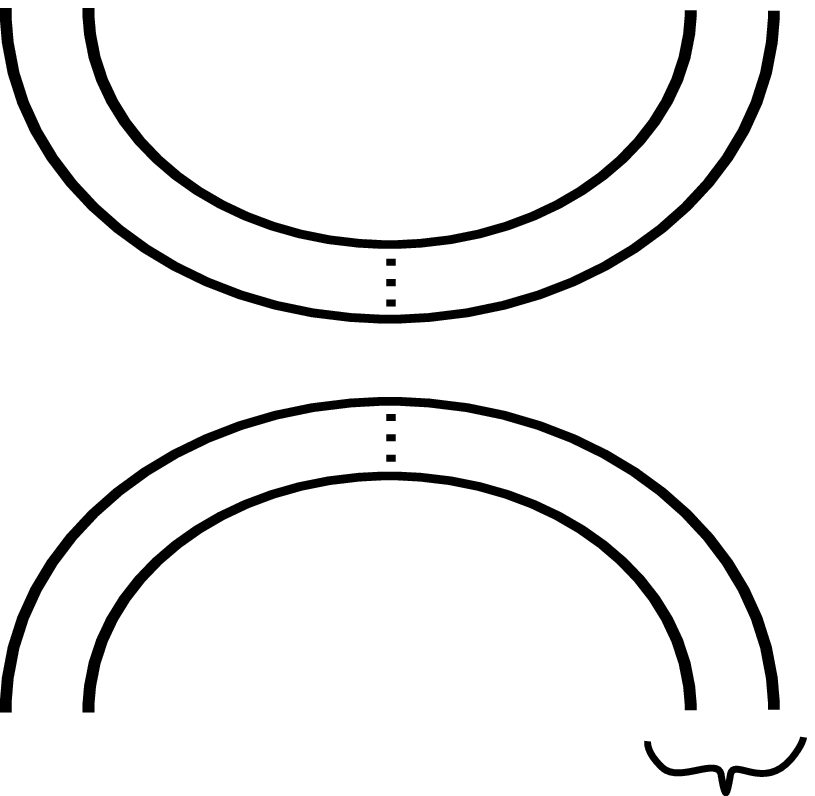}}
        \put(-8,-6){$n$}
       \end{minipage}\right)
   =\frac{([n]!)^2}{[2n]!}
   \end{eqnarray}
}
\end{lemma}

\begin{proof}
Applying Morrison's induction formula, Proposition 4.1 in \cite{Morrison}, on the left hand side of (\ref{morrison}), we obtain
{\small
\begin{eqnarray*}
\label{special0}
\underset{\in f^{(2n)}}{ \text{coeff}}\left(
\hspace{1pt}
  \begin{minipage}[h]{0.105\linewidth}
        \vspace{5pt}
        \scalebox{0.2}{\includegraphics{hooks}}
        \put(-8,-6){$n$}
       \end{minipage}\right)
   &=&\frac{[n]}{[2n]}\underset{\in f^{(2n-1)}}{ \text{coeff}}\left(
\hspace{1pt}
  \begin{minipage}[h]{0.162\linewidth}
        \vspace{5pt}
        \scalebox{0.2}{\includegraphics{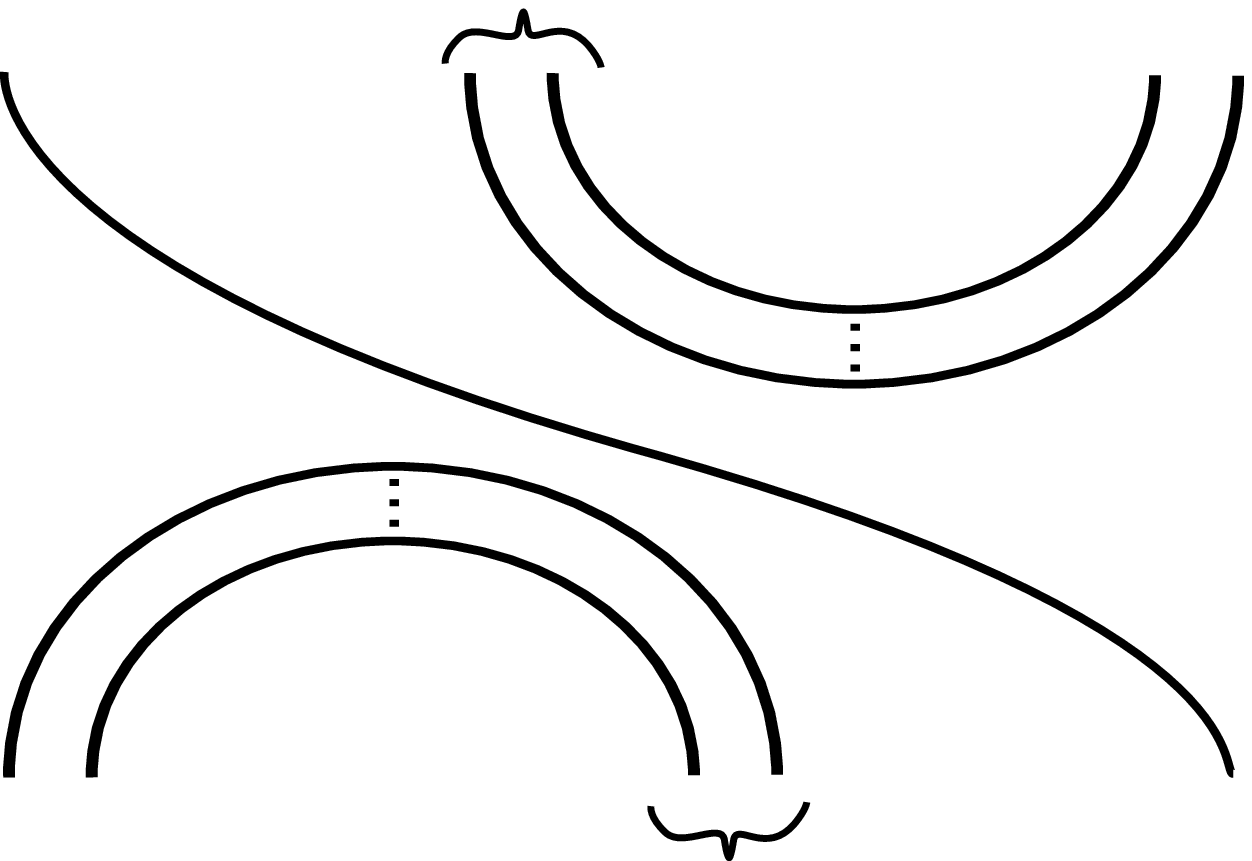}}
        \put(-40,-6){$n-1$}
        \put(-54,50){$n-1$}
        \put(-69,43){$1$}
       \end{minipage}\right)\\
         &=&\frac{[n][n-1]}{[2n][2n-1]}\underset{\in f^{(2n-2)}}{ \text{coeff}}\left(
\hspace{1pt}
  \begin{minipage}[h]{0.162\linewidth}
        \vspace{5pt}
        \scalebox{0.2}{\includegraphics{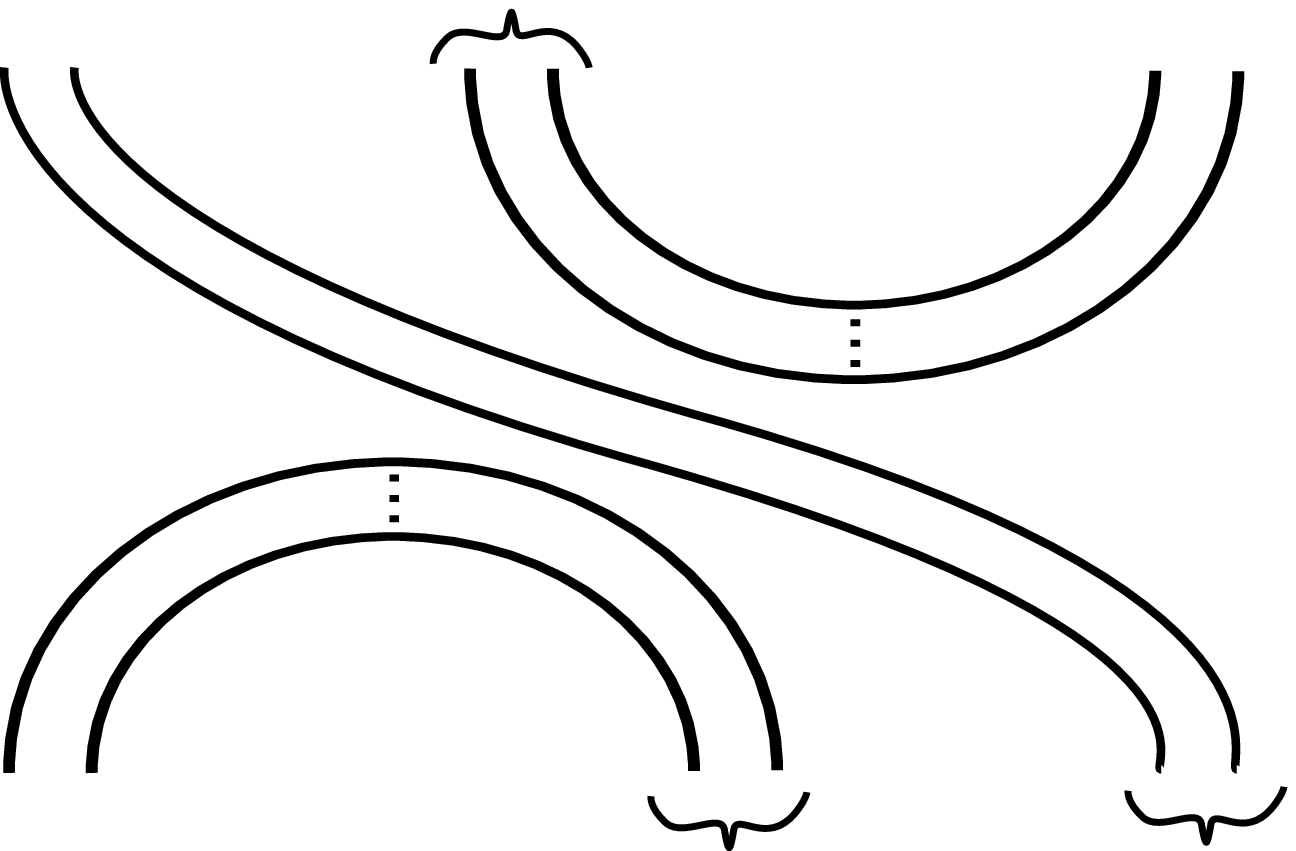}}
        \put(-42,-9){$n-2$}
        \put(-54,50){$n-2$}
         \put(-7,-9){$2$}
       \end{minipage}\right)\\
       &=&\frac{[n]!}{\prod\limits_{i=n+1}^{2n}[i]}=\frac{([n]!)^{2}}{[2n]!}
   \end{eqnarray*}
}
\end{proof}

\begin{proposition}
\label{thm1}
 For all adequate closures of the element $\tau_{n,n,2n}$ and for all $n\geq 0 $:
\begin{eqnarray}
   \begin{minipage}[h]{0.07\linewidth}
         \vspace{0pt}
         \scalebox{0.30}{\includegraphics{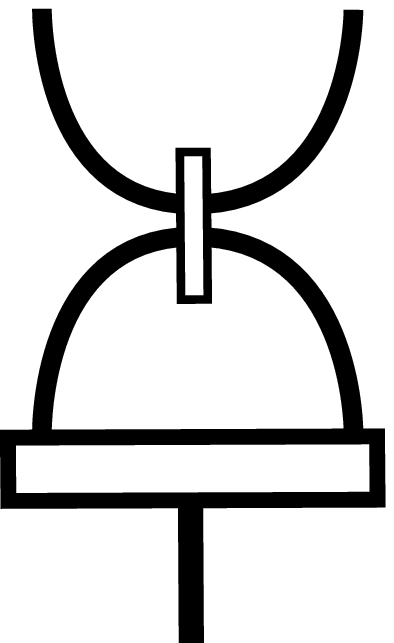}}
        \put(-1,60){$n$}
         \put(-40,60){$n$}
        \put(-22,-12){$2n$}
   \end{minipage}&\doteq_n&(q;q)_n
  \begin{minipage}[h]{0.16\linewidth}
        \vspace{0pt}
        \scalebox{0.30}{\includegraphics{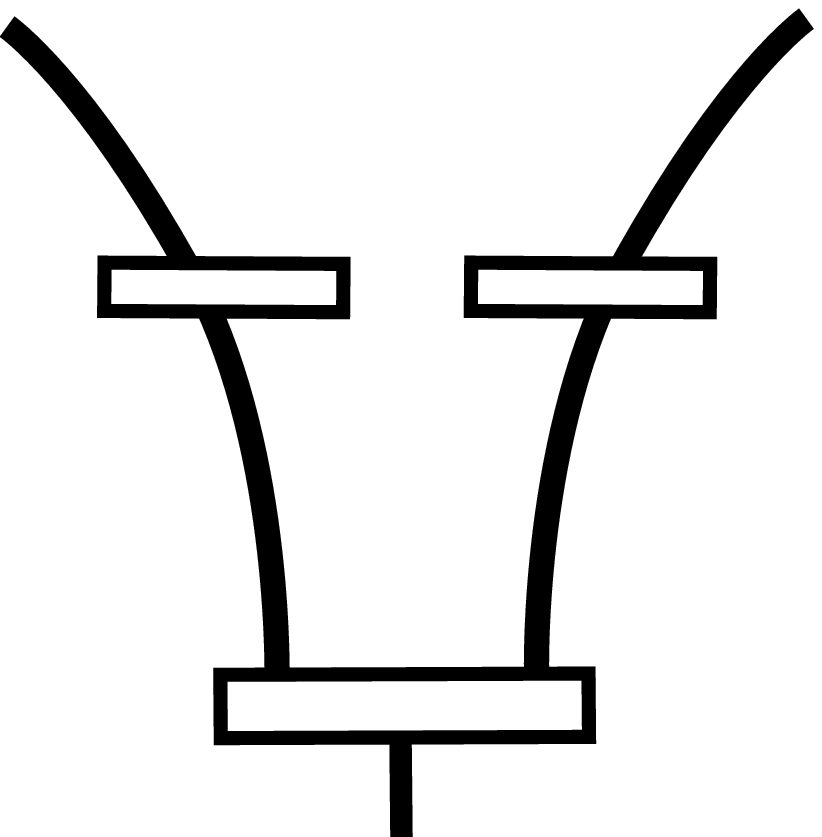}}
         \put(-45,-10){$2n$}
          \put(-1,60){$n$}
         \put(-77,60){$n$}
           \end{minipage}
           \label{di}
  \end{eqnarray}
\end{proposition}
\begin{proof}
Write $\Gamma_n$ to denote the skein element that appears on the left hand side of $\ref{di}$.
\begin{figure}[H]
  \centering
   {\includegraphics[scale=0.32]{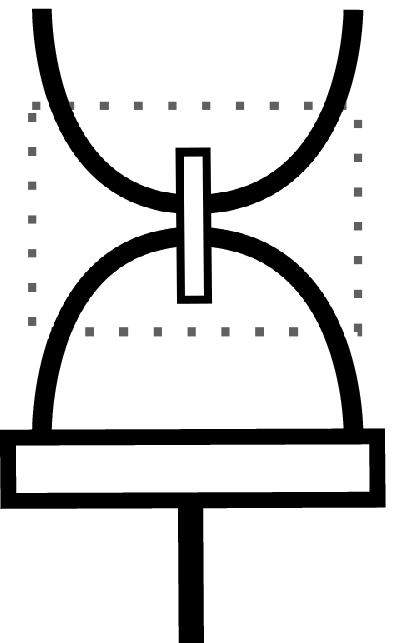}
    \put(-1,60){$n$}
         \put(-40,60){$n$}
         \put(-22,-12){$2n$}
     \caption{Expanding  $f^{(2n)}$}
  \label{DR}}
\end{figure}
Consider the idempotent $f^{(2n)}$ that appears in $\Gamma_n$ inside the square in Figure \ref{DR} and expand this element as a $\mathbb{Q}(A)$-linear summation of crossingless matching diagrams. Every crossingless matching diagram in this expansion, except for the diagram that appears in Figure \ref{hon}, is going to produce a hook to the bottom idempotent $f^{(2n)}$ in $\Gamma_n$ and hence the term with such crossingless matching diagram evaluates to zero.  
\begin{figure}[H]
  \centering
   {\includegraphics[scale=0.32]{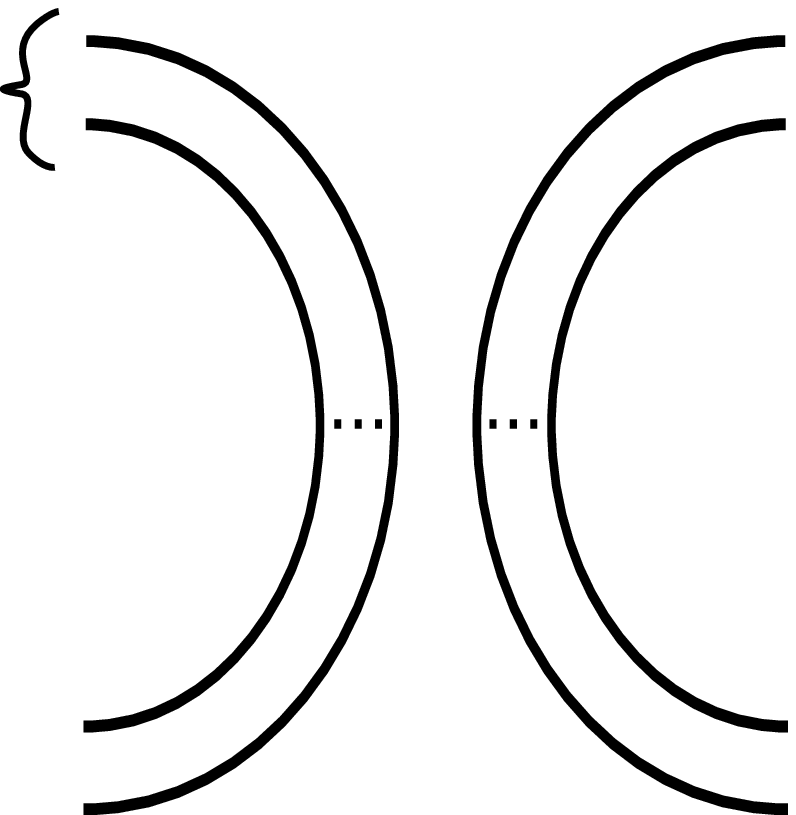}
    \put(-82,65){$n$}
     \caption{}
  \label{hon}}
\end{figure}
This allows us to write
\begin{eqnarray*}
   \begin{minipage}[h]{0.07\linewidth}
         \vspace{0pt}
         \scalebox{0.30}{\includegraphics{f_element}}
        \put(-1,60){$n$}
         \put(-40,60){$n$}
         \put(-22,-12){$2n$}
   \end{minipage}&=&\underset{\in f^{(2n)}}{ \text{coeff}}\left(
\hspace{1pt}
  \begin{minipage}[h]{0.105\linewidth}
        \vspace{5pt}
        \scalebox{0.2}{\includegraphics{hooks}}
        \put(-8,-6){$n$}
       \end{minipage}\right)
  \begin{minipage}[h]{0.16\linewidth}
        \vspace{0pt}
        \scalebox{0.30}{\includegraphics{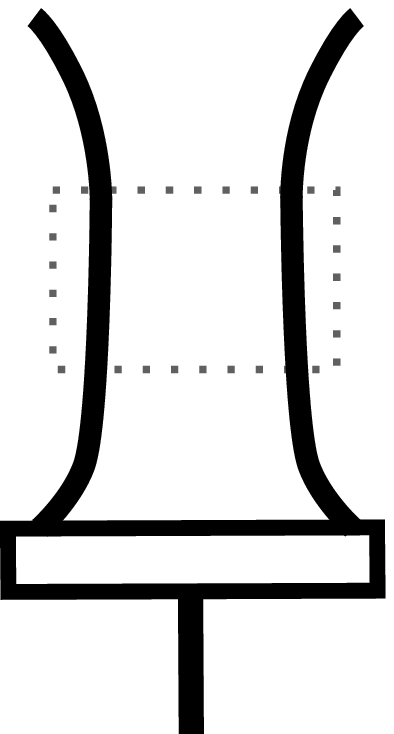}}
         \put(-25,-6){$2n$}
          \put(-1,60){$n$}
         \put(-39,60){$n$}
           \end{minipage}\\&=&\frac{([n]!)^{2}}{[2n]!} \begin{minipage}[h]{0.16\linewidth}
        \vspace{0pt}
        \scalebox{0.30}{\includegraphics{two_bubbles_4}}
         \put(-45,-10){$2n$}
          \put(-1,60){$n$}
         \put(-77,60){$n$}
           \end{minipage}
  \end{eqnarray*}  
 Using the fact $[n]!=q^{(n-n^2)/4} ( 1-q)^{-n} (q, q)_n$ we can write
  \begin{eqnarray*}
  \frac{([n]!)^{2}}{[2n]!}&=&q^{n^2/2} \frac{(q;q)^2_n}{(q;q)_{2n}}\\&=&q^{n^2/2}\frac{\left(\displaystyle\prod_{i=0}^{n-1}(1-q^{i+1})\right)^2}{\displaystyle\prod_{i=0}^{2n-1}(1-q^{i+1})}\\&=&q^{n^2/2}\frac{\displaystyle\prod_{i=0}^{n-1}(1-q^{i+1})}{\displaystyle\prod_{i=n}^{2n-1}(1-q^{i+1})}\\&=&q^{n^2/2}\displaystyle\prod_{i=0}^{n-1}\frac{(1-q^{i+1})}{(1-q^{i+n+1})}\\&\doteq_n&\displaystyle\prod_{i=0}^{n-1}(1-q^{i+1})=(q;q)_n
  \end{eqnarray*}

\end{proof}
\begin{proposition} For all adequate closures of the element $\tau_{n,n,2n}$ and for all $n\geq 0 $: 
\label{thm2}
\begin{eqnarray}
\label{eee}
   \begin{minipage}[h]{0.20\linewidth}
         \vspace{0pt}
         \scalebox{0.30}{\includegraphics{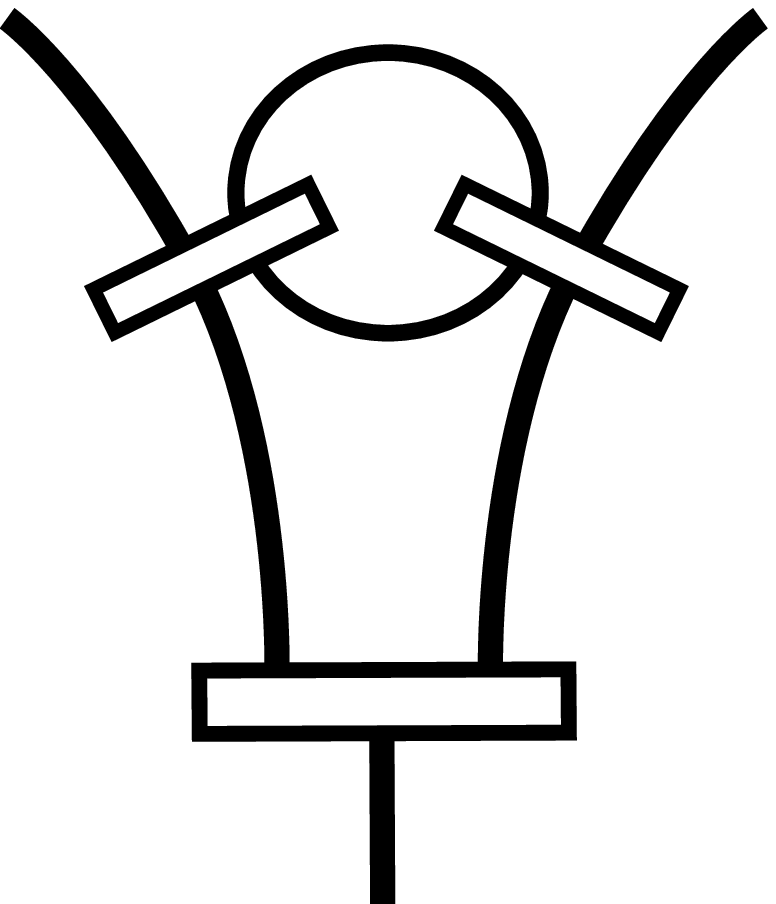}}
      \put(-45,-10){$2n$}
          \put(-1,60){$n$}
         \put(-77,60){$n$}
              \put(-35,80){$n$}
         \put(-35,40){$n$}
   \end{minipage}&\doteq_n&(q;q)_n
   \begin{minipage}[h]{0.16\linewidth}
        \vspace{0pt}
        \scalebox{0.30}{\includegraphics{two_bubbles_4}}
         \put(-45,-10){$2n$}
          \put(-1,60){$n$}
         \put(-77,60){$n$}
           \end{minipage}
  \end{eqnarray}
\end{proposition}
\begin{proof}
\begin{eqnarray*}
   \begin{minipage}[h]{0.20\linewidth}
         \vspace{-10pt}
         \scalebox{0.30}{\includegraphics{one_bubble}}
      \put(-45,-10){$2n$}
          \put(-1,60){$n$}
         \put(-77,60){$n$}
            \put(-35,80){$n$}
         \put(-35,40){$n$}
   \end{minipage}&=&\sum\limits_{i=0}^{n}\left\lceil 
\begin{array}{cc}
n & n \\ 
n & n%
\end{array}%
\right\rceil _{i}
   \begin{minipage}[h]{0.16\linewidth}
        \vspace{0pt}
        \scalebox{0.30}{\includegraphics{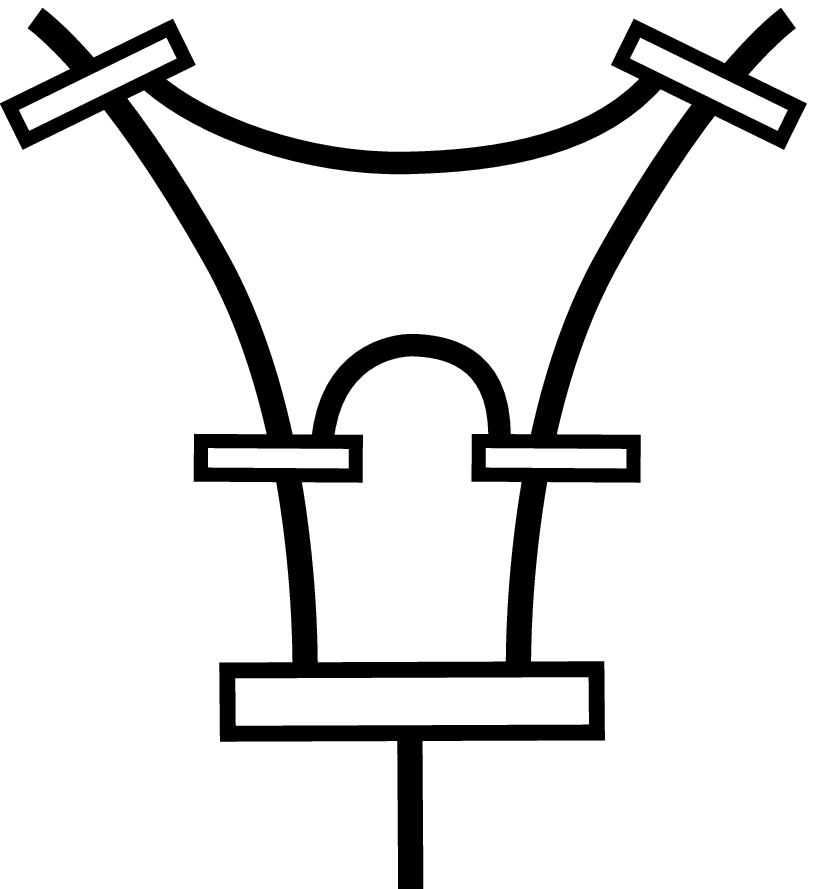}}
         \put(-45,-10){$2n$}
          \put(1,66){$n$}
         \put(-77,66){$n$}
            \put(-35,70){$i$}
         \put(-35,50){$i$}
           \end{minipage}\\&=&\left\lceil 
\begin{array}{cc}
n & n \\ 
n & n%
\end{array}%
\right\rceil _{0}
    \begin{minipage}[h]{0.16\linewidth}
        \vspace{0pt}
        \scalebox{0.30}{\includegraphics{two_bubbles_4}}
         \put(-45,-10){$2n$}
          \put(-1,60){$n$}
         \put(-77,60){$n$}
                   \end{minipage}
  \end{eqnarray*}
The first equation follows by applying the bubble expansion formula, Theorem \ref{main}, and the second last equation follows from the annihilation axiom of the Jones-Wenzl idempotent.
  
Using (\ref{qb}) and the fact that
\begin{equation}
\prod\limits_{i=0}^{j}[n-i]=q^{(2 + 3 j + j^2 - 2 n - 2 j n)/4} (1 - q)^{-1 - j}\frac{(q;q)_n}{(q;q)_{n-j-1}}
\end{equation}
we can write
\begin{equation}
\label{eq}
\left\lceil 
\begin{array}{cc}
n & n \\ 
n & n%
\end{array}%
\right\rceil _{0}=(-1)^nq^{-
 n/2} 
\frac{(q; q)^3_n(q;q)_{3n+1}}{(q;q)_{2 n}^2 (q ;q)_{2 n + 1}}.
\end{equation}

However,
\begin{equation*}
\frac{(q;q)_{3n+1}}{(q;q)_{2n+1}}
=1-q^{2n+2}+O(2n+3)=_{n}1,
\end{equation*}
and 
\begin{equation*}
\frac{(q;q)_{n}}{(q;q)_{2n}}
=_{n}1,
\end{equation*}
hence (\ref{eq}) yields:
\begin{equation*}
\left\lceil 
\begin{array}{cc}
n & n \\ 
n & n%
\end{array}%
\right\rceil _{0}\doteq_n(q;q)_n,
\end{equation*}
and the result follows.
\end{proof}
\begin{remark}
In \cite{Hajij} we showed that
\begin{equation*}
\left\lceil 
\begin{array}{cc}
n & n \\ 
n & n%
\end{array}%
\right\rceil _{0}\Delta_{2n}=\Theta(2n,2n,2n).
\end{equation*}
Hence the previous theorem implies
\begin{equation*}
\Theta(2n,2n,2n)\doteq_n \frac{(q,q)_n}{1-q}=(q^2,q)_n.
\end{equation*}
\end{remark}
Propositions \ref{thm1} and \ref{thm1} imply immediately the following result.
\begin{proposition}
 For all adequate closures of the left hand side skein element of the following equation and for all $n\geq 0 $ the following holds: 
\begin{eqnarray*}
   \begin{minipage}[h]{0.15\linewidth}
         \vspace{0pt}
         \scalebox{0.30}{\includegraphics{one_bubble}}
      \put(-45,-10){$2n$}
          \put(-1,60){$n$}
         \put(-77,60){$n$}
   \end{minipage}&\doteq_n&
   \begin{minipage}[h]{0.07\linewidth}
         \vspace{10pt}
         \scalebox{0.30}{\includegraphics{f_element}}
        \put(-1,60){$n$}
         \put(-40,60){$n$}
        \put(-22,-12){$2n$}
   \end{minipage}
  \end{eqnarray*}
\end{proposition}
\begin{lemma}
For $n\geq 1$:\\
\label{lma}
\begin{enumerate}
\item$
\sum\limits_{i=0}^{n}\left\lceil 
\begin{array}{cc}
n & n \\ 
n & n%
\end{array}%
\right\rceil _{i}\frac{\Delta _{2n}}{\Delta _{n+i}}\doteq_{n}\Psi(q^3,q).
$
\item$
\sum\limits_{i=0}^{n}\left\lceil 
\begin{array}{cc}
n & n \\ 
n & n%
\end{array}%
\right\rceil _{i}\left\lceil 
\begin{array}{cc}
n & i \\ 
n & n%
\end{array}%
\right\rceil _{0}\frac{\Delta _{2n}}{\Delta _{n+i}}\doteq_{n}f(-q^4,-q).
$
\end{enumerate}
\end{lemma}

\begin{proof}
\begin{enumerate}
\item
Set $P(n,i):=
\left\lceil 
\begin{array}{cc}
n & n \\ 
n & n%
\end{array}\right\rceil _{i}\frac{\Delta _{2n}}{\Delta _{n+i}}$. From the bubble expansion formula, Theorem \ref{main}, we obtain,
\begin{equation}
\label{PN}
P(n,i)
=q^{i(i-n)/2}\frac{[2n+1]}{[n+i+1]}\frac{\prod\limits_{j=0}^{n-i-1}[n-j]%
\prod\limits_{s=0}^{i-1}[n-s]^{2}}{\prod\limits_{h=0}^{n-1}[2n-h]^{2}}%
\binom{n}{i}_{q}\prod\limits_{k=0}^{n-i-1}[3n-i-k+1] 
\end{equation}

Using (\ref{qb}) and the fact that
\begin{equation}
\label{fact}
\prod\limits_{i=0}^{j}[n-i]=q^{(2 + 3 j + j^2 - 2 n - 2 j n)/4} (1 - q)^{-1 - j}\frac{(q;q)_n}{(q;q)_{n-j-1}}
\end{equation}

we can rewrite (\ref{PN}) to obtain the following:
\begin{equation}
\label{mn}
P(n,i)=q^{
 (2 i + 4 i^2 - 2 n)/4} 
 \frac{[2n+1]}{[n+i+1]}\frac{(q; q)^6_n(q;q)_{3n-i+1}}{(q;q)_{2 n}^2 (q ;q)_{2 n + 1} (q;q)_i^2 (q;q)_{n - i}^3}
\end{equation}
Now we shall study the first $n$ terms of $%
P(n,i+1)+P(n,i).$ We claim that 
\begin{equation*}
P(n,i)+P(n,i+1)\doteq_n P(n,i)+Q(n,i+1)
\end{equation*}
where $\frac{(q;q)_{n}}{(q;q)_{n-i}}Q(n,i)=P(n,i)$.
To prove this claim observe first that $ m( P{(n,i))}=(i+i^{2}-n).$ Note also that for all $1\leq i\leq n$:
\begin{equation*}
\frac{(q;q)_{n}}{(q;q)_{n-i}}=1-q^{(n-i+1)}+O(n-i+2)
\end{equation*}
 This implies that the minimal degree of $
Q(n,i)$ is equal to the minimal degree of $P(n,i)$. Thus
{\footnotesize
\begin{eqnarray*}
P(n,i)+P(n,i+1) &=&P(n,i)+\frac{(q,q)_{n}}{(q,q)_{n-i-1}}%
Q(n,i+1) \\
&=&P(n,i)+(1- q^{(n-i)}+O(n-i+1))Q(n,i+1) \\
&=&P(n,i)+Q(n,i+1)- q^{(n-i)}Q(n,i+1)+O(3 + 2 i + i^2)\\
&=&P(n,i)+Q(n,i+1)-
q^{n-i}(q^{i+i^{2}-n}q^{2+2i}+O(3 + 3 i + i^2 - n))+O(3 + 2 i + i^2)
\\
&=&P(n,i)+Q(n,i+1)-
q^{n+i+2}q^{i+i^{2}-n}+O(3 + 2 i + i^2)\\
&\doteq_{n}&P(n,i)+Q(n,i+1)
\end{eqnarray*}
}
The last equation is true since $%
m(P(n,i)+P(n,i+1))=i+i^{2}-n$ and $n+i+2>n+1$ for all positive integers $i$, and hence the terms $-q^{n+i+2}q^{i+i^{2}-n}+O(3+2i+ i^2)$ do not contribute the first $n$ terms of $P(n,i)+P(n,i+1)$. This proves our claim and hence we can write:
\begin{eqnarray*}
P(n,0)+P(n,2)+...+P(n,n) &=& \\
&=&P(n,0)+...+P(n,n-1)+Q(n,n) \\
&=&P(n,0)+Q(n,2)...+Q(n,n-1)+Q(n,n) \\
&=&Q(n,0)+Q(n,2)...+Q(n,n-1)+Q(n,n)
\end{eqnarray*}
The last equality follows from the fact that $P(n,0)=Q(n,0).$ Using this result and (\ref{mn}) we obtain
\begin{equation}
\label{fn}
\sum\limits_{i=0}^{n}P(n,i)\doteq_n \sum\limits_{i=0}^{n}q^{
 (i/2 + i^2)} 
 \frac{[2n+1]}{[n+i+1]}\frac{(q; q)^3_n(q;q)_{3n-i+1}}{(q;q)_{2 n}^2 (q ;q)_{2 n + 1} (q;q)_i^2}
\end{equation}
Now
\begin{equation}
\label{1}
\frac{(q;q)_{3n-i+1}}{(q;q)_{2n+1}}
=1-q^{2n+2}+O(2n+3)=_{n}1
\end{equation}
and similarly we can show that
\begin{equation}
\label{2}
\frac{(q,q)_{n}}{(q,q)_{2n}}=_{n}1 
\end{equation}
Putting (\ref{1}) and (\ref{2}) all in (\ref{fn}) we obtain:
\begin{eqnarray*}
\sum\limits_{i=0}^{n}P(n,i)&\doteq_{n}&(q;q)_{n}\sum\limits_{i=0}^{n}\frac{%
q^{i/2+i^{2}}}{(q;q)_{i}^{2}} \frac{[2n+1]}{[n+i+1]}\\&\doteq_{n}&(q;q)_{n}\sum\limits_{i=0}^{n}\frac{%
q^{i+i^{2}}}{(q;q)_{i}^{2}}\\&\doteq_{n}&\Psi(q^3,q).
\end{eqnarray*}
\item
Using \ref{fact} one can write 
\begin{equation}
\label{nnii0}
\left\lceil 
\begin{array}{cc}
n & i \\ 
n & n%
\end{array}%
\right\rceil _{0}=(-1)^nq^{-
 n/2} 
\frac{(q; q)_i(q; q)^2_n(q;q)_{2n+i+1}}{(q;q)_{n+i+1} (q ;q)_{ n +i}(q ;q)_{2 n }}.
\end{equation}
Equations \ref{PN} and \ref{nnii0} imply:
{\scriptsize
\begin{equation}
\left\lceil 
\begin{array}{cc}
n & n \\ 
n & n%
\end{array}%
\right\rceil _{i}\left\lceil 
\begin{array}{cc}
n & i \\ 
n & n%
\end{array}%
\right\rceil _{0}\frac{\Delta _{2n}}{\Delta _{n+i}}=(-1)^nq^{i/2 + i^2 - n} 
\frac{(q; q)^8_n(q;q)_{2n+i+1}(q;q)_{3n-i+1}}{(q;q)_{n+i+1} (q ;q)_{ n +i}(q ;q)^3_{ n -i}(q ;q)_{i }(q ;q)^3_{2 n }(q ;q)_{2n+1}}\frac{[2n+1]}{[n+i+1]}
\end{equation}
}
 Using similar calculations to the ones we did in (1), one can write

\begin{eqnarray*}
\sum\limits_{i=0}^{n}\left\lceil 
\begin{array}{cc}
n & n \\ 
n & n%
\end{array}%
\right\rceil _{i}\left\lceil 
\begin{array}{cc}
n & i \\ 
n & n%
\end{array}%
\right\rceil _{0}\frac{\Delta _{2n}}{\Delta _{n+i}}&\doteq_n&(q; q)_n \sum\limits_{i=0}^{n}
\frac{q^{i^2+i}}{(q;q)_{i}}\doteq_nf(-q^4,-q).
\end{eqnarray*}

\end{enumerate}
\end{proof}

\begin{proposition}
For adequate closures of the element $\tau_{n,n,2n}$ and for all $n\geq0$:
\begin{enumerate}

\item \begin{eqnarray*}
   \begin{minipage}[h]{0.20\linewidth}
         \vspace{-10pt}
         \scalebox{0.30}{\includegraphics{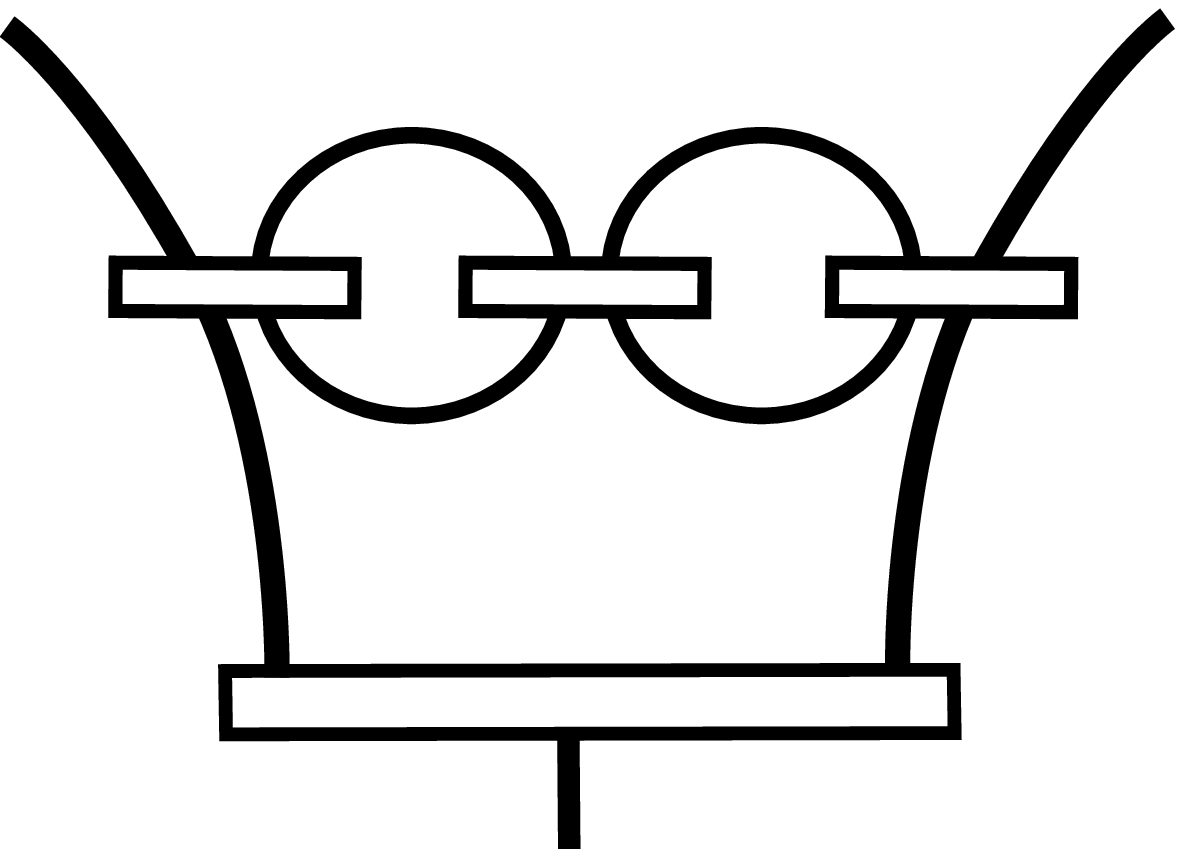}}
        \put(-68,+67){$n$}
         \put(-68,+27){$n$}
         \put(-1,60){$n$}
         \put(-108,60){$n$}
          \put(-40,+27){$n$}
         \put(-40,67){$n$}
         \put(-60,-10){$2n$}
   \end{minipage}&\doteq_n&\Psi(q^3,q)
   \begin{minipage}[h]{0.16\linewidth}
        \vspace{0pt}
        \scalebox{0.30}{\includegraphics{two_bubbles_4}}
         \put(-45,-10){$2n$}
          \put(-1,60){$n$}
         \put(-77,60){$n$}
           \end{minipage}
  \end{eqnarray*}
\item
\begin{eqnarray*}
    \begin{minipage}[h]{0.26\linewidth}
         \vspace{-10pt}
         \scalebox{0.28}{\includegraphics{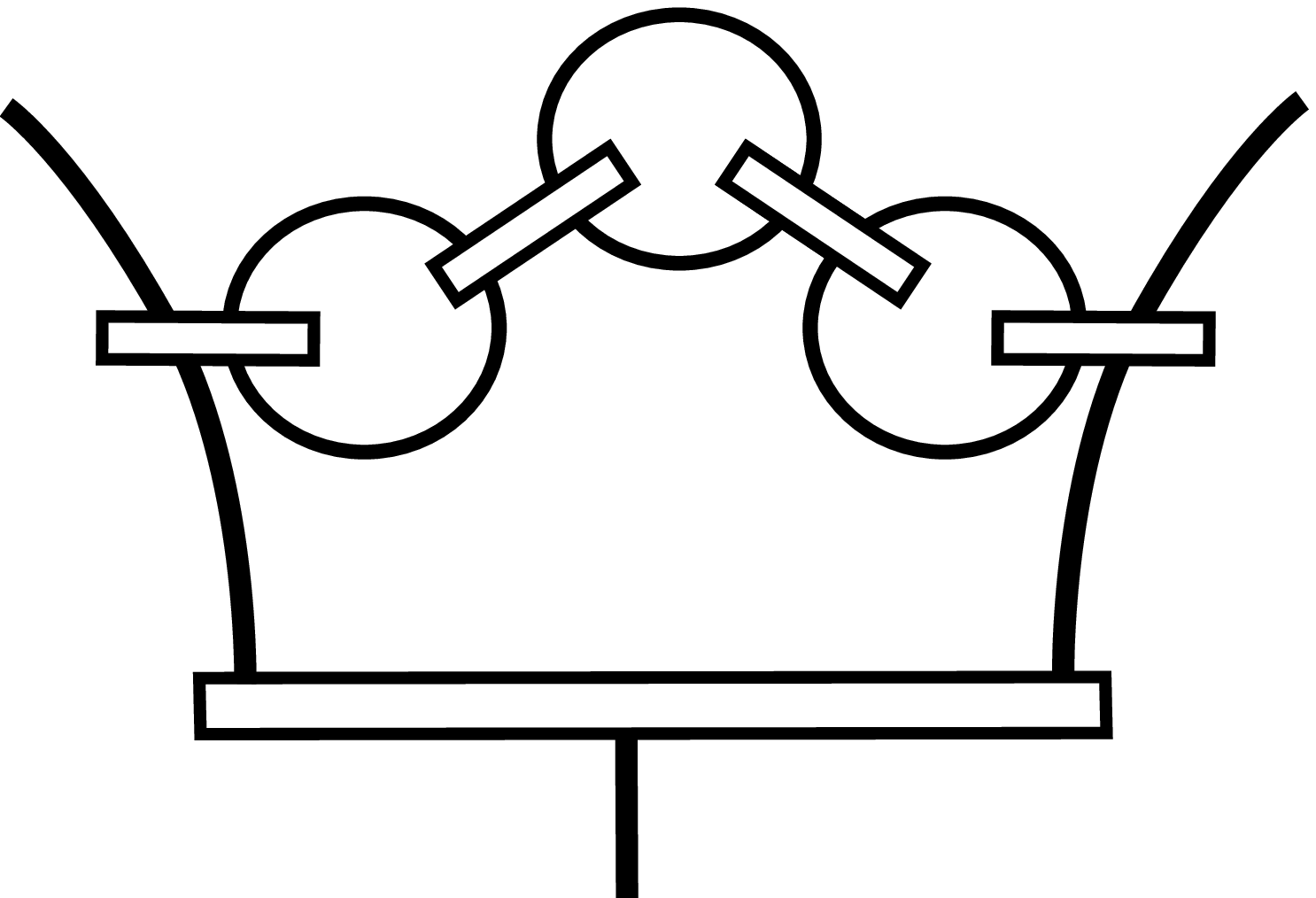}}
        \put(-88,+67){$n$}
         \put(-90,+33){$n$}
         \put(-1,60){$n$}
         \put(-128,60){$n$}
          \put(-40,+33){$n$}
         \put(-35,67){$n$}
         \put(-60,85){$n$}
         \put(-60,50){$n$}
         \put(-60,-10){$2n$}
   \end{minipage}&=&f(-q^4,-q)
   \begin{minipage}[h]{0.16\linewidth}
        \vspace{0pt}
        \scalebox{0.30}{\includegraphics{two_bubbles_4}}
         \put(-45,-10){$2n$}
          \put(-1,60){$n$}
         \put(-77,60){$n$}
           \end{minipage}
  \end{eqnarray*}
\end{enumerate}
\end{proposition}

\begin{proof}
\begin{enumerate}
\item
Applying the bubble expansion formula on the left bubble, we obtain,
\begin{eqnarray*}
   \begin{minipage}[h]{0.20\linewidth}
         \vspace{0pt}
         \scalebox{0.30}{\includegraphics{two_bubbles}}
        \put(-68,+67){$n$}
         \put(-68,+27){$n$}
         \put(-1,60){$n$}
         \put(-108,60){$n$}
          \put(-40,+27){$n$}
         \put(-40,67){$n$}
         \put(-60,-10){$2n$}
   \end{minipage}&=&\sum\limits_{i=0}^{n}\left\lceil 
\begin{array}{cc}
n & n \\ 
n & n%
\end{array}%
\right\rceil _{i}
   \begin{minipage}[h]{0.16\linewidth}
        \vspace{-10pt}
        \scalebox{0.30}{\includegraphics{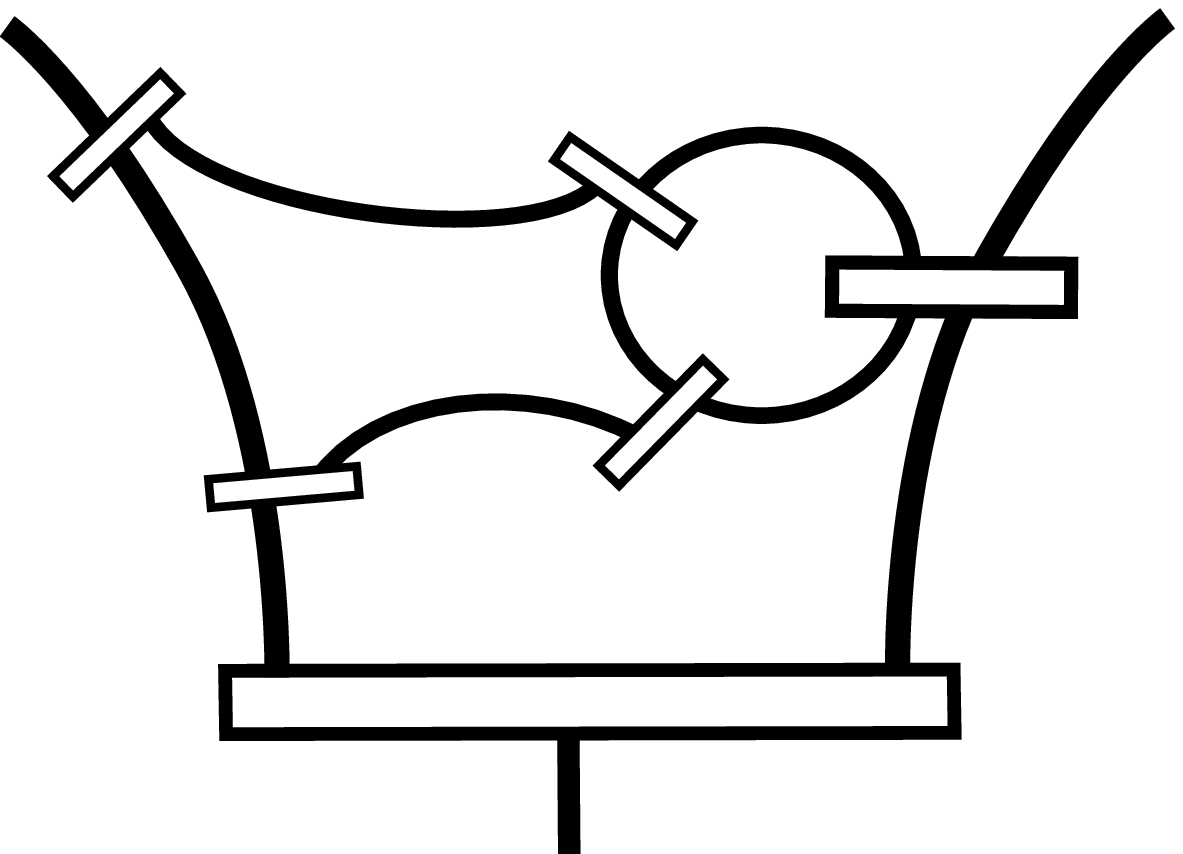}}
        \put(-68,+60){$i$}
         \put(-60,+30){$i$}
          \put(-60,-10){$2n$}
          \put(-1,60){$n$}
         \put(-108,60){$n$}
         \put(-40,67){$n$}
           \end{minipage}\\&=&\sum\limits_{i=0}^{n}\left\lceil 
\begin{array}{cc}
n & n \\ 
n & n%
\end{array}%
\right\rceil _{i}
   \begin{minipage}[h]{0.16\linewidth}
        \vspace{0pt}
        \scalebox{0.30}{\includegraphics{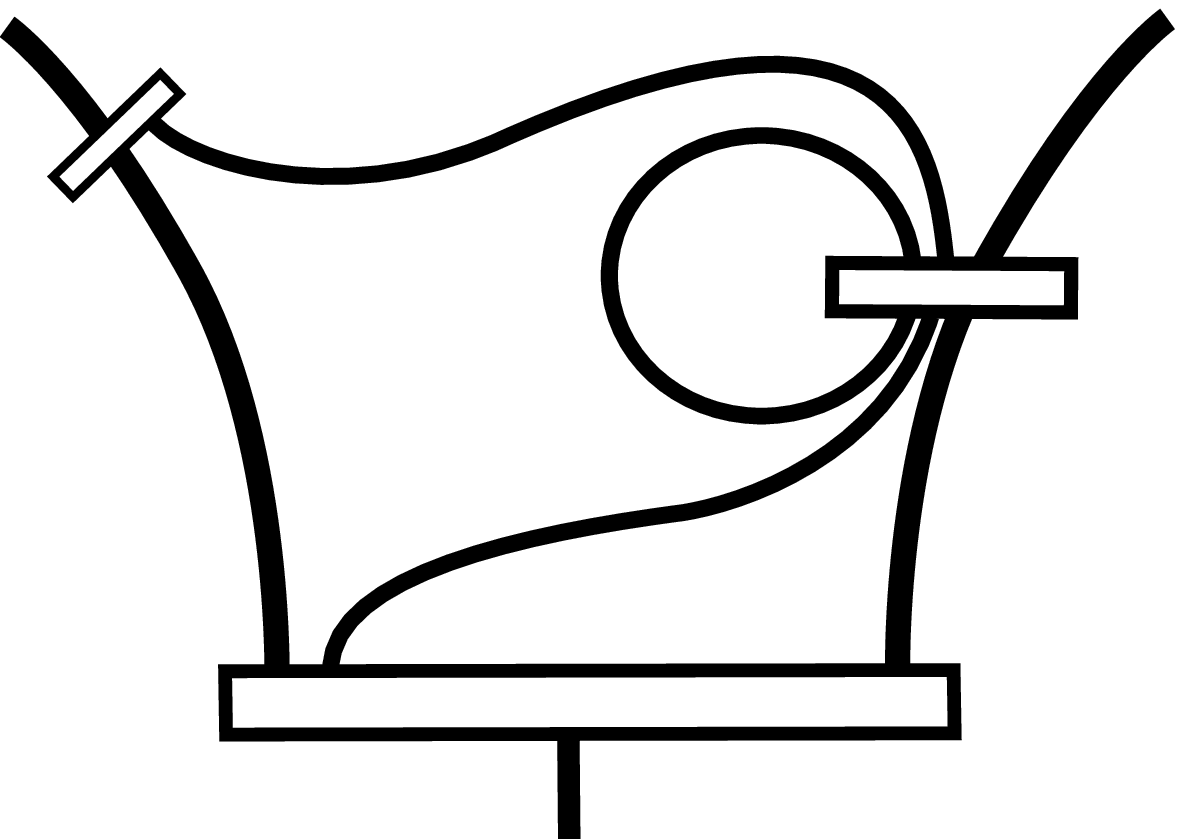}}
        \put(-68,+61){$i$}
         \put(-60,+30){$i$}
          \put(-60,-10){$2n$}
          \put(-1,60){$n$}
         \put(-108,60){$n$}
        \put(-77,+45){$n-i$}
           \end{minipage}
  \end{eqnarray*}
Using the property (\ref{properties2}) of the idempotent, we obtain
\begin{eqnarray*}
   \begin{minipage}[h]{0.20\linewidth}
         \vspace{0pt}
         \scalebox{0.30}{\includegraphics{two_bubbles}}
        \put(-68,+67){$n$}
         \put(-68,+27){$n$}
         \put(-1,60){$n$}
         \put(-108,60){$n$}
          \put(-40,+27){$n$}
         \put(-40,67){$n$}
         \put(-60,-10){$2n$}
   \end{minipage}=\sum\limits_{i=0}^{n}\left\lceil 
\begin{array}{cc}
n & n \\ 
n & n%
\end{array}%
\right\rceil _{i}\frac{\Delta _{2n}}{\Delta _{n+i}}
  \begin{minipage}[h]{0.16\linewidth}
        \vspace{5pt}
        \scalebox{0.30}{\includegraphics{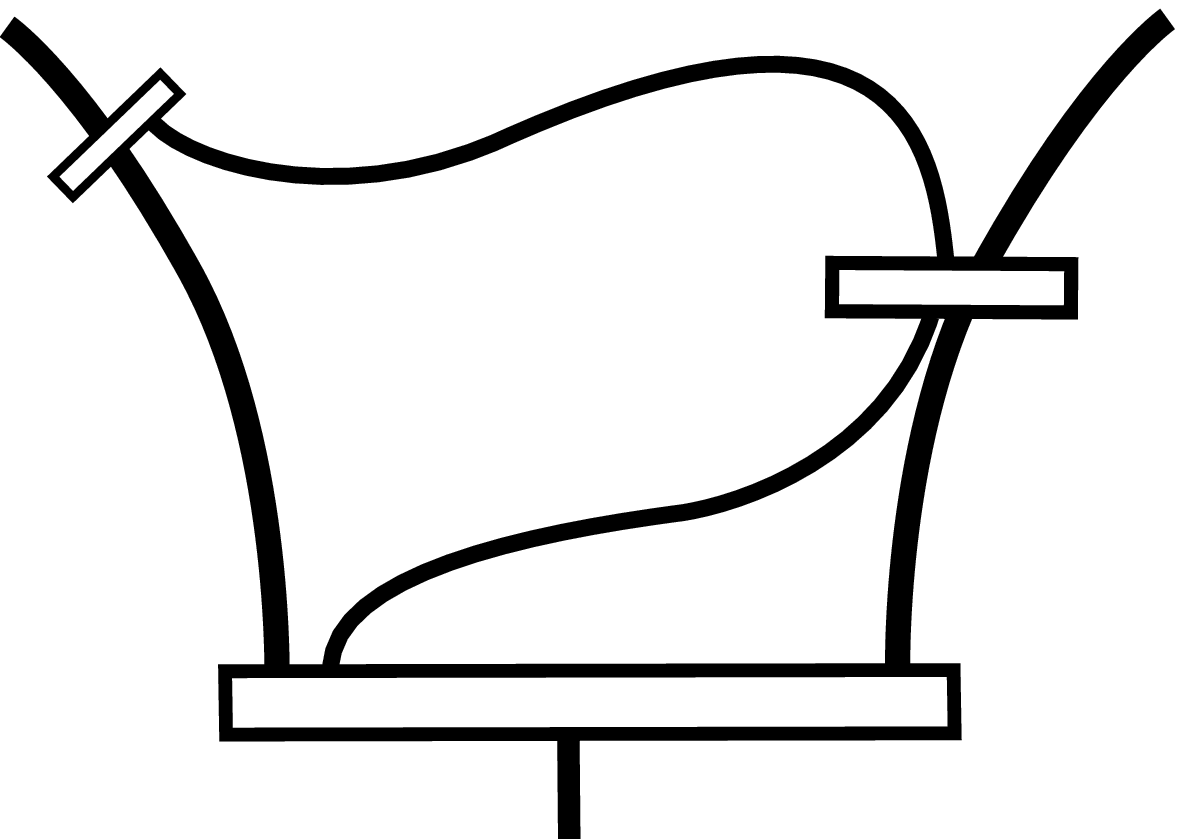}}
          \put(-68,+61){$i$}
         \put(-60,+30){$i$}
          \put(-60,-10){$2n$}
          \put(-1,60){$n$}
         \put(-108,60){$n$}
   \end{minipage}
  \end{eqnarray*}  
Using Lemma \ref{lma} (1) we obtain
\begin{eqnarray*}
    \begin{minipage}[h]{0.21\linewidth}
         \vspace{-10pt}
         \scalebox{0.30}{\includegraphics{two_bubbles}}
        \put(-68,+67){$n$}
         \put(-68,+27){$n$}
         \put(-1,60){$n$}
         \put(-108,60){$n$}
          \put(-40,+27){$n$}
         \put(-40,67){$n$}
         \put(-60,-10){$2n$}
   \end{minipage}&\doteq_n&\Psi(q^3,q)
   \begin{minipage}[h]{0.16\linewidth}
        \vspace{0pt}
        \scalebox{0.30}{\includegraphics{two_bubbles_4}}
         \put(-45,-10){$2n$}
          \put(-1,60){$n$}
         \put(-77,60){$n$}
           \end{minipage}
         \end{eqnarray*}
\item         
Using the bubble expansion formula on the left most bubble we obtain:
{\small
\begin{eqnarray*}
    \begin{minipage}[h]{0.26\linewidth}
         \vspace{5pt}
         \scalebox{0.28}{\includegraphics{three_bubbles}}
        \put(-88,+67){$n$}
         \put(-90,+33){$n$}
         \put(-1,60){$n$}
         \put(-128,60){$n$}
          \put(-40,+33){$n$}
         \put(-35,67){$n$}
         \put(-60,85){$n$}
         \put(-60,50){$n$}
         \put(-60,-10){$2n$}
   \end{minipage}&=&\sum\limits_{i=0}^{n}\left\lceil 
\begin{array}{cc}
n & n \\ 
n & n%
\end{array}%
\right\rceil _{i}
   \begin{minipage}[h]{0.16\linewidth}
        \vspace{5pt}
        \scalebox{0.28}{\includegraphics{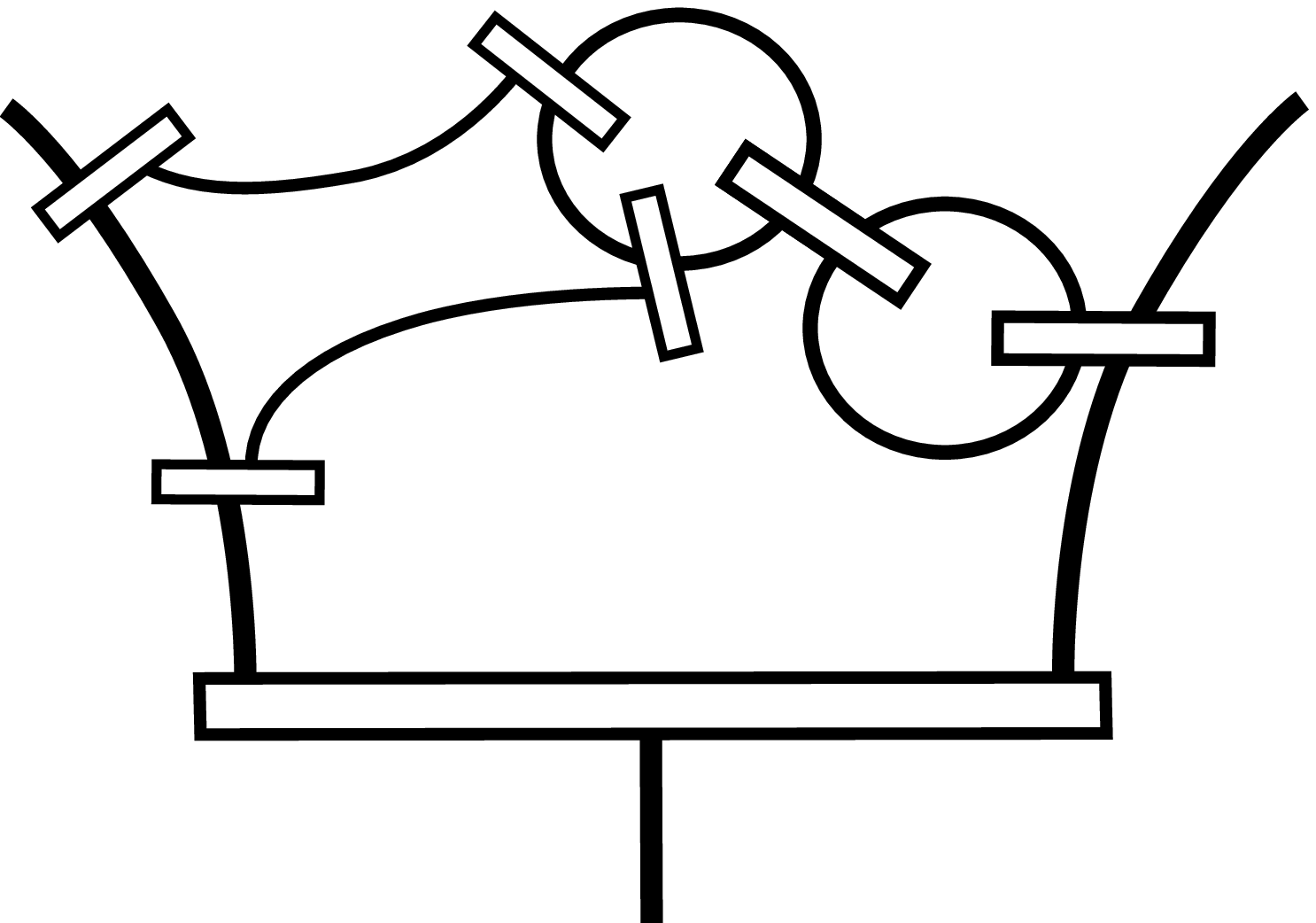}}
         \put(-60,-10){$2n$}
          \put(-1,60){$n$}
           \put(-128,60){$n$}
           \put(-30,68){$n$}
           \put(-38,35){$n$}
         \put(-89,73){$i$}
         \put(-89,45){$i$}
         \end{minipage}\\&=&\sum\limits_{i=0}^{n}\left\lceil 
\begin{array}{cc}
n & n \\ 
n & n%
\end{array}%
\right\rceil _{i}
   \begin{minipage}[h]{0.16\linewidth}
        \vspace{+5pt}
        \scalebox{0.28}{\includegraphics{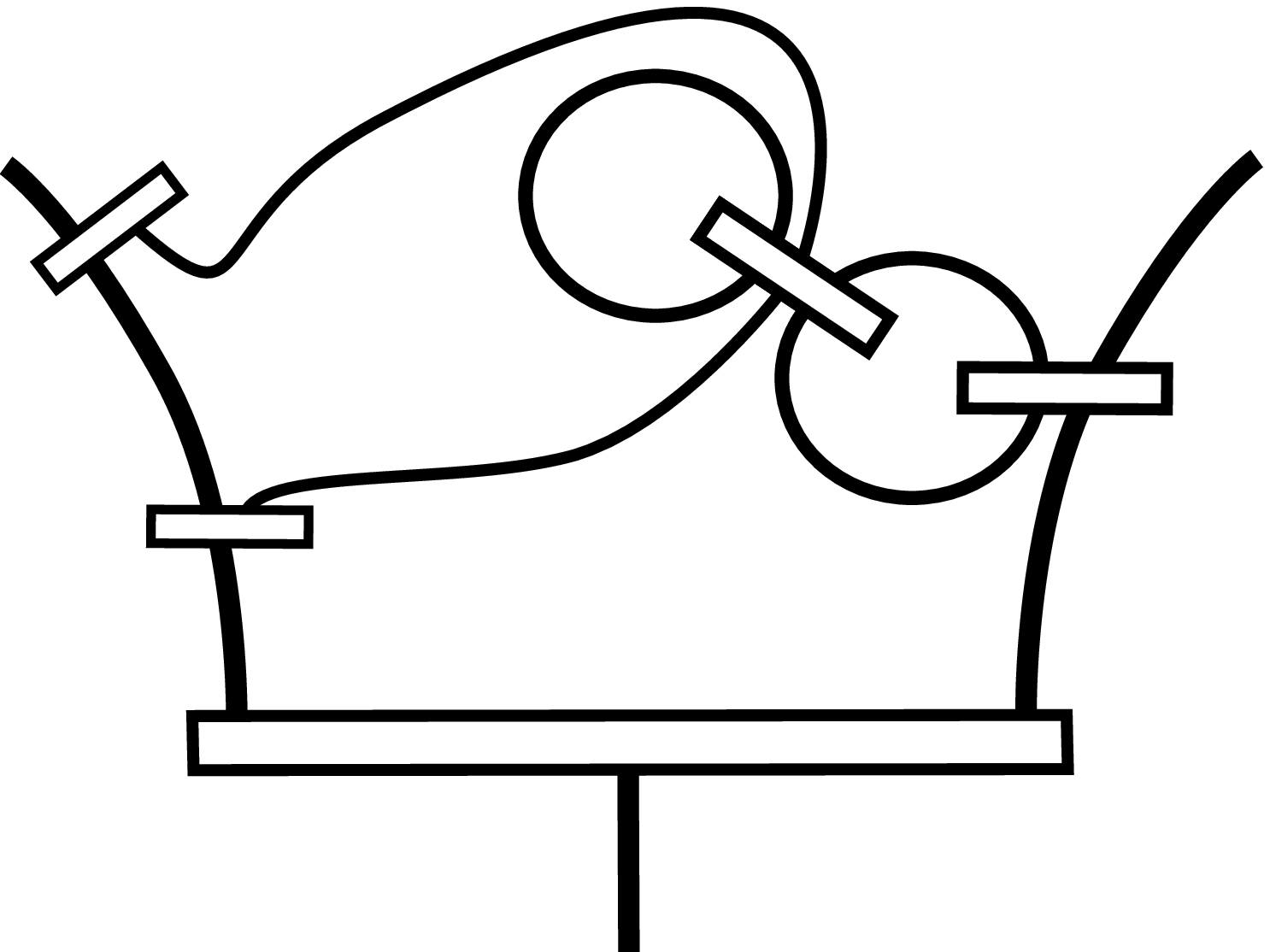}}
          \put(-60,-10){$2n$}
          \put(-1,60){$n$}
           \put(-128,60){$n$}
           \put(-30,68){$n$}
           \put(-38,35){$n$}
         \put(-89,79){$i$}
         \put(-90,58){$n-i$}      
           \end{minipage}
  \end{eqnarray*}
}
Using (\ref{properties2}) we can write the previous equation as:
{\small
\begin{eqnarray*}
    \begin{minipage}[h]{0.26\linewidth}
         \vspace{5pt}
         \scalebox{0.28}{\includegraphics{three_bubbles}}
        \put(-88,+67){$n$}
         \put(-90,+33){$n$}
         \put(-1,60){$n$}
         \put(-128,60){$n$}
          \put(-40,+33){$n$}
         \put(-35,67){$n$}
         \put(-60,85){$n$}
         \put(-60,50){$n$}
         \put(-60,-10){$2n$}
   \end{minipage}&=&\sum\limits_{i=0}^{n}\left\lceil 
\begin{array}{cc}
n & n \\ 
n & n%
\end{array}%
\right\rceil _{i}\frac{\Delta_{2n}}{\Delta_{n+i}}
   \begin{minipage}[h]{0.16\linewidth}
        \vspace{5pt}
        \scalebox{0.28}{\includegraphics{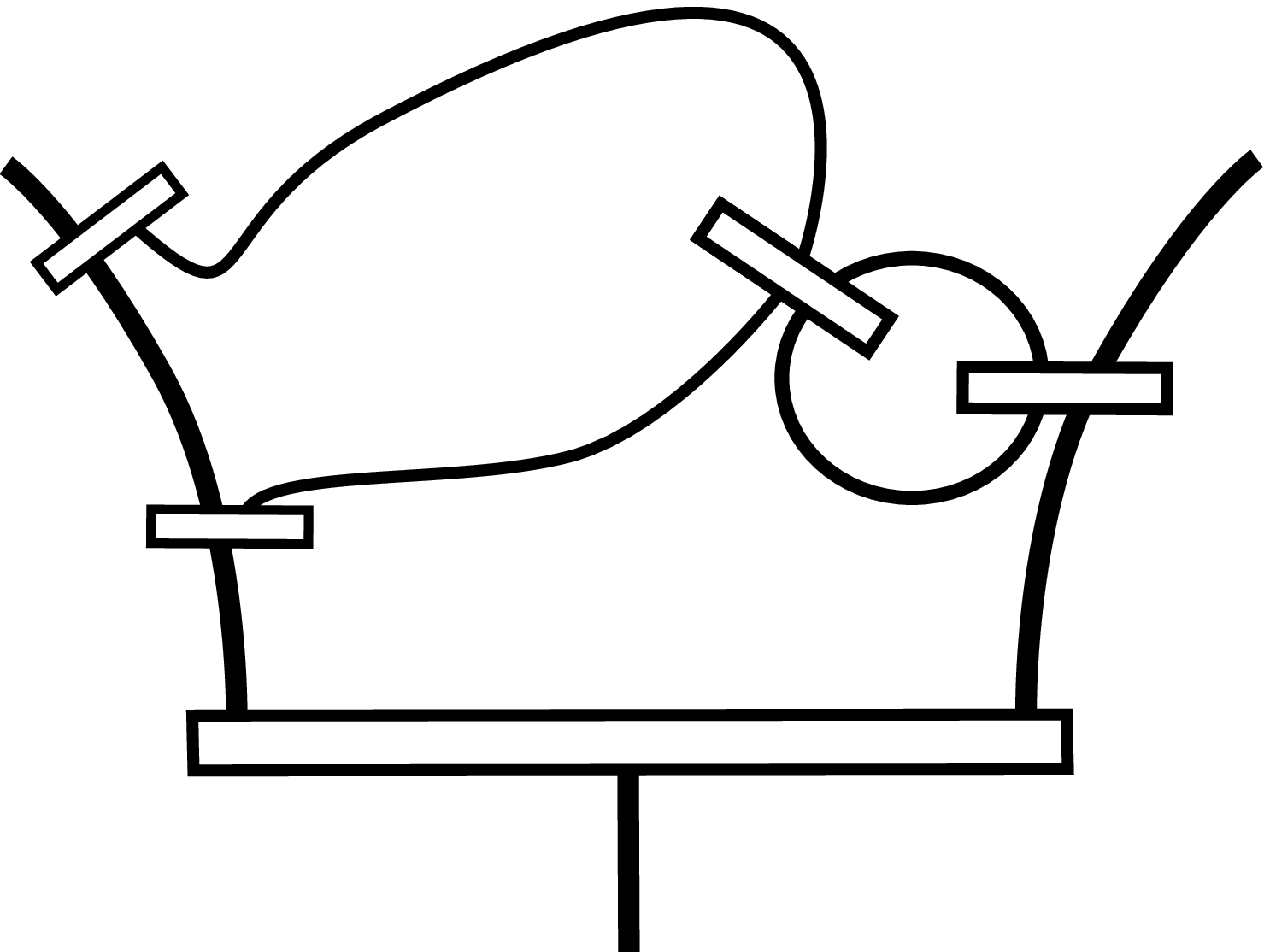}}
         \put(-60,-10){$2n$}
          \put(-1,60){$n$}
           \put(-128,60){$n$}
           \put(-30,68){$n$}
           \put(-38,35){$n$}
         \put(-89,79){$i$}
         \end{minipage}\\&=&\sum\limits_{i=0}^{n}\sum\limits_{j=0}^{i}\left\lceil 
\begin{array}{cc}
n & n \\ 
n & n%
\end{array}%
\right\rceil _{i}\left\lceil 
\begin{array}{cc}
n & i \\ 
n & n%
\end{array}%
\right\rceil _{j}\frac{\Delta_{2n}}{\Delta_{n+i}}
   \begin{minipage}[h]{0.16\linewidth}
        \vspace{5pt}
        \scalebox{0.28}{\includegraphics{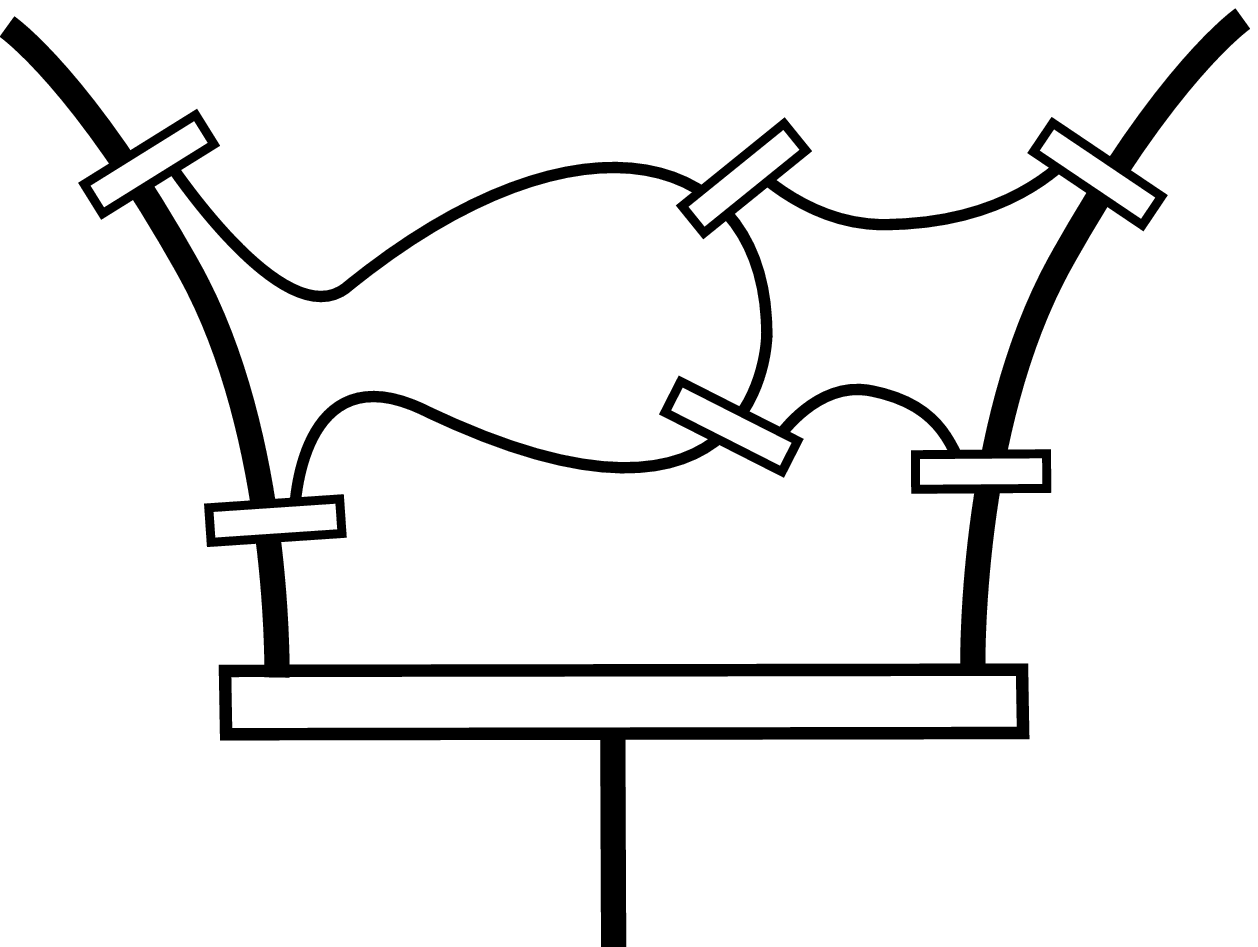}}
         \put(-60,-10){$2n$}
          \put(-1,60){$n$}
           \put(-105,60){$n$}
           \put(-30,68){$j$}
           \put(-35,35){$j$}
         \put(-70,64){$i$}    
           \end{minipage}
  \end{eqnarray*}
}
Hence,
{\small
\begin{eqnarray*}
    \begin{minipage}[h]{0.26\linewidth}
         \vspace{5pt}
         \scalebox{0.28}{\includegraphics{three_bubbles}}
        \put(-88,+67){$n$}
         \put(-90,+33){$n$}
         \put(-1,60){$n$}
         \put(-128,60){$n$}
          \put(-40,+33){$n$}
         \put(-35,67){$n$}
         \put(-60,85){$n$}
         \put(-60,50){$n$}
         \put(-60,-10){$2n$}
   \end{minipage}&=&\sum\limits_{i=0}^{n}\sum\limits_{j=0}^{i}\left\lceil 
\begin{array}{cc}
n & n \\ 
n & n%
\end{array}%
\right\rceil _{i}\left\lceil 
\begin{array}{cc}
n & i \\ 
n & n%
\end{array}%
\right\rceil _{j}\frac{\Delta_{2n}}{\Delta_{n+i}}
   \begin{minipage}[h]{0.16\linewidth}
        \vspace{5pt}
        \scalebox{0.28}{\includegraphics{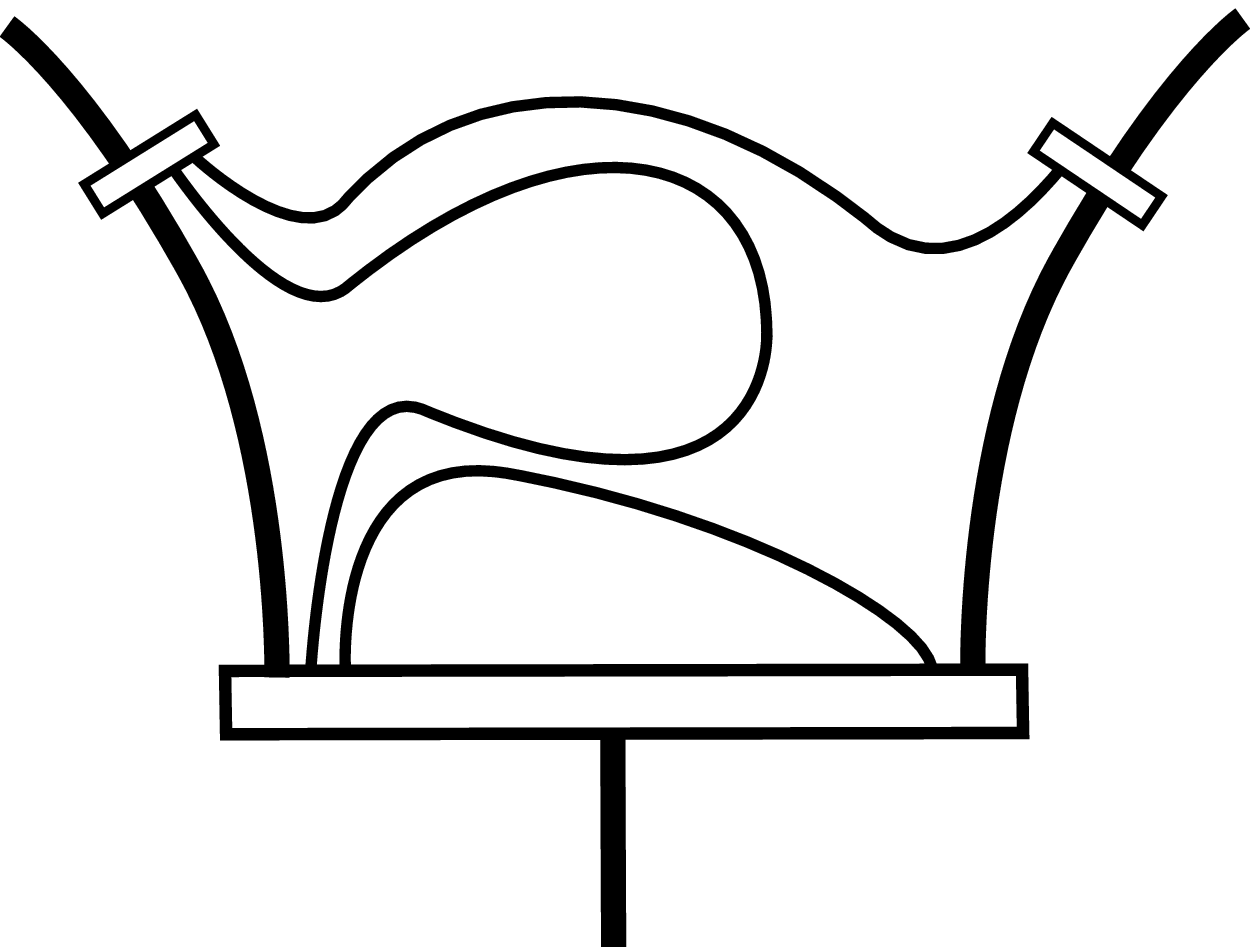}}
         \put(-60,-10){$2n$}
          \put(-1,60){$n$}
           \put(-105,60){$n$}
           \put(-30,68){$j$}
           \put(-35,35){$j$}
           \end{minipage}
  \end{eqnarray*}
}
The skein element in the on the right of the previous equation is zero unless $j=0$. Hence, Lemma \ref{lma} (2) yields the result.
\end{enumerate}

\end{proof}
More generally we have the following theorem.

\begin{theorem}
For all adequate closures of the element $\tau_{n,n,2n}$ and for all $n,k\geq1$:
\begin{enumerate}

\item \begin{eqnarray*}
   \begin{minipage}[h]{0.30\linewidth}
         \vspace{0pt}
         \scalebox{0.260}{\includegraphics{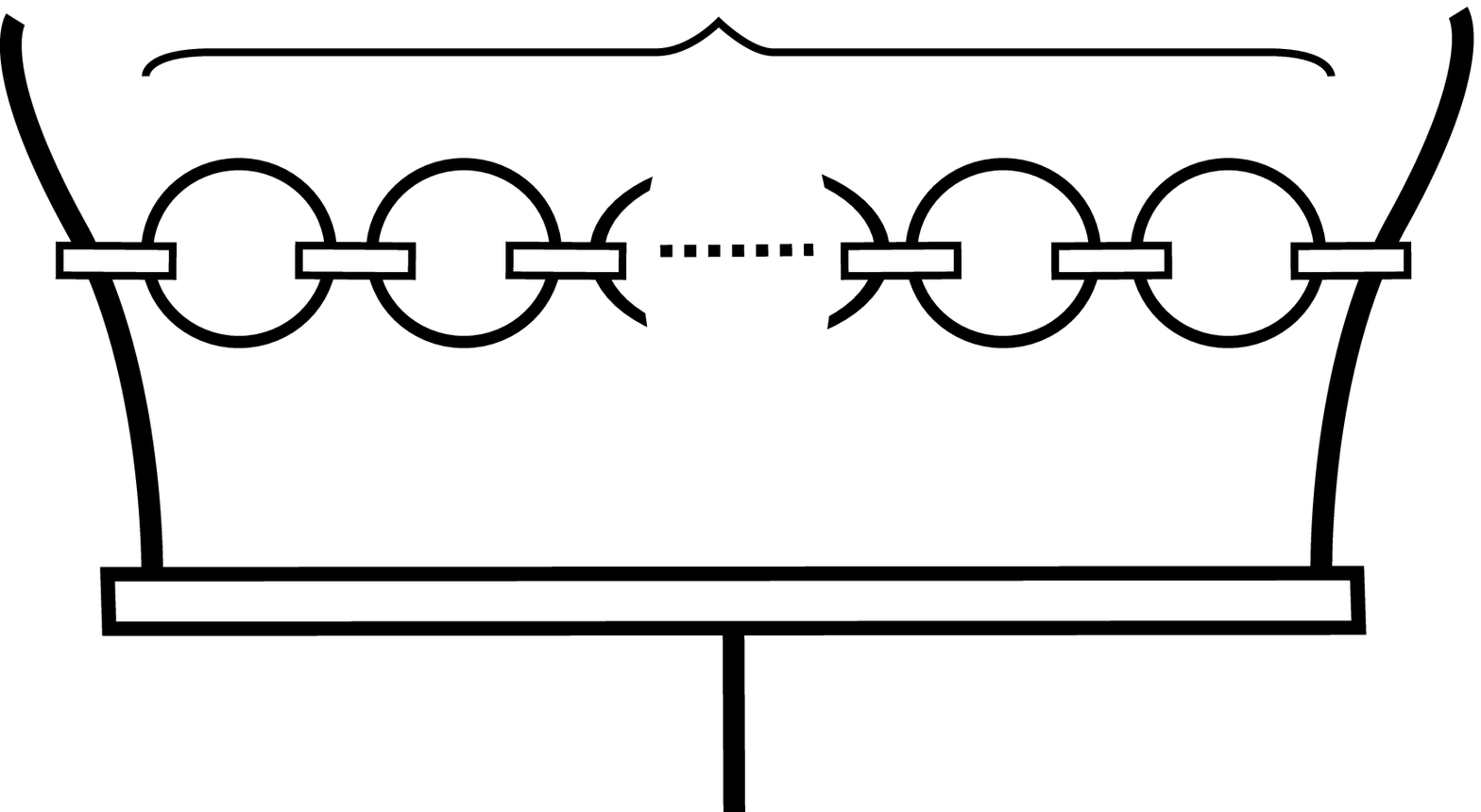}}
        \put(-90,+72){$2k$ bubbles}
         \put(-92,+35){$n$}
         \put(-112,+35){$n$}
             \put(-24,+35){$n$}
           \put(-45,+35){$n$}
             \put(-92,+58){$n$}
         \put(-112,+58){$n$}
             \put(-24,+58){$n$}
           \put(-45,+58){$n$}
         \put(4,60){$n$}
         \put(-138,60){$n$}
               \put(-75,-10){$2n$}
   \end{minipage}&\doteq_n&(q,q)_n\sum\limits_{l_1=0}^{n}\sum\limits_{l_{2}=0}^{n}...\sum\limits_{l_k=0}^{n}\frac{q^{\sum\limits_{j=1}^{k}(i_j(i_j+1))}}{(q,q)^2_{l_k}\prod\limits_{j=1}^{k-1}(q,q)_{l_j}}
   \begin{minipage}[h]{0.16\linewidth}
        \vspace{0pt}
        \scalebox{0.30}{\includegraphics{two_bubbles_4}}
         \put(-45,-10){$2n$}
          \put(-1,60){$n$}
         \put(-77,60){$n$}
           \end{minipage}
  \end{eqnarray*}
  where $i_j=\sum\limits_{s=j}^{k}l_s$.
\item \begin{eqnarray*}
   \begin{minipage}[h]{0.30\linewidth}
         \vspace{0pt}
         \scalebox{0.260}{\includegraphics{gn_bubbles}}
        \put(-90,+72){$2k+1$ bubbles}
         \put(-92,+35){$n$}
         \put(-112,+35){$n$}
             \put(-24,+35){$n$}
           \put(-45,+35){$n$}
             \put(-92,+58){$n$}
         \put(-112,+58){$n$}
             \put(-24,+58){$n$}
           \put(-45,+58){$n$}
         \put(4,60){$n$}
         \put(-138,60){$n$}
               \put(-75,-10){$2n$}
   \end{minipage}&\doteq_n&(q,q)_n\sum\limits_{l_1=0}^{n}\sum\limits_{l_{2}=0}^{n}...\sum\limits_{l_k=0}^{n}\frac{q^{\sum\limits_{j=1}^{k}(i_j(i_j+1))}}{\prod\limits_{j=1}^{k}(q,q)_{l_j}}
   \begin{minipage}[h]{0.16\linewidth}
        \vspace{0pt}
        \scalebox{0.30}{\includegraphics{two_bubbles_4}}
         \put(-45,-10){$2n$}
          \put(-1,60){$n$}
         \put(-77,60){$n$}
           \end{minipage}
  \end{eqnarray*}
  where $i_j=\sum\limits_{s=j}^{k}l_s$.
\end{enumerate}
\end{theorem}
\begin{proof}
\begin{enumerate}
\item
We proceed as in the previous theorem and we apply the bubble expansion formula on the left most bubble we obtain:
\begin{eqnarray*}
   \begin{minipage}[h]{0.32\linewidth}
         \vspace{0pt}
         \scalebox{0.260}{\includegraphics{gn_bubbles}}
        \put(-90,+72){$2k$ bubbles}
         \put(-93,+34){$n$}
         \put(-112,+34){$n$}
             \put(-25,+34){$n$}
           \put(-45,+34){$n$}
             \put(-93,+58){$n$}
         \put(-112,+58){$n$}
             \put(-25,+58){$n$}
           \put(-45,+58){$n$}
         \put(3,60){$n$}
         \put(-138,60){$n$}
               \put(-75,-10){$2n$}
   \end{minipage}&=&\sum\limits_{i_1=0}^{n}\left\lceil 
\begin{array}{cc}
n & n \\ 
n & n%
\end{array}%
\right\rceil _{i_1} 
   \begin{minipage}[h]{0.16\linewidth}
        \vspace{0pt}
        \scalebox{0.260}{\includegraphics{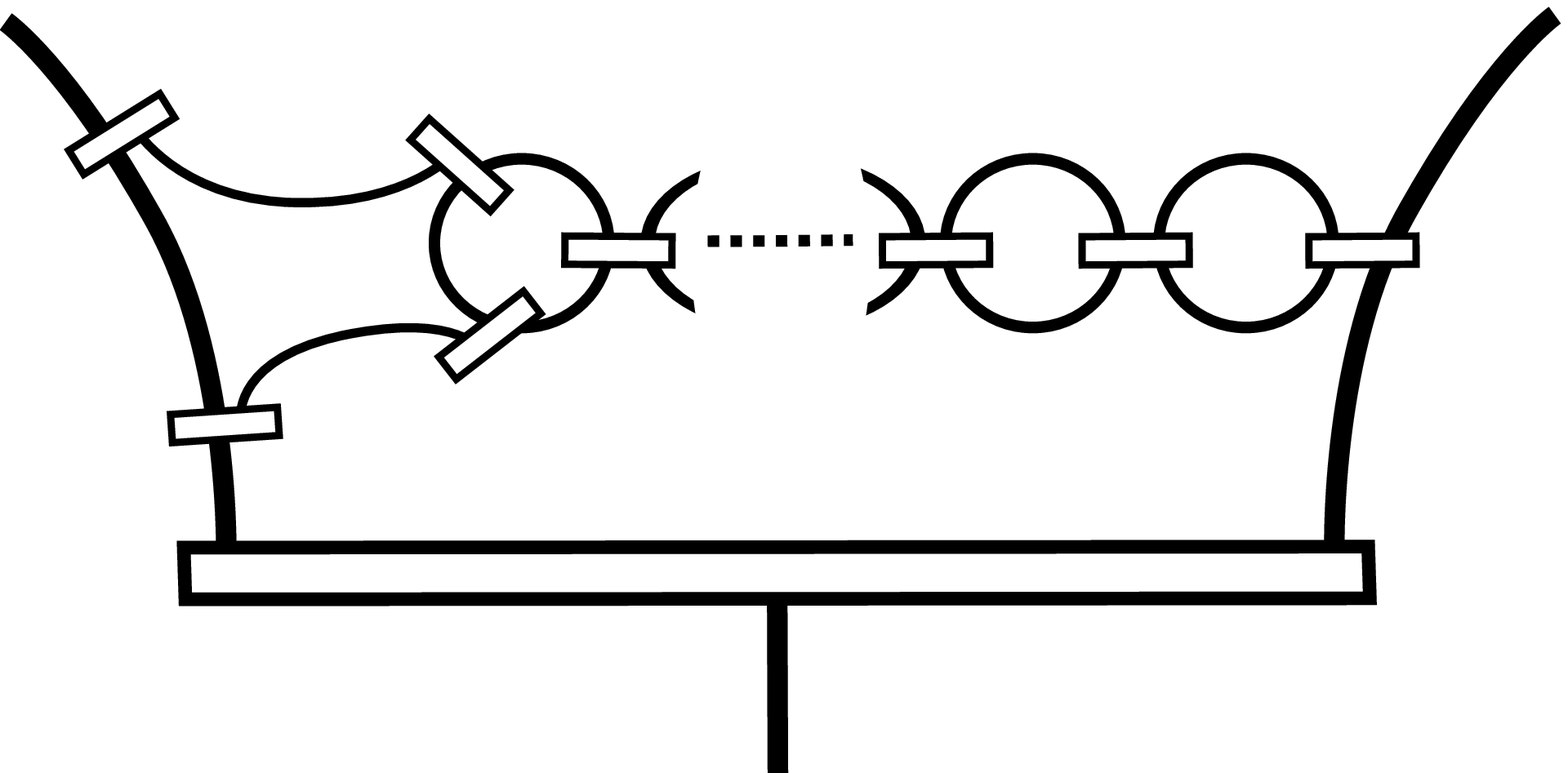}}
         \put(-95,+35){$n$}
         \put(-113,+32){$i_1$}
             \put(-30,+35){$n$}
           \put(-52,+35){$n$}
             \put(-95,+58){$n$}
         \put(-115,+58){$i_1$}
             \put(-30,+58){$n$}
           \put(-52,+58){$n$}
         \put(-1,60){$n$}
         \put(-148,60){$n$}
           \put(-75,-10){$2n$}
           \end{minipage}\\&=&\sum\limits_{i_1=0}^{n}\left\lceil 
\begin{array}{cc}
n & n \\ 
n & n%
\end{array}%
\right\rceil _{i_1} \frac{\Delta_{2n}}{\Delta_{n+i_1}}
   \begin{minipage}[h]{0.16\linewidth}
        \vspace{10pt}
        \scalebox{0.250}{\includegraphics{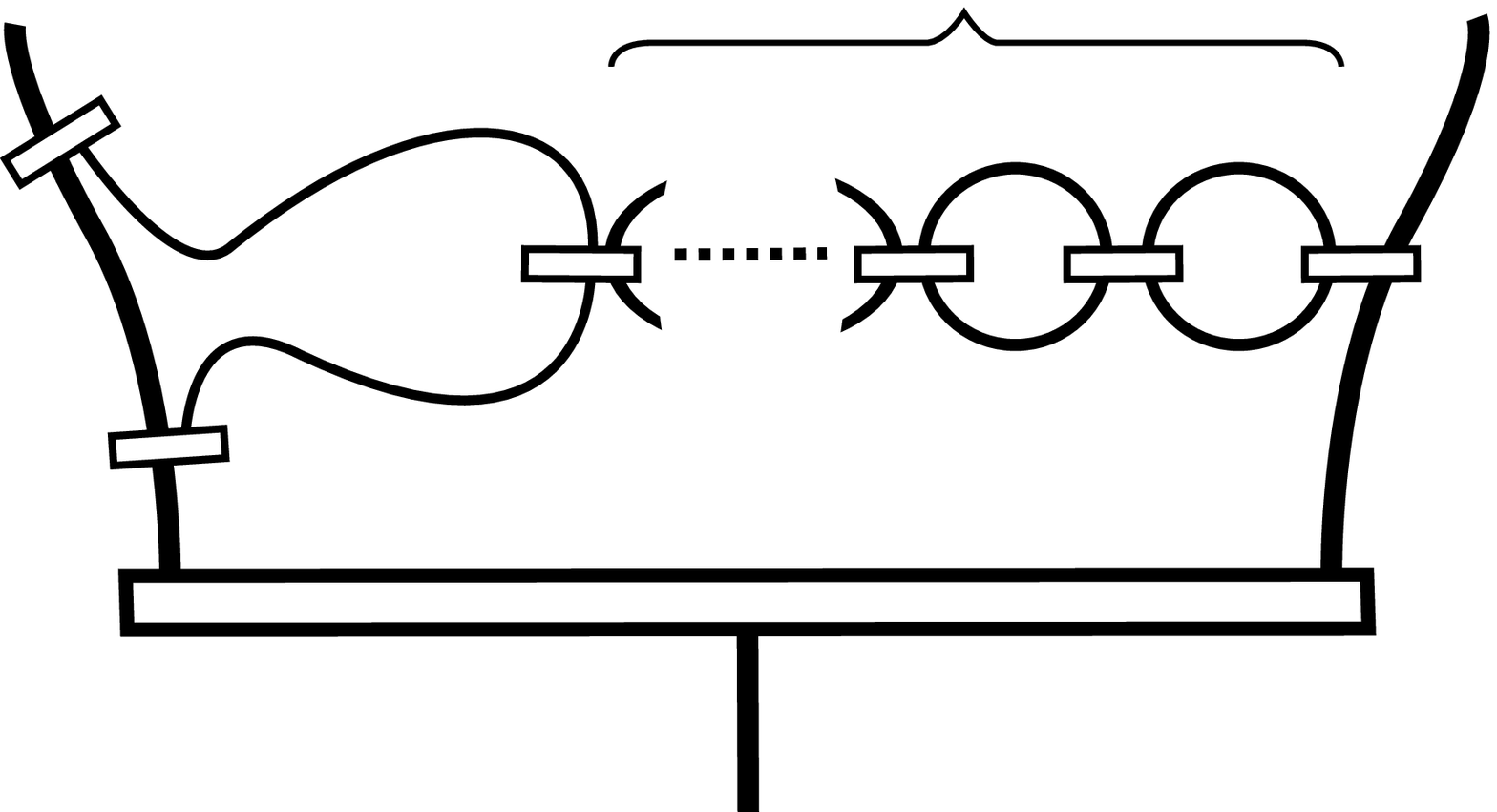}}
                  \put(-100,+27){$i_1$}
             \put(-25,+33){$n$}
           \put(-45,+33){$n$}
                \put(-105,+58){$i_1$}
             \put(-25,+58){$n$}
           \put(-45,+58){$n$}
         \put(-1,68){$n$}
          \put(-80,70){$2k-2$ bubbles}
         \put(-130,68){$n$}
           \put(-75,-10){$2n$}
           \end{minipage}
  \end{eqnarray*}
 Each time we apply the bubble expansion formula we eliminate two bubbles. Hence, after $k$ applications of the bubble expansion formula we obtain:
 {\scriptsize
 \begin{eqnarray*}
   \begin{minipage}[h]{0.25\linewidth}
         \vspace{0pt}
         \scalebox{0.240}{\includegraphics{gn_bubbles}}
        \put(-75,+67){$2k$ bubbles}
         \put(-86,+33){$n$}
         \put(-102,+33){$n$}
             \put(-23,+33){$n$}
           \put(-40,+33){$n$}
             \put(-86,+54){$n$}
         \put(-102,+54){$n$}
             \put(-23,+54){$n$}
           \put(-40,+54){$n$}
         \put(3,60){$n$}
         \put(-125,60){$n$}
               \put(-65,-5){$2n$}
   \end{minipage}&=&\sum\limits_{i_1=0}^{n}\sum\limits_{i_2=0}^{i_1}...\sum\limits_{i_k=0}^{i_{k-1}}\left\lceil 
\begin{array}{cc}
n & n \\ 
n & n%
\end{array}%
\right\rceil _{i_1}
\frac{\Delta_{2n}}{\Delta_{n+i_1}}
\prod_{j=2}^k
\left\lceil 
\begin{array}{cc}
n & i_{j-1} \\ 
n & n%
\end{array}%
\right\rceil _{i_{j}}
\frac{\Delta_{2n}}{\Delta_{n+i_j}}  
   \begin{minipage}[h]{0.16\linewidth}
        \vspace{0pt}
        \scalebox{0.260}{\includegraphics{two_bubbles_4}}
         \put(-40,-10){$2n$}
          \put(2,60){$n$}
         \put(-70,60){$n$}
           \end{minipage}
  \end{eqnarray*}
  }
Using \ref{fact} we can compute
  \begin{equation*}
\left\lceil 
\begin{array}{cc}
n & i \\ 
n & n%
\end{array}%
\right\rceil _{j}=(-1)^{j+n}q^{j^2+j/2-n/2}\frac{(q,q)^2_i(q,q)^4_n(q,q)_{2n+i-j+1}}{(q,q)_{i-j}(q,q)^2_j(q,q)_{2n}(q,q)_{n+i}(q,q)_{n+i+1}(q,q)^2_{n-j}}  
  \end{equation*}
Similar calculations to the ones we did in Lemma \ref{lma} implies:
{\scriptsize
  \begin{equation}
 \sum\limits_{i_1=0}^{n}\sum\limits_{i_2=0}^{i_1}...\sum\limits_{i_k=0}^{i_{k-1}}\left\lceil 
\begin{array}{cc}
n & n \\ 
n & n%
\end{array}%
\right\rceil _{i_1}
\frac{\Delta_{2n}}{\Delta_{n+i_1}}
\prod_{j=2}^k
\left\lceil 
\begin{array}{cc}
n & i_{j-1} \\ 
n & n%
\end{array}%
\right\rceil _{i_{j}}
\frac{\Delta_{2n}}{\Delta_{n+i_j}}\doteq_n(q,q)_n\sum\limits_{i_1=0}^{n}\sum\limits_{i_2=0}^{i_1}...\sum\limits_{i_k=0}^{i_{k-1}}\frac{q^{\sum\limits_{j=1}^{k}(i_j(i_j+1))}}{(q,q)^2_{i_k}\prod\limits_{j=2}^{k}(q,q)_{i_{j-1}-i_j}}
  \end{equation}
}
The previous summation can be written as
\begin{equation*}
\label{complicated111}
(q,q)_n\sum\limits_{i_1=0}^{n}\sum\limits_{i_2=0}^{i_1}...\sum\limits_{i_k=0}^{i_{k-1}}\frac{q^{\sum\limits_{j=1}^{k}(i_j(i_j+1))}}{(q,q)^2_{i_k}\prod\limits_{j=2}^{k}(q,q)_{i_{j-1}-i_j}}=(q,q)_n\sum\limits_{i_k=0}^{n}\sum\limits_{i_{k-1}=i_k}^{n}...\sum\limits_{i_1=i_2}^{n}\frac{q^{\sum\limits_{j=1}^{k}(i_j(i_j+1))}}{(q,q)^2_{i_k}\prod\limits_{j=2}^{k}(q,q)_{i_{j-1}-i_j}}
\end{equation*}
Now if we set $l_{j}=i_{j}-i_{j+1}$ for $j=1,...,k-1$ and $l_k=i_k$, we obtain $i_j=\sum\limits_{s=j}^{k}l_s$ and hence we can rewrite the right side of the previous equation as
\begin{equation*}
(q,q)_n\sum\limits_{i_k=0}^{n}\sum\limits_{i_{k-1}=i_k}^{n}...\sum\limits_{i_1=i_2}^{n}\frac{q^{\sum\limits_{j=1}^{k}(i_j(i_j+1))}}{(q,q)^2_{i_k}\prod\limits_{j=2}^{k}(q,q)_{i_{j-1}-i_j}}=(q,q)_n\sum\limits_{l_1=0}^{n}\sum\limits_{l_{2}=0}^{n}...\sum\limits_{l_k=0}^{n}\frac{q^{\sum\limits_{j=1}^{k}(i_j(i_j+1))}}{(q,q)^2_{l_k}\prod\limits_{j=1}^{k-1}(q,q)_{l_j}}
\end{equation*}
where $i_j=\sum\limits_{s=j}^{k}l_s$. Hence the result follows.
  \item
We apply the bubble expansion formula $k$ times we obtain:
{\footnotesize
 \begin{eqnarray*}
   \begin{minipage}[h]{0.25\linewidth}
         \vspace{0pt}
         \scalebox{0.240}{\includegraphics{gn_bubbles}}
        \put(-75,+67){$2k+1$ bubbles}
         \put(-86,+33){$n$}
         \put(-102,+33){$n$}
             \put(-23,+33){$n$}
           \put(-40,+33){$n$}
             \put(-86,+54){$n$}
         \put(-102,+54){$n$}
             \put(-23,+54){$n$}
           \put(-40,+54){$n$}
         \put(3,60){$n$}
         \put(-125,60){$n$}
               \put(-65,-5){$2n$}
   \end{minipage}&=&\sum\limits_{i_1=0}^{n}\sum\limits_{i_2=0}^{i_1}...\sum\limits_{i_k=0}^{i_{k-1}}\left\lceil 
\begin{array}{cc}
n & n \\ 
n & n%
\end{array}%
\right\rceil _{i_1}
\frac{\Delta_{2n}}{\Delta_{n+i_1}}
\prod_{j=2}^k
\left\lceil 
\begin{array}{cc}
n & i_{j-1} \\ 
n & n%
\end{array}%
\right\rceil _{i_{j}}
\frac{\Delta_{2n}}{\Delta_{n+i_j}}  
   \begin{minipage}[h]{0.14\linewidth}
        \vspace{10pt}
        \scalebox{0.260}{\includegraphics{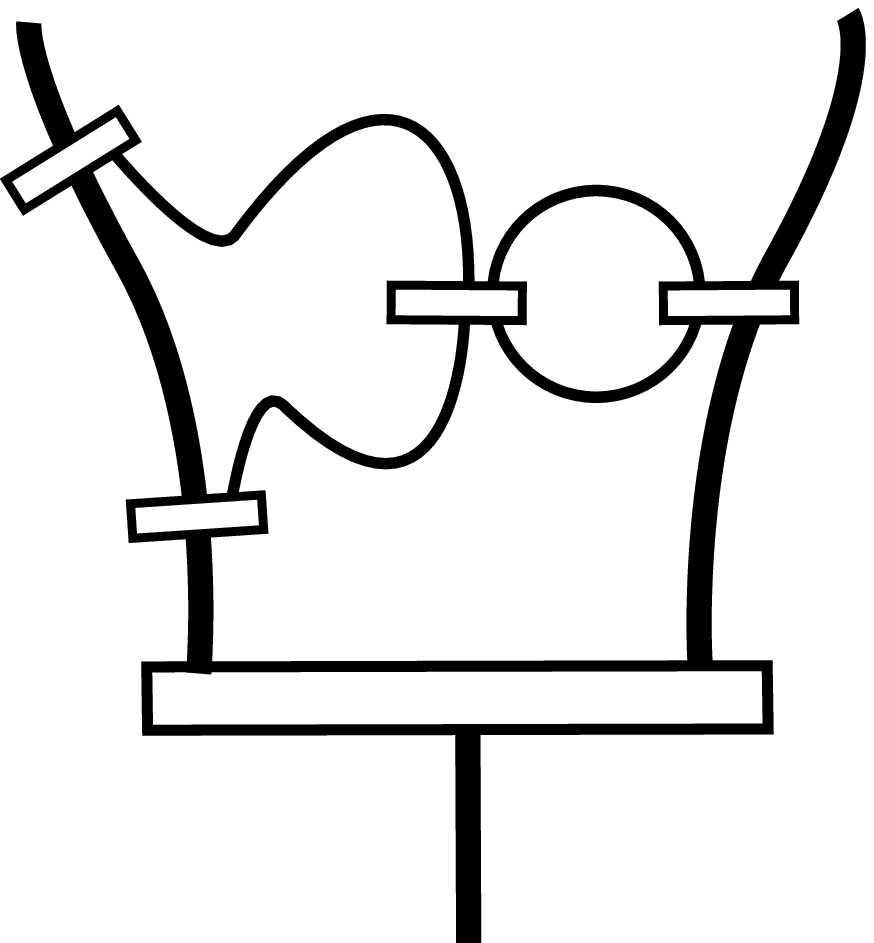}}
         \put(-50,-10){$2n$}
          \put(-50,60){$i_k$}
          \put(-40,28){$i_k$}
           \put(-70,60){$n$}
         \put(0,60){$n$}
          \put(-23,60){$n$}
          \put(-23,35){$n$}
           \end{minipage}
  \end{eqnarray*}
}
Hence,
{\scriptsize
 \begin{eqnarray*}
   \begin{minipage}[h]{0.23\linewidth}
         \vspace{0pt}
         \scalebox{0.220}{\includegraphics{gn_bubbles}}
        \put(-75,+62){$2k+1$ bubbles}
         \put(-78,+28){$n$}
         \put(-92,+28){$n$}
             \put(-23,+28){$n$}
           \put(-40,+28){$n$}
             \put(-78,+50){$n$}
         \put(-92,+50){$n$}
             \put(-23,+50){$n$}
           \put(-40,+50){$n$}
         \put(-1,60){$n$}
         \put(-110,60){$n$}
               \put(-65,-5){$2n$}
   \end{minipage}&=&\sum\limits_{i_1=0}^{n}\sum\limits_{i_2=0}^{i_1}...\sum\limits_{i_k=0}^{i_{k-1}}\left\lceil 
\begin{array}{cc}
n & n \\ 
n & n%
\end{array}%
\right\rceil _{i_1}
\left\lceil 
\begin{array}{cc}
n & i_{k} \\ 
n & n%
\end{array}%
\right\rceil _{0}
\frac{\Delta_{2n}}{\Delta_{n+i_1}}
\prod_{j=2}^k
\left\lceil 
\begin{array}{cc}
n & i_{j-1} \\ 
n & n%
\end{array}%
\right\rceil _{i_{j}}
\frac{\Delta_{2n}}{\Delta_{n+i_j}}  
   \begin{minipage}[h]{0.10\linewidth}
        \vspace{10pt}
        \scalebox{0.20}{\includegraphics{two_bubbles_4}}
         \put(-30,-10){$2n$}
          \put(-5,50){$n$}
         \put(-50,50){$n$}
           \end{minipage}
  \end{eqnarray*}
}
and one can do computations to the coefficient in the last equation similar to the ones we did in Lemma \ref{lma} and obtain:
{\scriptsize
  \begin{equation*}
 \sum\limits_{i_1=0}^{n}\sum\limits_{i_2=0}^{i_1}...\sum\limits_{i_k=0}^{i_{k-1}}\left\lceil 
\begin{array}{cc}
n & n \\ 
n & n%
\end{array}%
\right\rceil _{i_1}\left\lceil 
\begin{array}{cc}
n & i_{k} \\ 
n & n%
\end{array}%
\right\rceil _{0}
\frac{\Delta_{2n}}{\Delta_{n+i_1}}
\prod_{j=2}^k
\left\lceil 
\begin{array}{cc}
n & i_{j-1} \\ 
n & n%
\end{array}%
\right\rceil _{i_{j}}
\frac{\Delta_{2n}}{\Delta_{n+i_j}}\doteq_n(q,q)_n\sum\limits_{i_1=0}^{n}\sum\limits_{i_2=0}^{i_1}...\sum\limits_{i_k=0}^{i_{k-1}}\frac{q^{\sum\limits_{j=1}^{k}(i_j(i_j+1))}}{(q,q)_{i_k}\prod\limits_{j=2}^{k}(q,q)_{i_{j-1}-i_j}}
  \end{equation*}
}  
The previous summation can be rewritten as follows:
  \begin{equation*}
(q,q)_n\sum\limits_{i_1=0}^{n}\sum\limits_{i_2=0}^{i_1}...\sum\limits_{i_k=0}^{i_{k-1}}\frac{q^{\sum\limits_{j=1}^{k}(i_j(i_j+1))}}{(q,q)_{i_k}\prod\limits_{j=2}^{k}(q,q)_{i_{j-1}-i_j}}=(q,q)_n\sum\limits_{i_k=0}^{n}\sum\limits_{i_{k-1}=i_k}^{n}...\sum\limits_{i_1=i_2}^{n}\frac{q^{\sum\limits_{j=1}^{k}(i_j(i_j+1))}}{(q,q)_{i_k}\prod\limits_{j=2}^{k}(q,q)_{i_{j-1}-i_j}}
\end{equation*}
Set $l_j=i_{j}-i_{j+1}$ for $j=1,...,k-1$ and $l_k=i_k$, we obtain $i_j=\sum\limits_{s=j}^{k}l_s$ and hence we can rewrite the previous equation:
\begin{equation*}
(q,q)_n\sum\limits_{i_k=0}^{n}\sum\limits_{i_{k-1}=i_k}^{n}...\sum\limits_{i_1=i_2}^{n}\frac{q^{\sum\limits_{j=1}^{k}(i_j(i_j+1))}}{(q,q)_{i_k}\prod\limits_{j=2}^{k}(q,q)_{i_{j-1}-i_j}}=(q,q)_n\sum\limits_{l_1=0}^{n}\sum\limits_{l_{2}=0}^{n}...\sum\limits_{l_k=0}^{n}\frac{q^{\sum\limits_{j=1}^{k}(i_j(i_j+1))}}{\prod\limits_{j=1}^{k}(q,q)_{l_j}}
\end{equation*}
where $i_j=\sum\limits_{s=j}^{k}l_s$.
  \end{enumerate}
\end{proof}
\begin{corollary}
\label{torususe}
For all adequate closures of the element $f^{(n)}$ and for all $n,k\geq1$:
\begin{enumerate}
{\small
\item \begin{eqnarray*}
   \begin{minipage}[h]{0.10\linewidth}
         \vspace{0pt}
         \scalebox{0.33}{\includegraphics{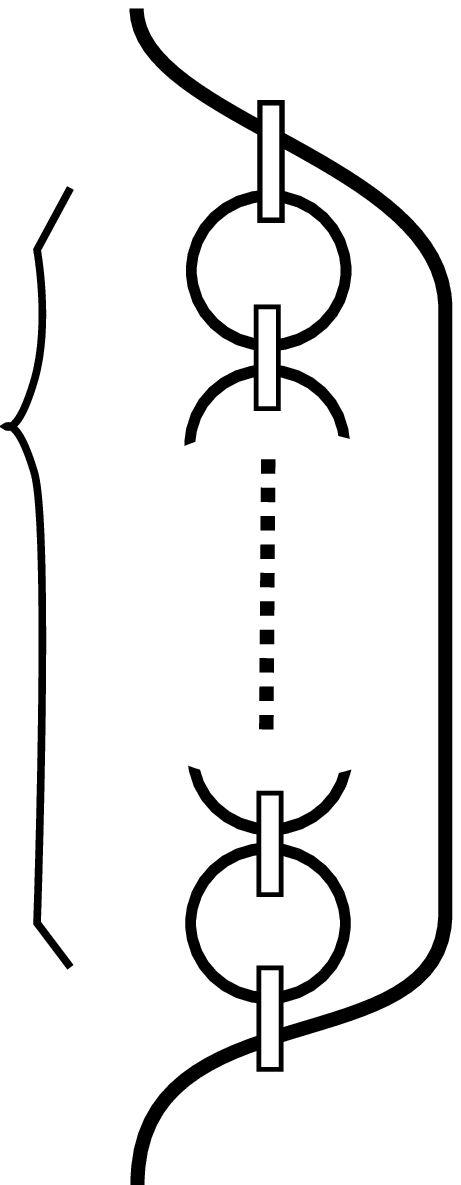}}
         \put(-35,-7){$n$}
         \put(-39,22){$n$}
           \put(-118,72){$2k+1$ bubbles}
          \put(-16,22){$n$}
          \put(-39,82){$n$}
          \put(-16,82){$n$}
         \put(-35,115){$n$}
         \end{minipage}&\doteq_n&(q,q)_n\sum\limits_{l_1=0}^{n}\sum\limits_{l_{2}=0}^{n}...\sum\limits_{l_k=0}^{n}\frac{q^{\sum\limits_{j=1}^{k}(i_j(i_j+1))}}{(q,q)^2_{l_k}\prod\limits_{j=1}^{k-1}(q,q)_{l_j}}
   \begin{minipage}[h]{0.18\linewidth}
        \vspace{0pt}
        \scalebox{0.130}{\includegraphics{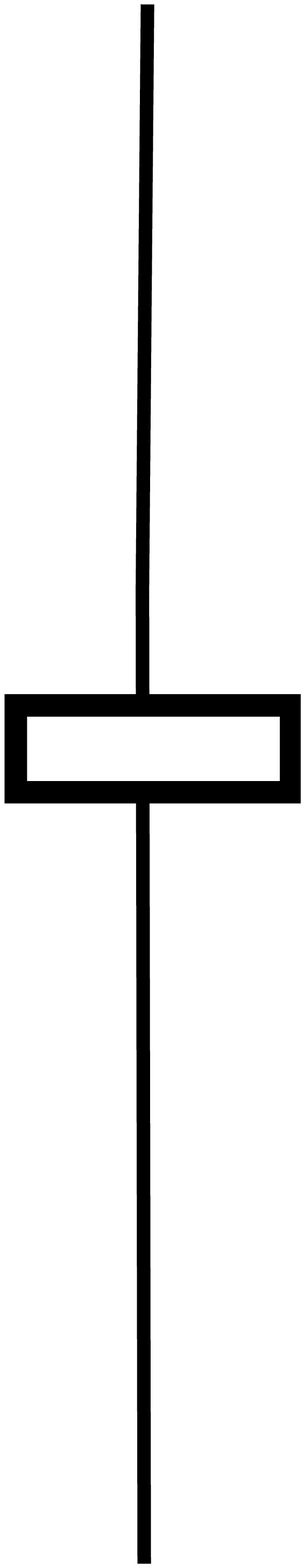}}
                 \put(-5,90){$n$}
                \end{minipage}
  \end{eqnarray*}
 where $i_j=\sum\limits_{s=j}^{k}l_s$. 
\item \begin{eqnarray*}
   \begin{minipage}[h]{0.10\linewidth}
         \vspace{0pt}
         \scalebox{0.33}{\includegraphics{bubble_iso}}
       \put(-35,-7){$n$}
         \put(-39,22){$n$}
          \put(-16,22){$n$}
          \put(-39,82){$n$}
          \put(-16,82){$n$}
         \put(-35,115){$n$}
           \put(-103,72){$2k$ bubbles}
                \end{minipage}&\doteq_n&(q,q)_n\sum\limits_{l_1=0}^{n}\sum\limits_{l_{2}=0}^{n}...\sum\limits_{l_{k-1}=0}^{n}\frac{q^{\sum\limits_{j=1}^{k-1}(i_j(i_j+1))}}{\prod\limits_{j=1}^{k-1}(q,q)_{l_j}}
   \begin{minipage}[h]{0.18\linewidth}
        \vspace{0pt}
        \scalebox{0.130}{\includegraphics{space4}}
                 \put(-5,90){$n$}
                \end{minipage}
  \end{eqnarray*}
  }
  where $i_j=\sum\limits_{s=j}^{k-1}l_s$.
\end{enumerate}
\end{corollary}
\begin{proof}
(1) Let $F:T_{n,n,2n}\longrightarrow T_{n,n}$ be the wiring linear map defined by
{\small
\item \begin{eqnarray*}
   \begin{minipage}[h]{.5\linewidth}
         \vspace{0pt}
                  \scalebox{0.30}{\includegraphics{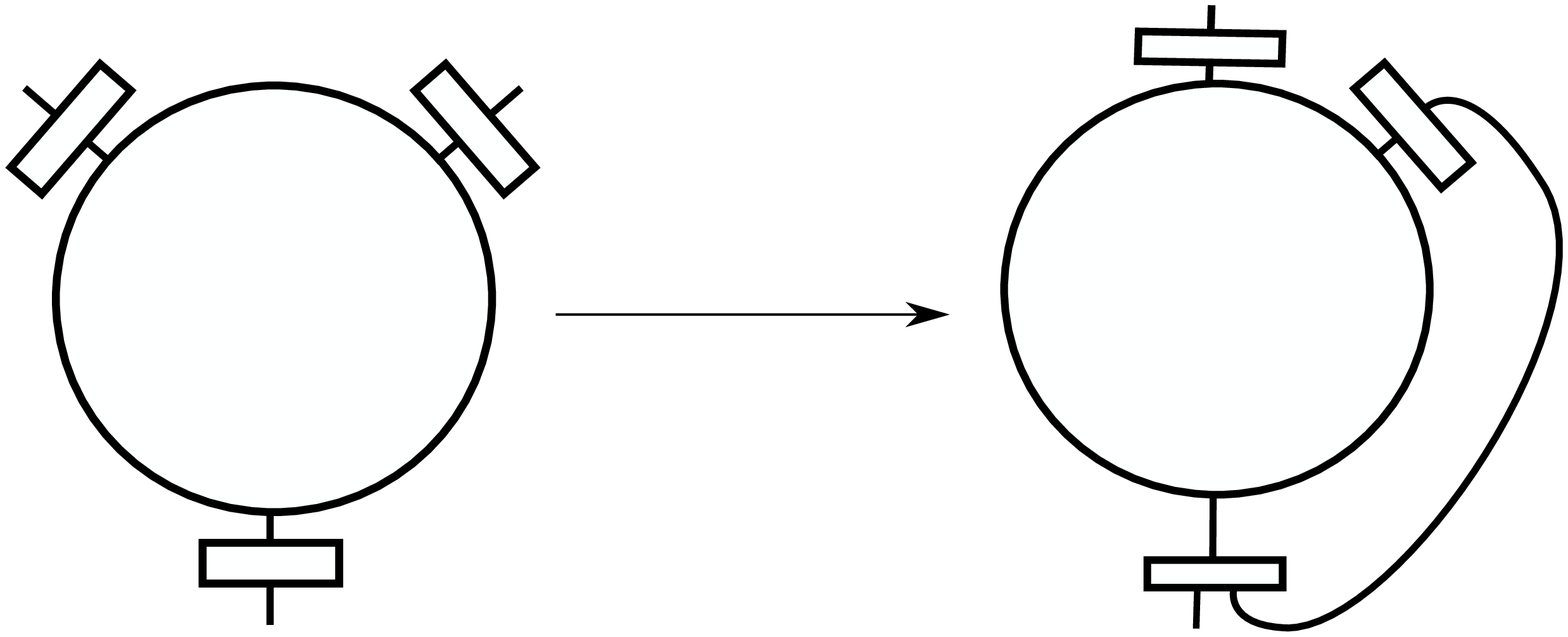}}
          \put(-35,-7){$n$}
           \put(-55,-7){$n$}
            \put(-55,85){$n$}
             \put(-180,-7){$2n$}
               \put(-110,50){$F$}
            \put(-150,80){$n$}
            \put(-215,80){$n$}
         \end{minipage}
  \end{eqnarray*}
  }
  This map is clearly an isomorphism. The result follows by noticing that 
{\footnotesize  
 \begin{eqnarray*}F\Bigg(\hspace{10pt}
   \begin{minipage}[h]{0.27\linewidth}
         \vspace{0pt}
         \scalebox{0.240}{\includegraphics{gn_bubbles}}
             \put(-86,+33){$n$}
         \put(-102,+33){$n$}
             \put(-23,+33){$n$}
           \put(-40,+33){$n$}
             \put(-86,+54){$n$}
         \put(-102,+54){$n$}
             \put(-23,+54){$n$}
           \put(-40,+54){$n$}
         \put(3,60){$n$}
         \put(-125,60){$n$}
         \put(-90,65){$2k$ bubbles}
               \put(-65,-5){$2n$}
   \end{minipage}\Bigg)=\begin{minipage}[h]{0.10\linewidth}
         \vspace{0pt}
         \scalebox{0.33}{\includegraphics{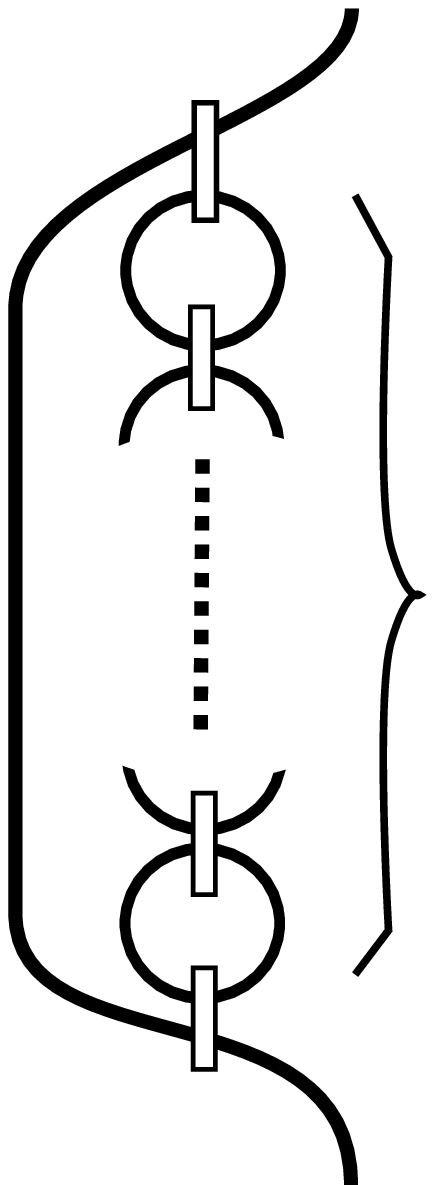}}
         \put(-15,-7){$n$}
         \put(-39,22){$n$}
           \put(-16,22){$n$}
          \put(-39,82){$n$}
          \put(-16,82){$n$}
         \put(-15,115){$n$}
         \put(0,65){$2k+1$ bubbles}
         \end{minipage}
  \end{eqnarray*}}
 (2) The proof is similar to (1). 
\end{proof}
The previous theorem and its corollary give an interesting proof of the Andrews-Gordon identities for the theta function and corresponding identities for the false theta function. We give this proof in section \ref{section6}.
\begin{theorem}
\label{states}
For all adequate closures of the element $\tau_{2n,2n,2n}$ and for all $n\geq 0 $:
\begin{eqnarray*}
   \begin{minipage}[h]{0.20\linewidth}
         \vspace{10pt}
         \scalebox{0.3}{\includegraphics{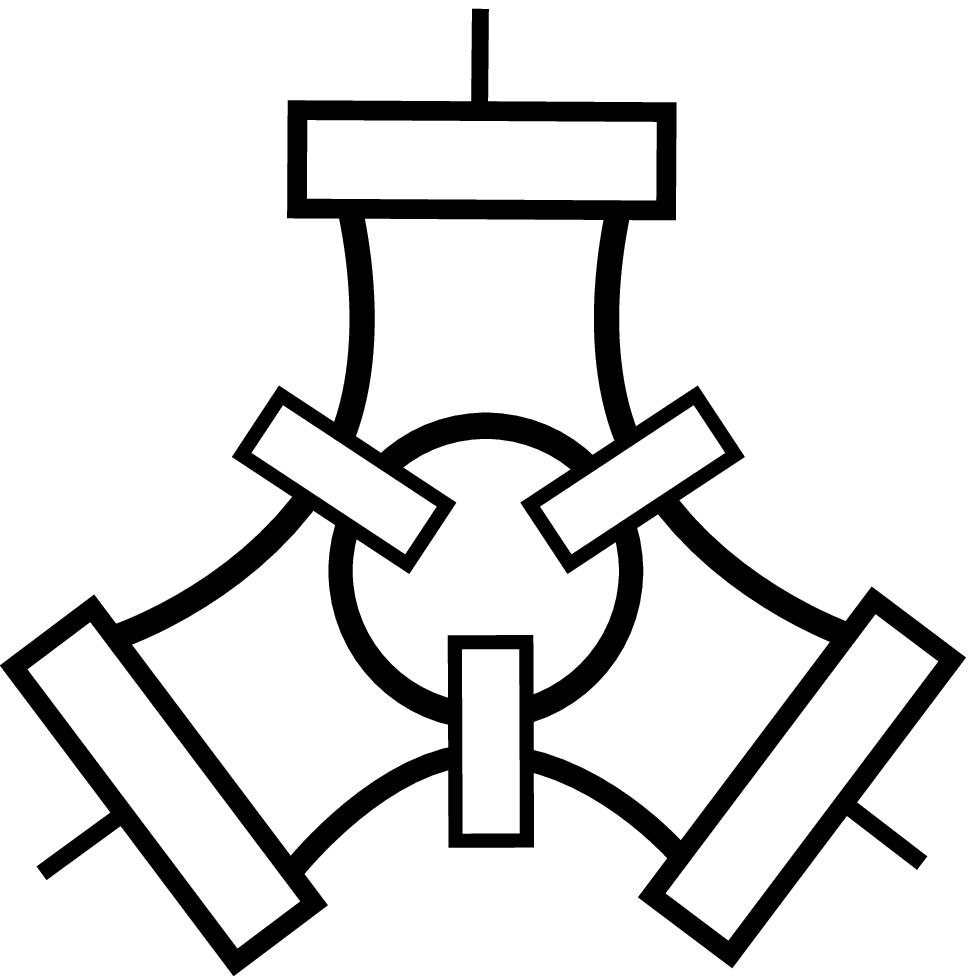}}
  \put(-47,52){$n$} 
  \put(-63,27){$n$}           
 \put(-27,27){$n$}  
          \put(-50,87){$2n$}
          \put(-10,-2){$2n$}
         \put(-90,-2){$2n$}
   \end{minipage}&\doteq_n&\Lambda(q)
   \begin{minipage}[h]{0.16\linewidth}
        \vspace{0pt}
        \scalebox{0.30}{\includegraphics{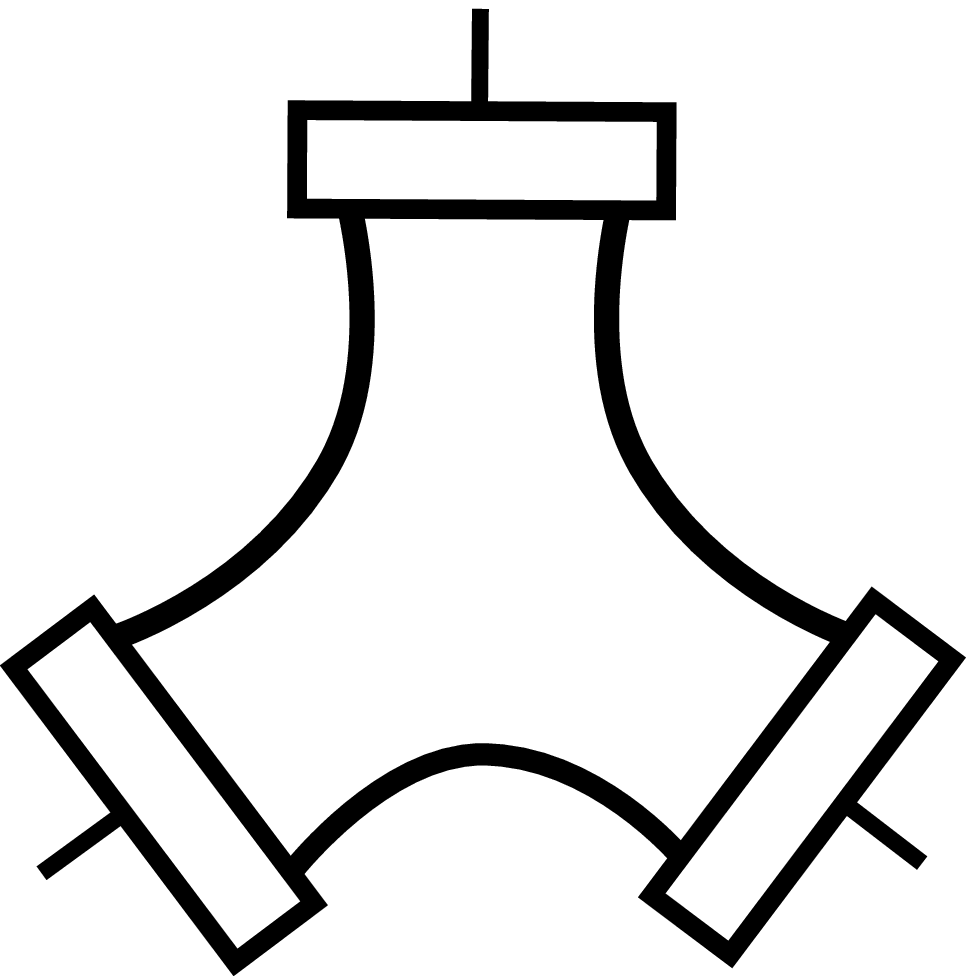}}
 \put(-50,87){$2n$}
          \put(-10,-2){$2n$}
         \put(-90,-2){$2n$}
           \end{minipage}
  \end{eqnarray*}
  where
  \begin{equation}
  \label{lamba}
  \Lambda(q)=(q;q)^2_\infty\sum\limits_{i=0}^{\infty}\frac{(-1)^{i} q^{(i +3i^2)/2}}{(q;q)^3_i}
  \end{equation}
\end{theorem}
\begin{proof}
First note that
\begin{eqnarray*}
   \begin{minipage}[h]{0.20\linewidth}
         \vspace{-10pt}
         \scalebox{0.30}{\includegraphics{all_colored_2n}}
  \put(-47,52){$n$} 
  \put(-63,27){$n$}           
 \put(-27,27){$n$}  
          \put(-50,87){$2n$}
          \put(-10,-2){$2n$}
         \put(-90,-2){$2n$}
   \end{minipage}&=& \begin{minipage}[h]{0.21\linewidth}
         \vspace{-0pt}
         \scalebox{0.35}{\includegraphics{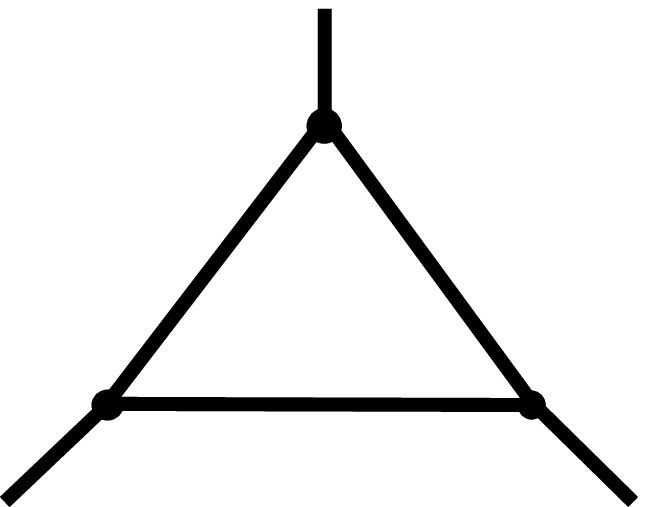}}
     \put(-69,-10){$2n$}
          \put(-1,-10){$2n$}
         \put(-26,50){$2n$}
         \put(-59,25){$2n$}
         \put(-16,25){$2n$}
          \put(-36,0){$2n$}
                    \end{minipage}
  \end{eqnarray*}
  and hence
  
\begin{eqnarray}
\label{firsty}
    \begin{minipage}[h]{0.21\linewidth}
         \vspace{-0pt}
         \scalebox{0.35}{\includegraphics{tet_identity}}
     \put(-69,-10){$2n$}
          \put(-1,-10){$2n$}
         \put(-26,50){$2n$}
         \put(-59,25){$2n$}
         \put(-16,25){$2n$}
          \put(-36,0){$2n$}
                    \end{minipage}&=&\frac{Tet\left[ 
\begin{array}{ccc}
2n & 2n & 2n \\ 
2n & 2n & 2n%
\end{array}%
\right]}{\Theta(2n,2n,2n)}
   \begin{minipage}[h]{0.16\linewidth}
        \vspace{-0pt}
        \scalebox{0.25}{\includegraphics{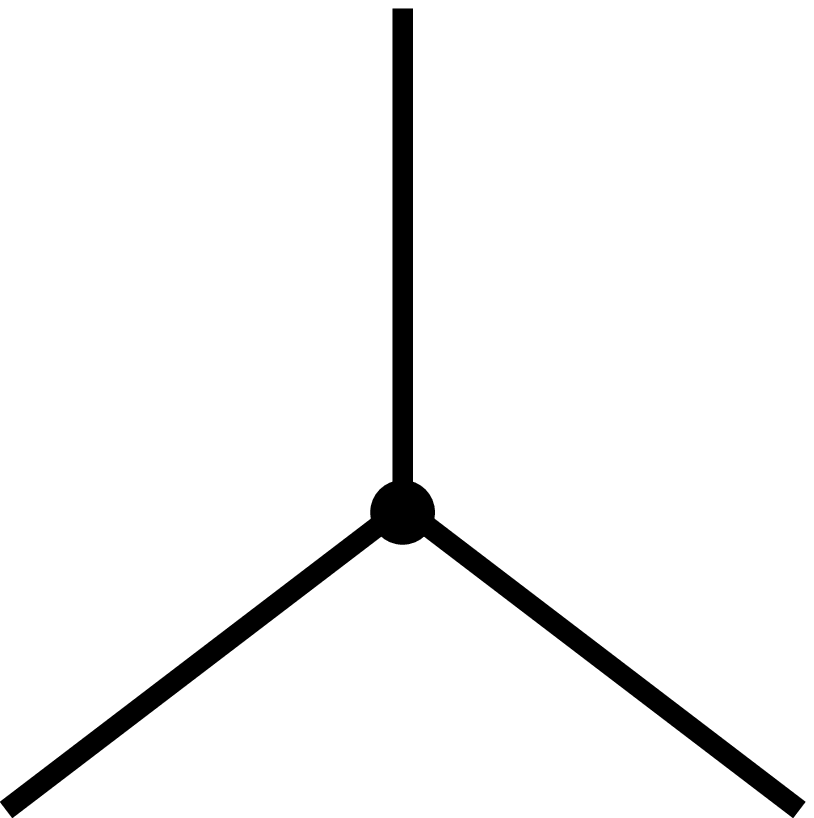}}
         \put(-60,-10){$2n$}
          \put(-1,-10){$2n$}
         \put(-26,50){$2n$}
           \end{minipage}
  \end{eqnarray}
As before we only have to show:
\begin{equation*}
\frac{Tet\left[ 
\begin{array}{ccc}
2n & 2n & 2n \\ 
2n & 2n & 2n%
\end{array}%
\right]}{\Theta(2n,2n,2n)}\doteq_n \Lambda (q)
\end{equation*}
To this end, note first that the bubble expansion formula implies 
\begin{equation*}
\label{one}
\Theta(2n,2n,2n)=
\left\lceil 
\begin{array}{cc}
n & n \\ 
n & n%
\end{array}%
\right\rceil _{0}\Delta_{2n}=q^{-
 n/2}
\frac{(q; q)^3_n(q;q)_{3n+1}}{(q;q)_{2 n}^2 (q ;q)_{2 n + 1}}[2n+1].
\end{equation*}
and hence Theorem \ref{thm2} implies
\begin{equation}
\label{one1}
\Theta(2n,2n,2n)\doteq_n\frac{(q;q)_n}{(1-q)}
\end{equation}
On the other hand one could use the Tetrahedron coefficient formula in \cite{Masbaum1} to obtain
\begin{equation}
\label{tet}
Tet\left[ 
\begin{array}{ccc}
2n & 2n & 2n \\ 
2n & 2n & 2n%
\end{array}%
\right]=\frac{([n]!)^{12}}{([2n]!)^6}\sum\limits_{i=3n}^{4n} 
 \frac{(-1)^i[i+1]!}{([4n - i]^!)^3([i - 3n]!)^4}.
\end{equation}
Using the identity
\begin{equation*}
\prod\limits_{i=0}^{j}[n-i]=q^{(2 + 3 j + j^2 - 2 n - 2 j n)/4} (1 - q)^{-1 - j}\frac{(q;q)_n}{(q;q)_{n-j-1}}
\end{equation*}
the equation (\ref{tet}) can be written as
\begin{equation*}
Tet\left[ 
\begin{array}{ccc}
2n & 2n & 2n \\ 
2n & 2n & 2n%
\end{array}%
\right]=\frac{q^{3n^2}(q;q)^{12}_n}{(q;q)^6_{2n}}\sum\limits_{i=3n}^{4n} 
 \frac{(-1)^{1 + 3 n + i}q^{-i/2 + 3i^2/2 - 12 i n + 21 n^2}(q,q)_{i+1}}{(1-q)(q,q)_{i-3n}^4(q,q)_{4n-i}^3}
\end{equation*}
One can simplify the previous equation to obtain
\begin{equation*}
Tet\left[ 
\begin{array}{ccc}
2n & 2n & 2n \\ 
2n & 2n & 2n%
\end{array}%
\right]=\frac{q^{-2n}(q;q)^{12}_n}{(q;q)^6_{2n}}\sum\limits_{i=0}^{n} 
 \frac{(-1)^{i}q^{(i +3i^2)/2}(q,q)_{4n-i}}{(1-q)(q,q)_{n-i}^4(q,q)_{i}^3}
\end{equation*}
Using the same techniques we used in Lemma \ref{lma} we can write
\begin{equation}
\label{final}
Tet\left[ 
\begin{array}{ccc}
2n & 2n & 2n \\ 
2n & 2n & 2n%
\end{array}%
\right]\doteq_n(q;q)^3_n\sum\limits_{i=0}^{\infty}\frac{(-1)^{i} q^{(i +3i^2)/2}}{(1-q)(q;q)^3_i}
\end{equation}
Putting (\ref{one1}) and (\ref{final}) in (\ref{firsty}) yield the result.
\end{proof}
\begin{remark}
The tail of the tetrahedron whose edges all colored $2n$ is computed in previous theorem, see equation (\ref{final}). The tail of this element can be seen to be $\Lambda (q)(q^2;q)_{n}$. This tail was also computed by  Garoufalidis and Lˆe in \cite{klb}.
\end{remark}
\begin{example}
All edges in the following graphs are colored $2n$.
\begin{eqnarray*}
    \begin{minipage}[h]{0.13\linewidth}
         \vspace{-0pt}
         \scalebox{0.24}{\includegraphics{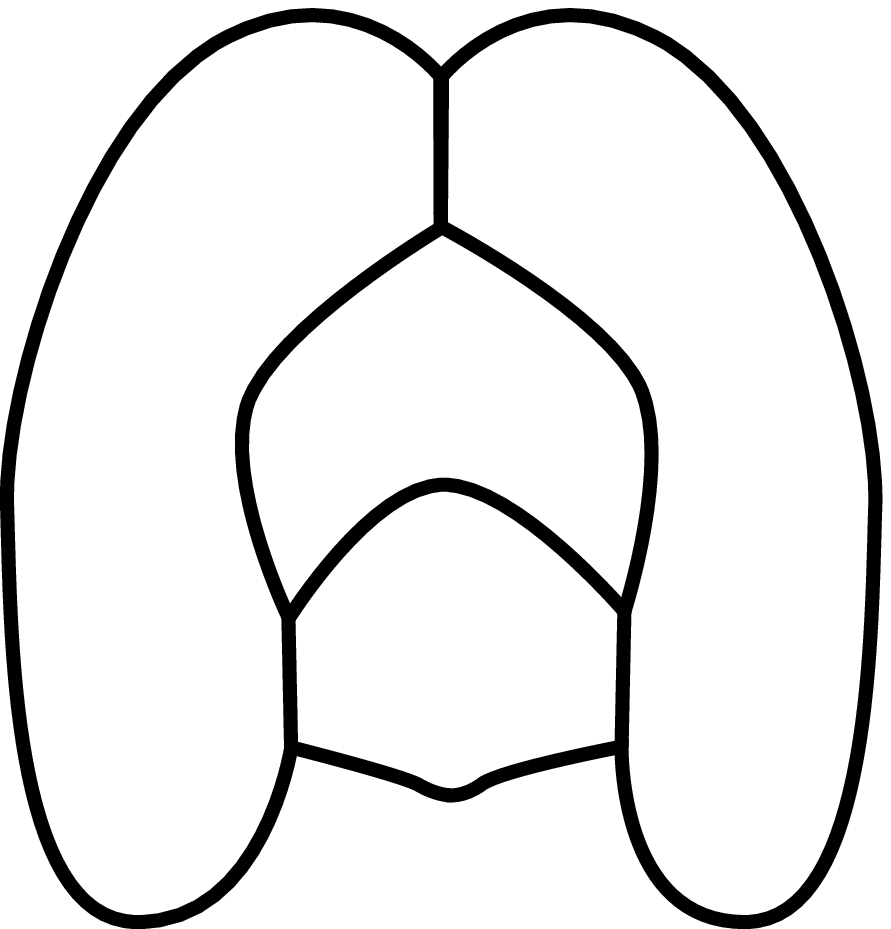}}
          \end{minipage}&\doteq_n&\Lambda(q)
   \begin{minipage}[h]{0.13\linewidth}
        \vspace{-0pt}
        \scalebox{0.24}{\includegraphics{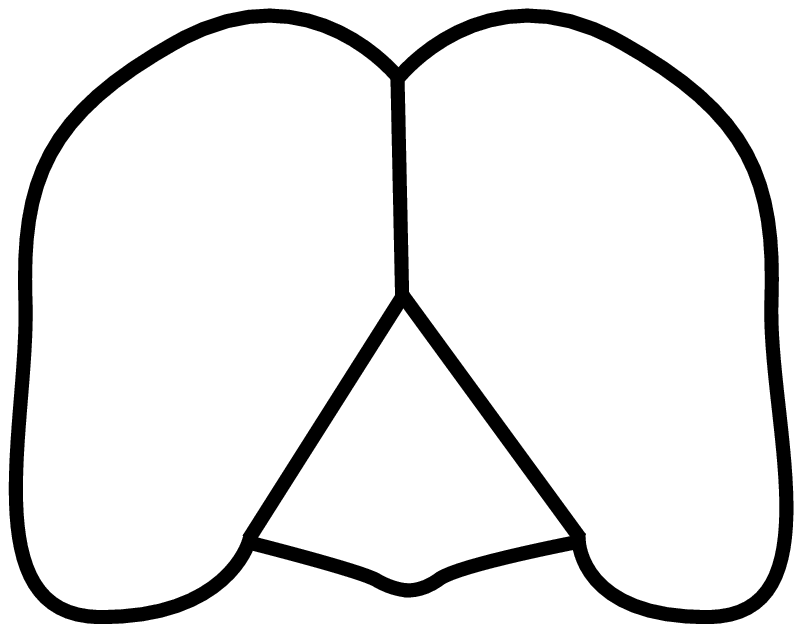}}
           \end{minipage}\\&\doteq_n&\Lambda^2 (q)(q^2;q)_{\infty}
                  \end{eqnarray*}
                  The first equation follows from the previous theorem and the second one follows from Proposition \ref{eee}.
\end{example}
\begin{example}
\label{inadequate}
All arcs in the following skein elements are colored $n$.
\begin{eqnarray*}
    \begin{minipage}[h]{0.17\linewidth}
         \vspace{-0pt}
         \scalebox{0.2}{\includegraphics{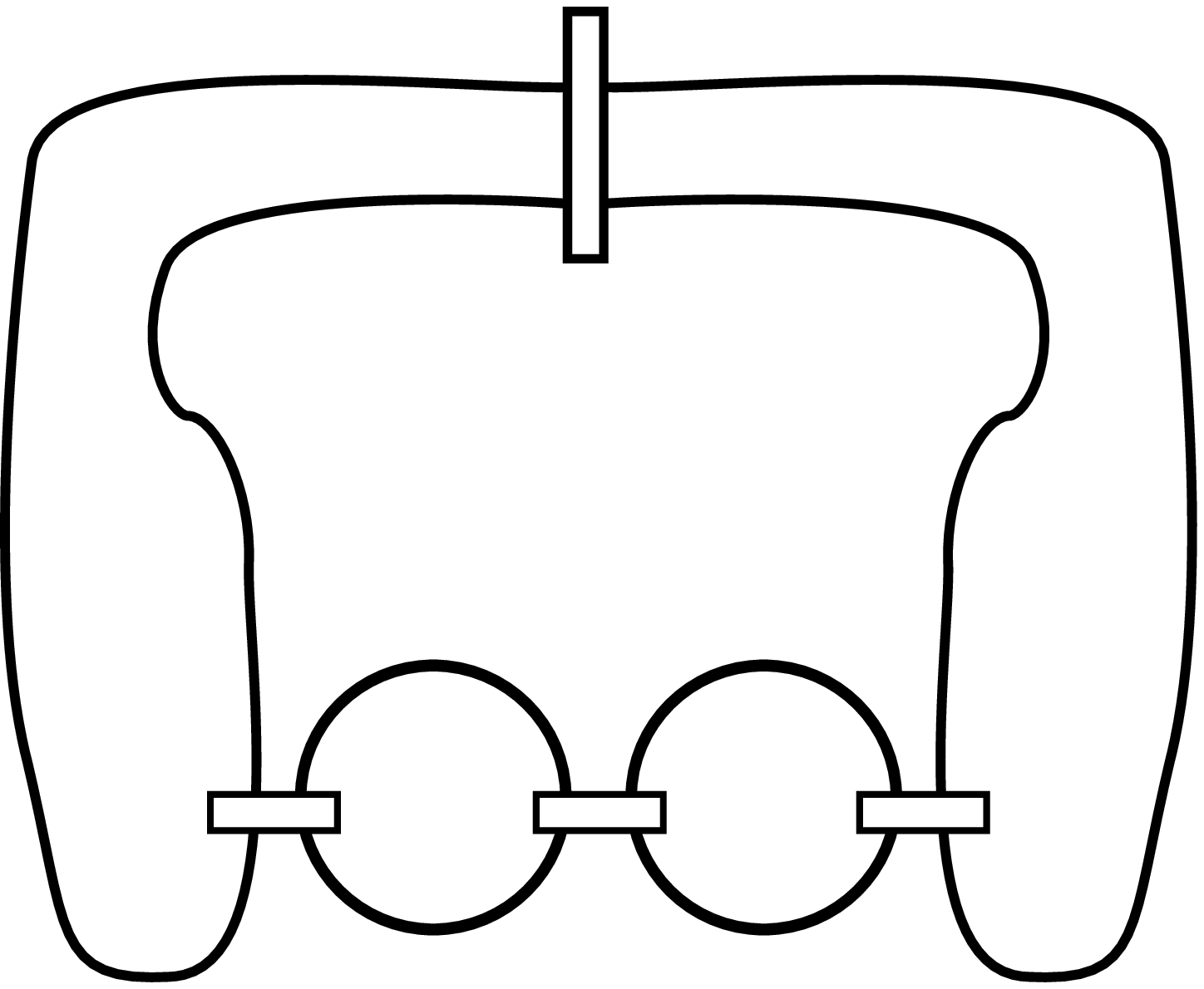}}
          \end{minipage}&\doteq_n&(q,q)_{\infty}
   \begin{minipage}[h]{0.2\linewidth}
        \vspace{-0pt}
        \scalebox{0.2}{\includegraphics{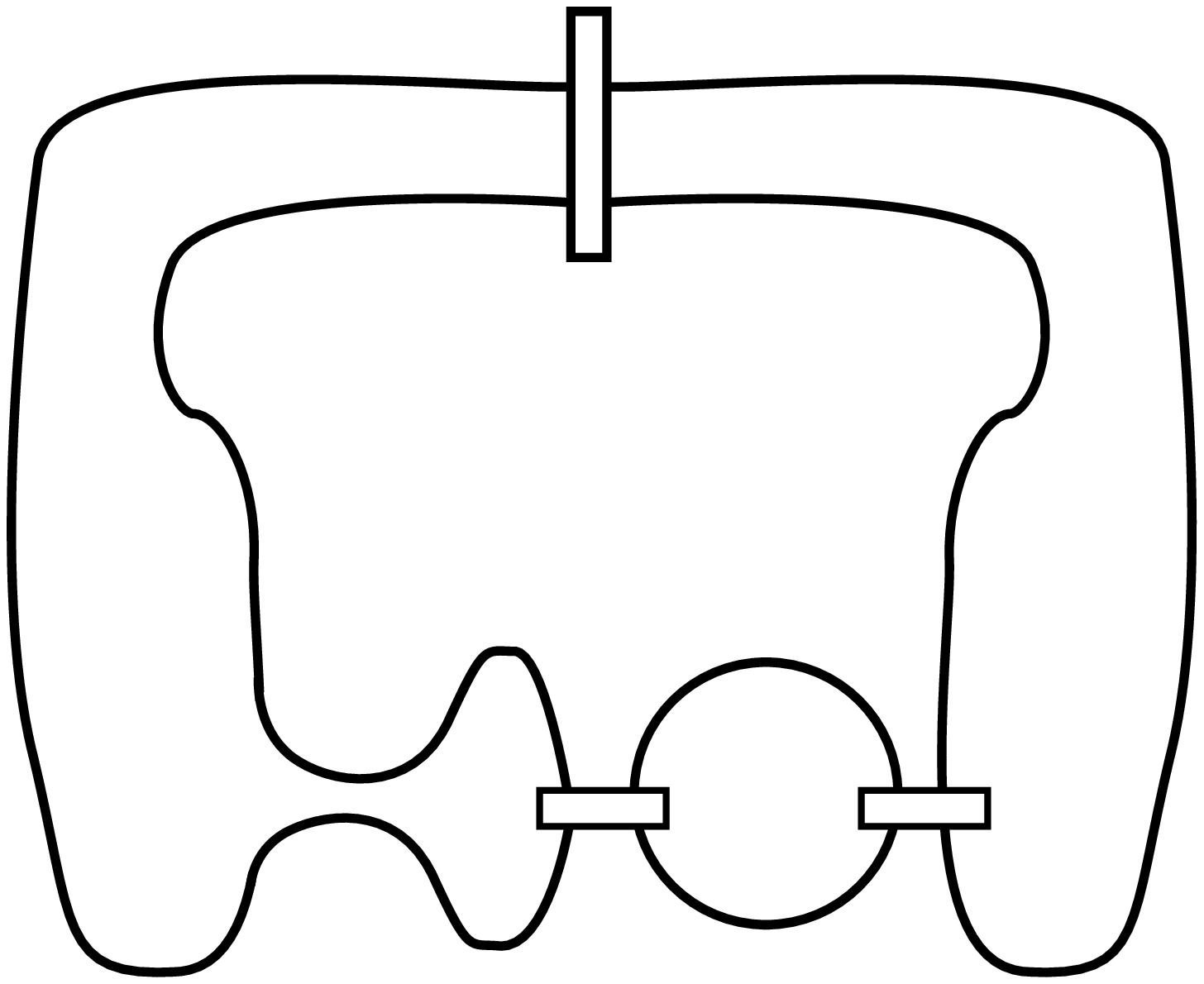}}
           \end{minipage}\\&\doteq_n&(q,q)^2_{\infty}
           \begin{minipage}[h]{0.1\linewidth}
         \vspace{-0pt}
         \scalebox{0.2}{\includegraphics{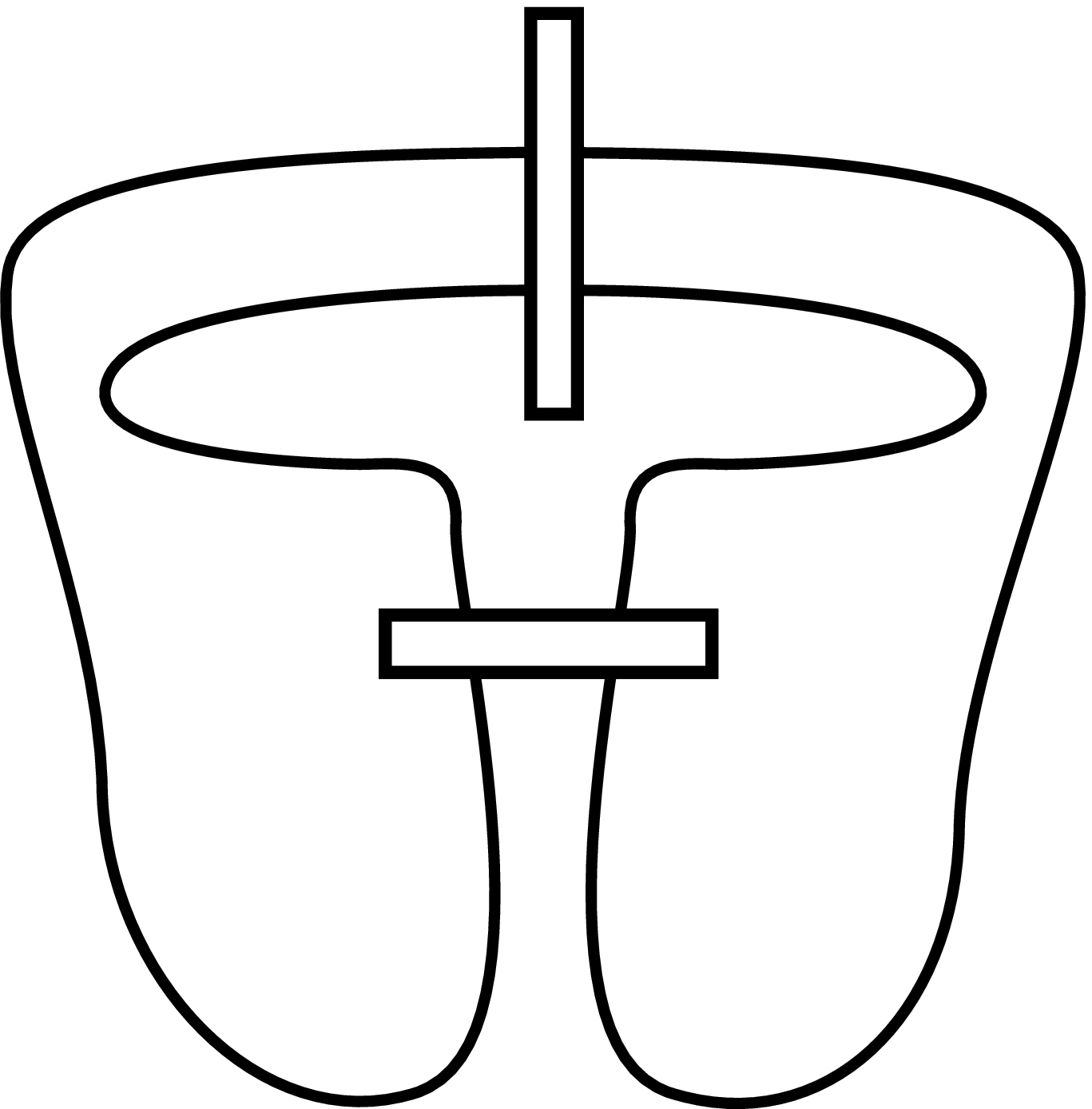}}
        \end{minipage}
        \end{eqnarray*}
The first and the second equations follow from Proposition \ref{thm1}. Observe that Proposition \ref{thm1} also implies
\begin{eqnarray*}
\begin{minipage}[h]{0.22\linewidth}
         \vspace{-0pt}
        $Tet\left[ 
\begin{array}{ccc}
2n & n & n \\ 
2n & n & n%
\end{array}\right]$\end{minipage}=
    \begin{minipage}[h]{0.17\linewidth}
         \vspace{-0pt}
         \scalebox{0.2}{\includegraphics{tet}}
          \put(-26,50){$n$}
         \put(-60,50){$n$}
           \put(-53,20){$n$}
           \put(-32,20){$n$}
           \end{minipage}&\doteq_n&(q,q)_n
   \begin{minipage}[h]{0.2\linewidth}
        \vspace{-0pt}
        \scalebox{0.2}{\includegraphics{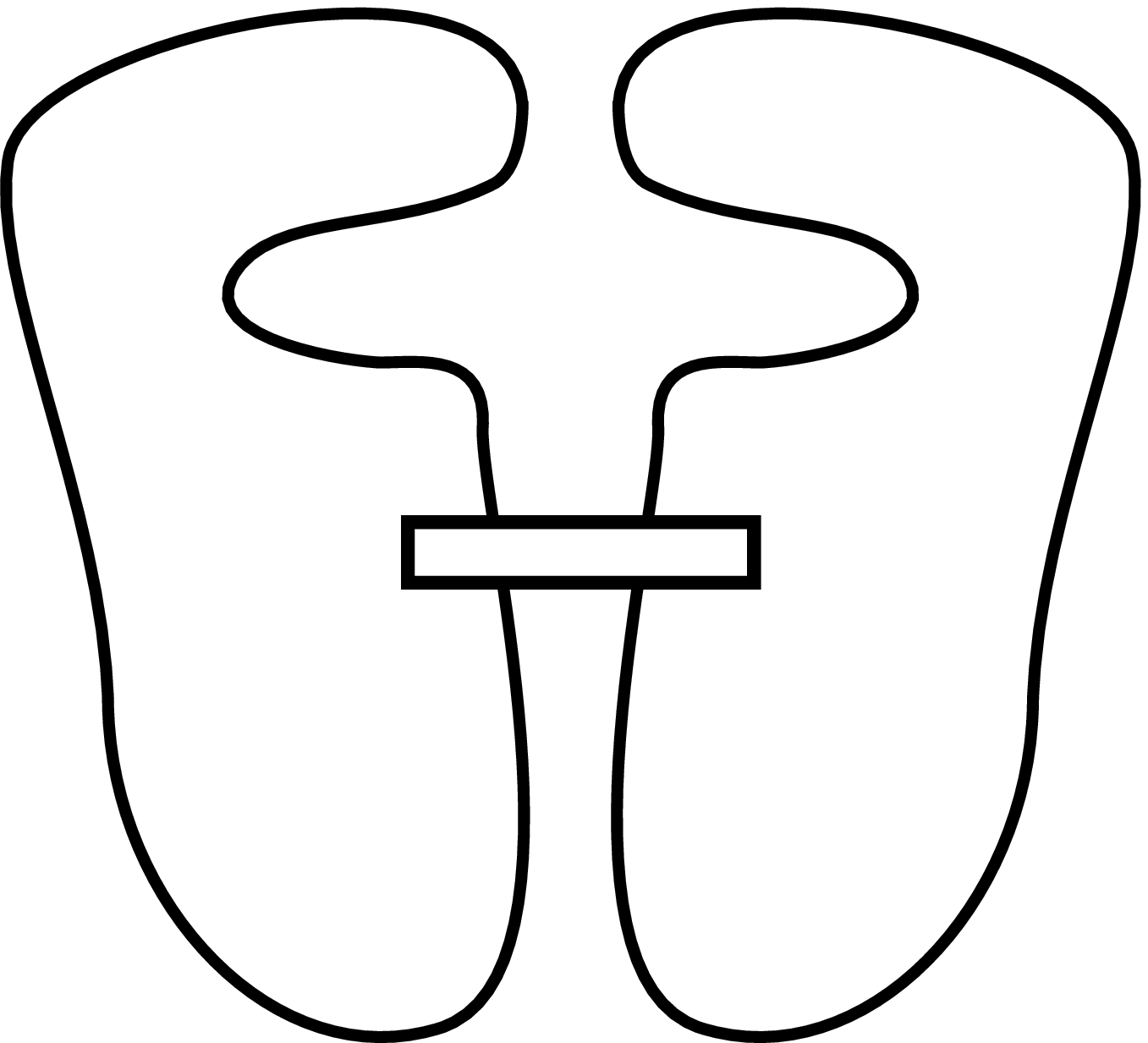}}
      \put(-87,50){$n$}
             \put(5,50){$n$}
           \end{minipage}\doteq_n\frac{(q,q)_n}{1-q}\doteq_n
           \begin{minipage}[h]{0.1\linewidth}
         \vspace{-0pt}
        $(q^2;q)_{n}.$\end{minipage}
  \end{eqnarray*}
  Hence, 
  \begin{eqnarray*}
    \begin{minipage}[h]{0.17\linewidth}
         \vspace{-0pt}
         \scalebox{0.2}{\includegraphics{n_bubble}}
          \end{minipage}&\doteq_n&(q^2,q)_{\infty}(q,q)^2_{\infty}.
        \end{eqnarray*}
 Similarly, one can compute
 \begin{eqnarray*}
    \begin{minipage}[h]{0.37\linewidth}
         \vspace{-0pt}
         \scalebox{0.2}{\includegraphics{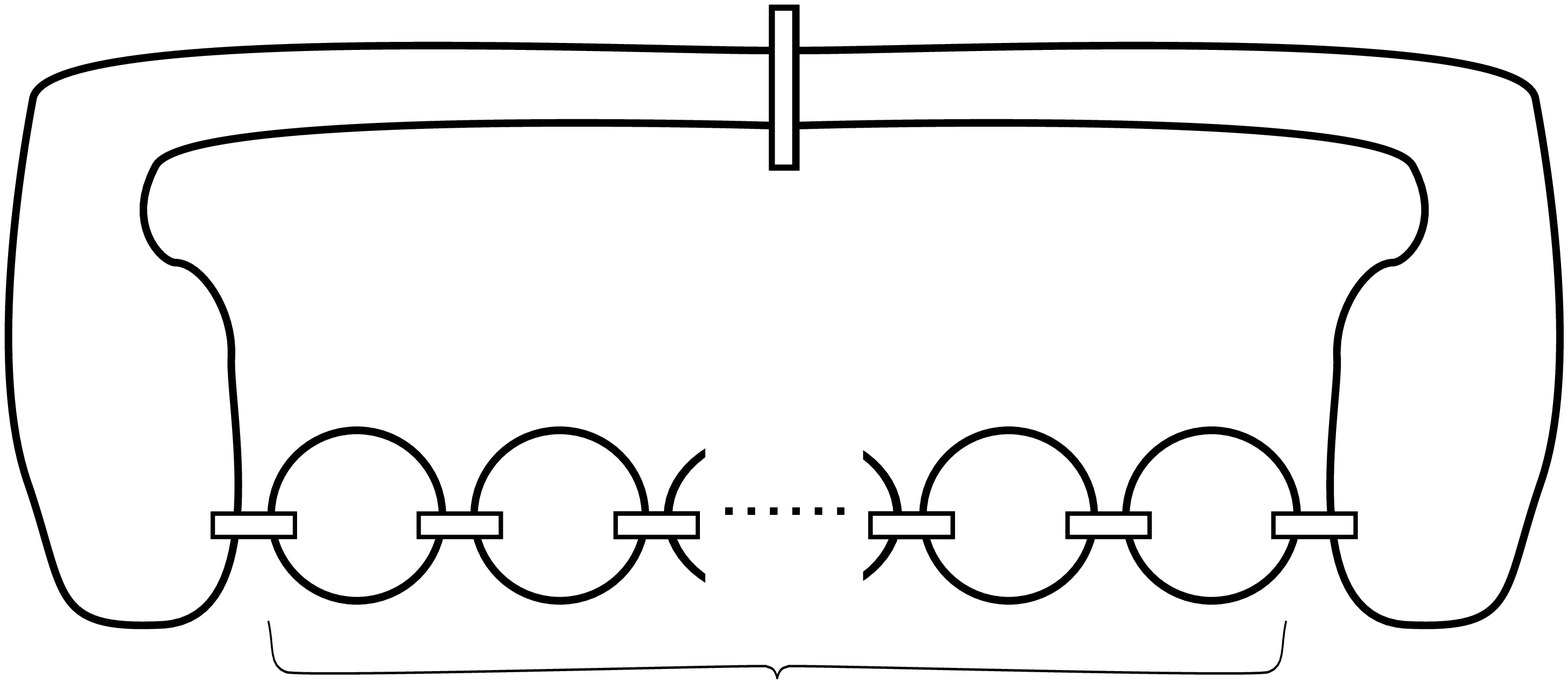}}
         \put(-96,-13){$m$ bubbles}
          \end{minipage}\doteq_n(q^2,q)_{\infty}(q,q)^m_{\infty}.
        \end{eqnarray*}
Note that the skein elements in this example are all inadequate. In particular the skein elements $Tet\left[ 
\begin{array}{ccc}
2n & n & n \\ 
2n & n & n%
\end{array}\right]$ is inadequate. 
\end{example}
\section{Tail multiplication structures on quantum spin networks}
\label{section5}
In \cite{Cody2} C. Armond and O. Dasbach defined a product structure on the tail of the color Jones polynomial. In this section we will define a few product structures on the tail of trivalent graphs in $\mathcal{S}(S^2)$ using similar techniques to the ones in \cite{Cody2}. Let $\Gamma_1$ and $\Gamma_2$ be trivalent graphs in $\mathcal{S}(S^2)$. Suppose that each of $\Gamma_1$ and $\Gamma_2$ contains the trivalent graph $\tau_{2n,2n,2n}$ as in Figure \ref{graph}.
  \begin{figure}[H]
  \centering
    {\includegraphics[scale=0.27]{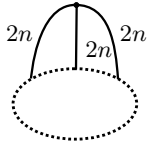}
    \footnotesize{
        \put(-50,38){$2n$}
        \put(-7,38){$2n$}
        \put(-20,32){$2n$}
        }
    \caption{The graph $\Gamma$ with a trivalent graph $\tau_{2n,2n,2n}$}
    \label{graph}
 }
\end{figure} 
Define the map
\begin{equation*}
[,]_1:\mathcal{S}(S^2)\times\mathcal{S}(S^2)\longrightarrow\mathcal{S}(S^2)
\end{equation*}
via the wiring map shown below.
  \begin{figure}[H]
  \centering
    {\includegraphics[scale=0.27]{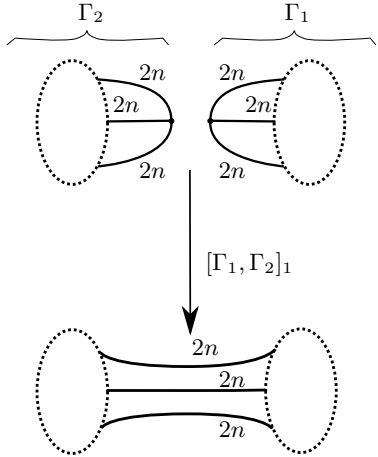}
    \footnotesize{
      \put(-60,142){$2n$}
        \put(-60,105){$2n$}
        \put(-50,130){$2n$}
         \put(-90,142){$2n$}
        \put(-90,105){$2n$}
        \put(-100,130){$2n$}
      \put(-35,165){$\Gamma_1$}
     \put(-113,165){$\Gamma_2$}
       \put(-65,70){$[\Gamma_1,\Gamma_2]_1$}
            \put(-70,38){$2n$}
        \put(-60,26){$2n$}
        \put(-60,7){$2n$}
        }
    \caption{The product $[\Gamma_1,\Gamma_2]_1$}
    \label{thisisit}
 }
\end{figure}
The proof of the following theorem is analogous to the proof of Theorem 5.1 in \cite{Cody2}.
\begin{theorem}
\label{codyy}
Let $\Gamma_1$ and $\Gamma_2$ as defined above.  Suppose further that  $T_{\Gamma_1}$ and $T_{\Gamma_2}$ exist. Then
\begin{equation*}
[\Gamma_1,\Gamma_2]_1\doteq_n\frac{1}{(q^2,q)_n}T_{\Gamma_1}T_{\Gamma_2}
\end{equation*}
\end{theorem}
\begin{proof}
If you regard the element $\tau_{2n,2n,2n}$ as a map of the outside, then the fact that the space $T_{2n,2n,2n}$ is one dimensional generated by the the graph $\tau_{2n,2n,2n}$ allows us to write
\begin{equation*}
\Gamma_i=f_i(q)\Theta(2n,2n,2n),
\end{equation*}
where $f_i(q) \in \mathbb{Q}(q)$ for $i=1,2$.
On the other hand one can also use the same fact to write
\begin{equation*}
[\hat{\Gamma}_1,\hat{\Gamma}_2]_1=f_1(q)f_2(q)\Theta(2n,2n,2n).
\end{equation*}
By assumption we have
\begin{equation*}
T_{\hat{\Gamma}_i}\doteq f_i(q)\Theta(2n,2n,2n)
\end{equation*}
for $i=1,2$. Hence

\begin{eqnarray*}
[\hat{\Gamma}_1,\hat{\Gamma}_2]_1&\doteq_n&T_{\hat{\Gamma}_1}\frac{T_{\hat{\Gamma}_2}}{\Theta(2n,2n,2n)}\\&\doteq_n&\frac{1}{(q^2,q)_n}T_{\hat{\Gamma}_1}T_{\hat{\Gamma}_2}
\end{eqnarray*}
\end{proof}
Similarly, suppose that $\Upsilon_1$ and $\Upsilon_2$ are trivalent graphs in $\mathcal{S}(S^2)$ and each of them contains the idempotent $f^{(2n)}$ as in Figure \ref{graph22}. Suppose further that $\Xi_1$ and $\Xi_2$ are trivalent graphs in $\mathcal{S}(S^2)$ and each of them contains the idempotent $f^{(n)}$ as shown in Figure \ref{graph22} below.
\begin{figure}[H]
 \subfigure{\centering
    {\includegraphics[scale=0.27]{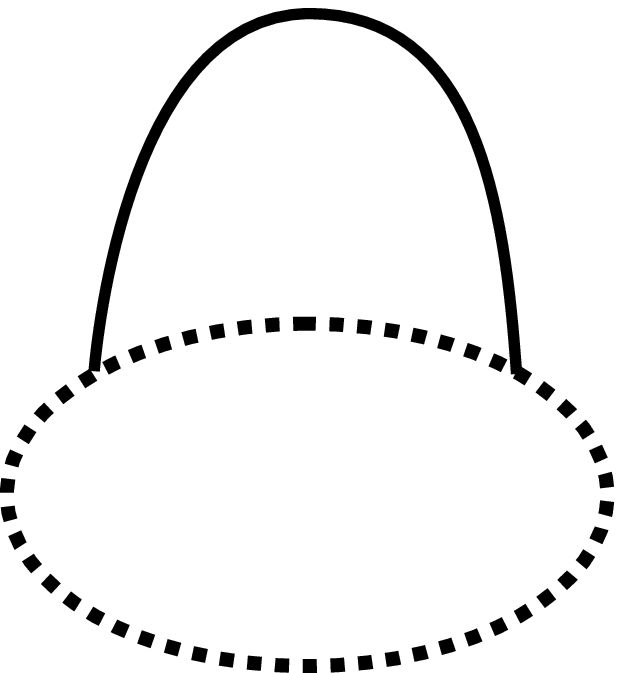}
    \footnotesize{
        \put(-50,38){$2n$}
        }}}
        \subfigure{
        \vspace{50pt}}
         \subfigure{
        \vspace{50pt}}
         \subfigure{
        \vspace{50pt}}
         \subfigure{
        \vspace{50pt}}
         \subfigure{
        \vspace{50pt}}
         \subfigure{
        \vspace{50pt}}
         \subfigure{
        \vspace{50pt}}
         \subfigure{
        \vspace{50pt}}
 \subfigure{
   \centering
    {\includegraphics[scale=0.27]{graph_2}
    \footnotesize{
        \put(-50,38){$n$}
        }}
   }
   \caption{The graph $\Upsilon$ is on the left and the graph $\Xi$ is on the right.}
   \label{graph22}
\end{figure}
Define the maps
\begin{equation*}
[,]_i:\mathcal{S}(S^2)\times\mathcal{S}(S^2)\longrightarrow\mathcal{S}(S^2)
\end{equation*}
for $i=2,3$ as shown below.
\begin{figure}[H]
 \subfigure{\centering
    {\includegraphics[scale=0.27]{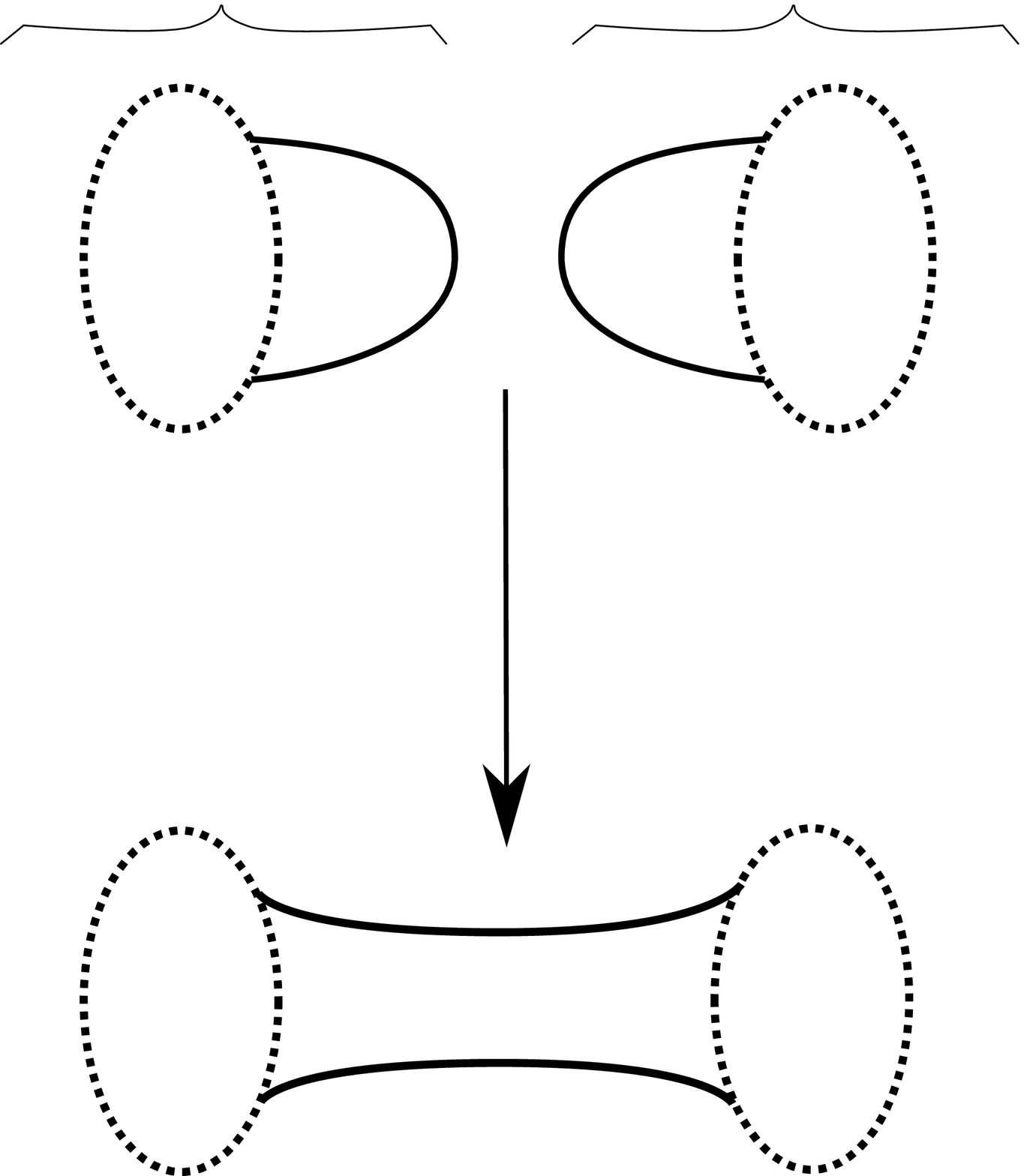}
    \footnotesize{
      \put(-60,142){$2n$}
         \put(-90,142){$2n$}
      \put(-35,165){$\Upsilon_1$}
     \put(-113,165){$\Upsilon_2$}
       \put(-65,70){$[\Upsilon_1,\Upsilon_2]_2$}
            \put(-70,38){$2n$}
        \put(-70,7){$2n$}
        }}}
        \subfigure{
        \vspace{50pt}}
         \subfigure{
        \vspace{50pt}}
 \subfigure{
   \centering
    {\includegraphics[scale=0.27]{mult_1}
    \footnotesize{
      \put(-60,142){$n$}
         \put(-90,142){$n$}
      \put(-35,165){$\Xi_1$}
     \put(-113,165){$\Xi_2$}
       \put(-65,70){$[\Xi_1,\Xi_2]_3$}
            \put(-70,38){$n$}
        \put(-70,7){$n$}
        }}
   }
   \caption{The product $[\Upsilon_1,\Upsilon_2]_2$ is on the left and the product $[\Xi_1,\Xi_2]_3$ is on the right.}
\end{figure}
As before these maps induce multiplication structures on skein elements in $\mathcal{S}(S^2)$ in the following sense.
\begin{theorem}
Suppose that $\Upsilon_1$ and $\Upsilon_2$ are trivalent graphs in $\mathcal{S}(S^2)$ and suppose that each of them contains the projector $f^{(n)}$ or $f^{(2n)}$ as in Figure \ref{graph22}. Suppose further that  $T_{\Upsilon_1}$ and $T_{\Upsilon_2}$ exist. Then
\begin{equation*}
[\Upsilon_1,\Upsilon_2]_i\doteq_n (1-q)T_{\Upsilon_1}T_{\Upsilon_2}
\end{equation*}
for $i=2,3$.
\end{theorem}
\begin{proof}
The proof follows from the fact that space $T_{a,a}$ is one dimensional generated by $f^{(a)}$ and
\begin{equation*}
\frac{1}{\Delta_{n}}\doteq_n \frac{1}{\Delta_{2n}}\doteq_n 1-q.
\end{equation*}
The rest of the proof is identical to the proof of \ref{codyy}.
\end{proof}
\begin{remark}
Note that the previous two products are just connect sum of two skein elements. The reason we include them here is to show that all they can be obtained in the same way Armond and Dasbach obtained their product in \cite{Cody} and we obtained the product $[,]_1$ above.  
\end{remark}
\section{Applications}
\label{section6}
\subsection{The tail of the Colored Jones polynomial}
In \cite{Cody} C. Armond and O. Dasbach introduced the tail of the colored Jones polynomial. The existence of the tail of the colored Jones polynomial of an alternating links was conjectured by Dasbach and Lin \cite{Dasbach} and in \cite{Cody2} C. Armond proved that the tail of colored Jones polynomial of adequate links exists. Higher order stability of the coefficients of the colored Jones polynomial of alternating links is studied by Garoufalidis and Le in \cite{klb}. Explicit calculations were done on the knot table to determine the tail of colored Jones polynomial of alternating links in \cite{Cody}. The knot $8_5$ is the first knot on the knot table whose tail could not be determined by a direct application of techniques in \cite{Cody}. In \cite{Hajij} we use theorem \ref{cody thm} and the bubble expansion formula to compute the tail of the $8_5$ and we prove that it equals to:
\begin{equation*}
T_{8_5}(q)=(q^{2};q)_{\infty}(q;q)_{\infty}\sum\limits_{k=0}^{\infty}\frac{q^{k+k^{2}}}{%
(q;q)_{k}}(\sum\limits_{i=0}^{k}q^{(-2i(k-i))}\left[ 
\begin{array}{c}
k \\ 
i%
\end{array}%
\right] _{q}^{2})
\end{equation*}
Recently, Garoufalidis and Vuong gave an algorithm for computing the tail of any alternating link \cite{klb2}. In this section we apply the results we obtain in section \ref{section4} to study the tail of the color Jones polynomial.\\
\begin{remark}
\label{normalize}
When dealing with the tail of colored Jones polynomial of a link $L$ we usually compute the tail of normalized polynomial $\tilde{J}_{n,L}(q)/\Delta_n(q)$. We will adapt this convention in this section.
\end{remark}
  
We recall here the definition of the unreduced colored Jones polynomial. Let $L$ be a framed link in $S^3$. Decorate every component of $L$, according to its framing, by the $n^{th}$ Jones-Wenzl idempotent and consider this decorated framed link as an element of $\mathcal{S}(S^3)$. Up to a power of $\pm A$, that depends on the framing of $L$, the value of this element is the $n^{th}$ (unreduced) colored Jones polynomial $\tilde{J}_{n,L}(A)$. Recall from section \ref{3} that the skein elements $\{S_B^{(n)}(D)\}_{n\in \mathbb{N}}$ are obtained from all-$B$ smoothing Kauffman state of $D$. See Figure \ref{allB1}. It was proven in \cite{Cody} that tail of the unreduced colored Jones polynomial depends only on the sequence $\{S_B^{(n)}(D)\}_{n\in \mathbb{N}}$. We state this theorem here.
\begin{theorem}(C. Armond \cite{Cody2})
\label{cody thm}
Let $L$ be a link in $S^3$ and $D$ be a reduced alternating knot diagram of $L$. Then
\begin{equation*}
\tilde{J}_{n,L}(q)\doteq_{(n+1)}S_B^{(n)}(D).
\end{equation*} 
\end{theorem}
Let $D$ be a link diagram. The \textit{$A$-graph} $A(D)$ and the \textit{$B$-graph} $B(D)$ are two graphs associated to the all-$A$ smoothing and all-$B$ smoothing states of $D$. The set of vertices  of $A(D)$ is equal to the set of circles in $S_A(D)$. Moreover, an edge in the set of edges of $A(D)$ is obtained by joining two vertices of $A(D)$ for each crossing in $D$ between the corresponding circles.  We obtain the \textit{reduced $A$-graph} $A(D)^{\prime}$ by keeping the same set of vertices of $A(D)$ and replacing parallel edges by a single edge. See Figure \ref{thisisit} for an example. We define the $B$-graph $B(D)$ and reduced $B$-graph $B(D)^{\prime}$ similarly.
 
 \begin{figure}[H]
  \centering
    {\includegraphics[scale=0.06]{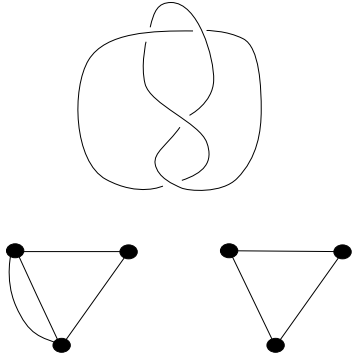}
    \caption{The knot $4_1$, its $A$-graph on the left and its reduced $A$-graph on the right}
    \label{thisisit}
 }
\end{figure}

Theorem \ref{cody thm} implies that the tail of the unreduced colored Jones polynomial of an alternating link only depends on the reduced $B$-graph:  
\begin{theorem}(C. Armond, O. Dasbach \cite{Cody})
\label{cody thm2}
Let $L_1$ and $L_2$ be two alternating links with alternating diagrams $D_1$ and $D_2$. If the graph $B(D_1)^{\prime}$ coincides with $B(D_2)^{\prime}$, then $T_{K_1}=T_{K_2}$.
\end{theorem}
Remark \ref{important} implies the following important result. 
\begin{theorem}
\label{newproduct}
    (1) Let $G$ be the graph shown on the right hand side of the following identity, then there exists a $q$-power series $A(q)$ series such that
     \begin{eqnarray}T\Bigg(
     \label{equationone1}
    \begin{minipage}[h]{0.18\linewidth}
         \vspace{-7pt}
         \scalebox{0.42}{\includegraphics{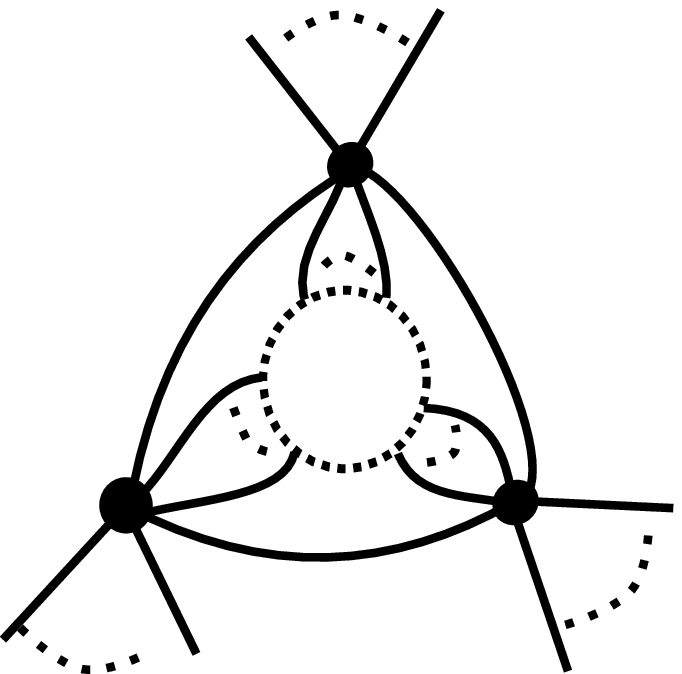}}
                   \put(-43,+32){\small{$G$}}
           \end{minipage}\Bigg)&\doteq_n&A(q)
   T\Bigg(\begin{minipage}[h]{0.18\linewidth}
        \vspace{0pt}
        \scalebox{0.42}{\includegraphics{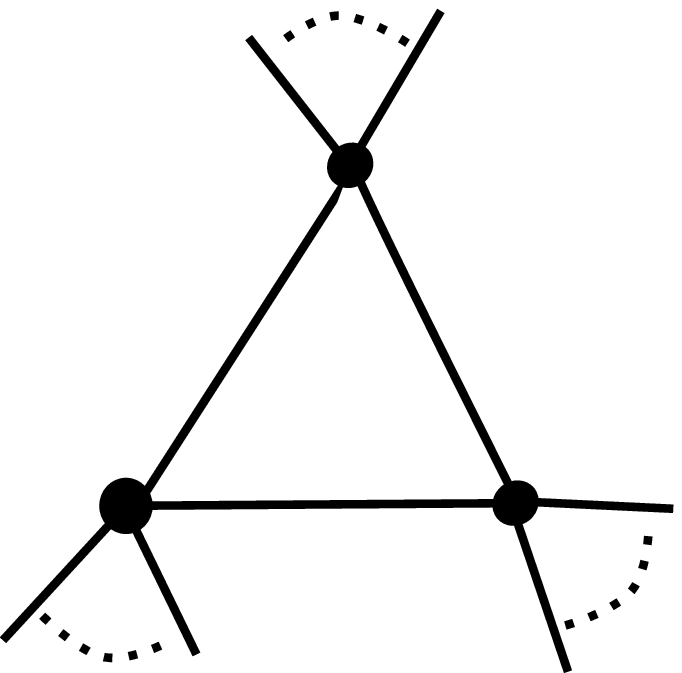}}
          \end{minipage}\Bigg)
            \end{eqnarray}
(2) (Armond and Dasbach \cite{Cody}) Let $G^{\prime}$ be the graph shown on the right hand side of the following identity, then there exists a $q$-power series $A^{\prime}(q)$ series such that
               \begin{eqnarray}
                  T\Bigg( \label{equationtwo2}
    \begin{minipage}[h]{0.17\linewidth}
         \vspace{-8pt}
         \scalebox{0.34}{\includegraphics{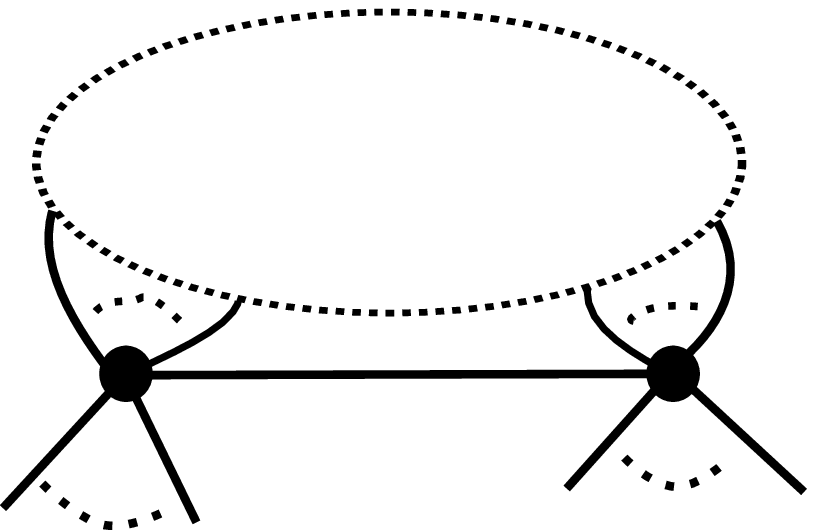}}
          \put(-43,+32){\small{$G^{\prime}$}}
           \end{minipage}\Bigg)&\doteq_n&A^{\prime}(q)T\Bigg(
   \begin{minipage}[h]{0.20\linewidth}
        \vspace{0pt}
        \scalebox{0.34}{\includegraphics{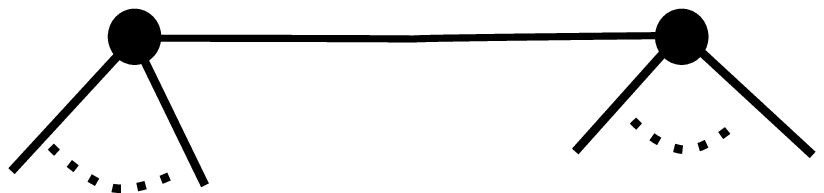}}
          \end{minipage}\Bigg)
            \end{eqnarray}
\end{theorem}
\begin{proof}
(1) Note that equation (\ref{equationone1}) is equivalent to 
 \begin{eqnarray}
 \label{onemoretime}
    \begin{minipage}[h]{0.19\linewidth}
         \vspace{0pt}
         \scalebox{0.33}{\includegraphics{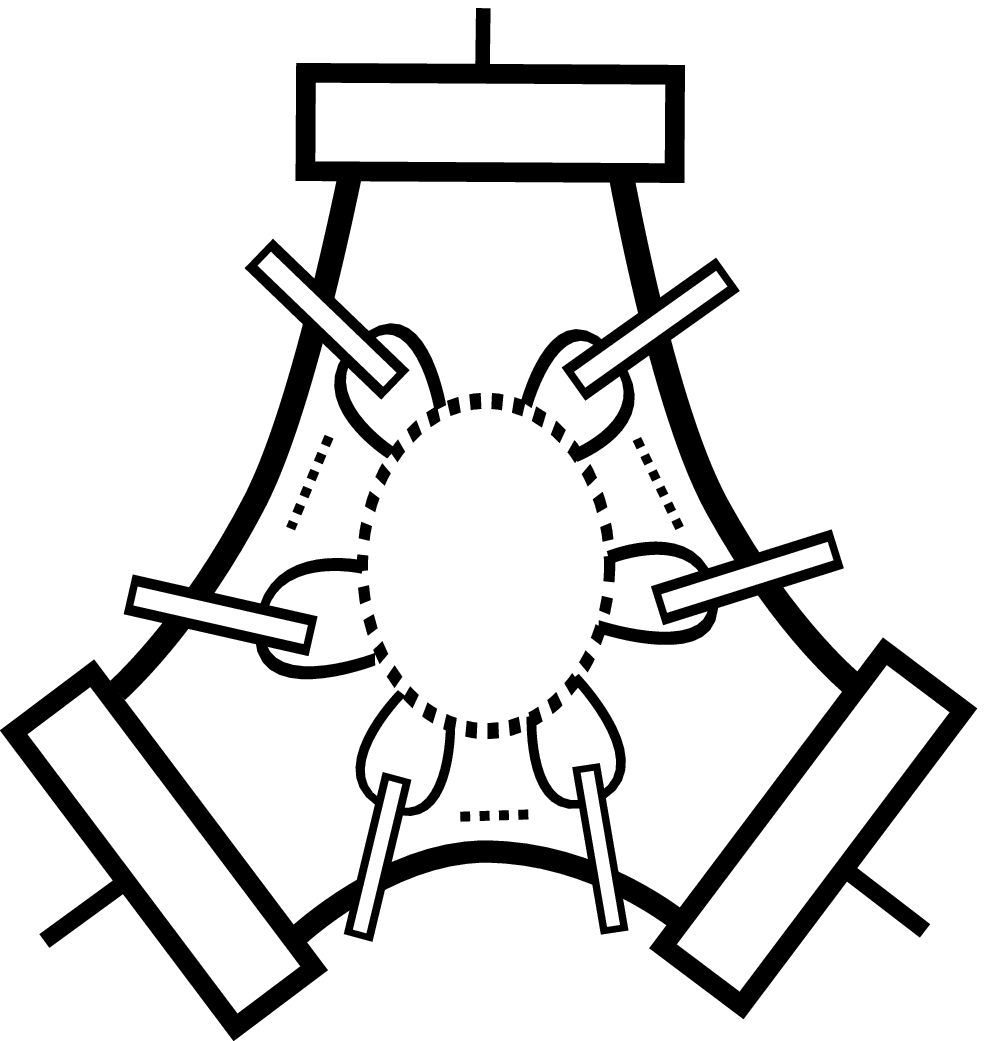}}
          \put(-49,+43){\small{$S$}}
            \put(-50,99){$2n$}
          \put(-10,-2){$2n$}
         \put(-90,-2){$2n$}
           \end{minipage}&\doteq_n&A(q)
   \begin{minipage}[h]{0.21\linewidth}
        \vspace{0pt}
        \scalebox{0.33}{\includegraphics{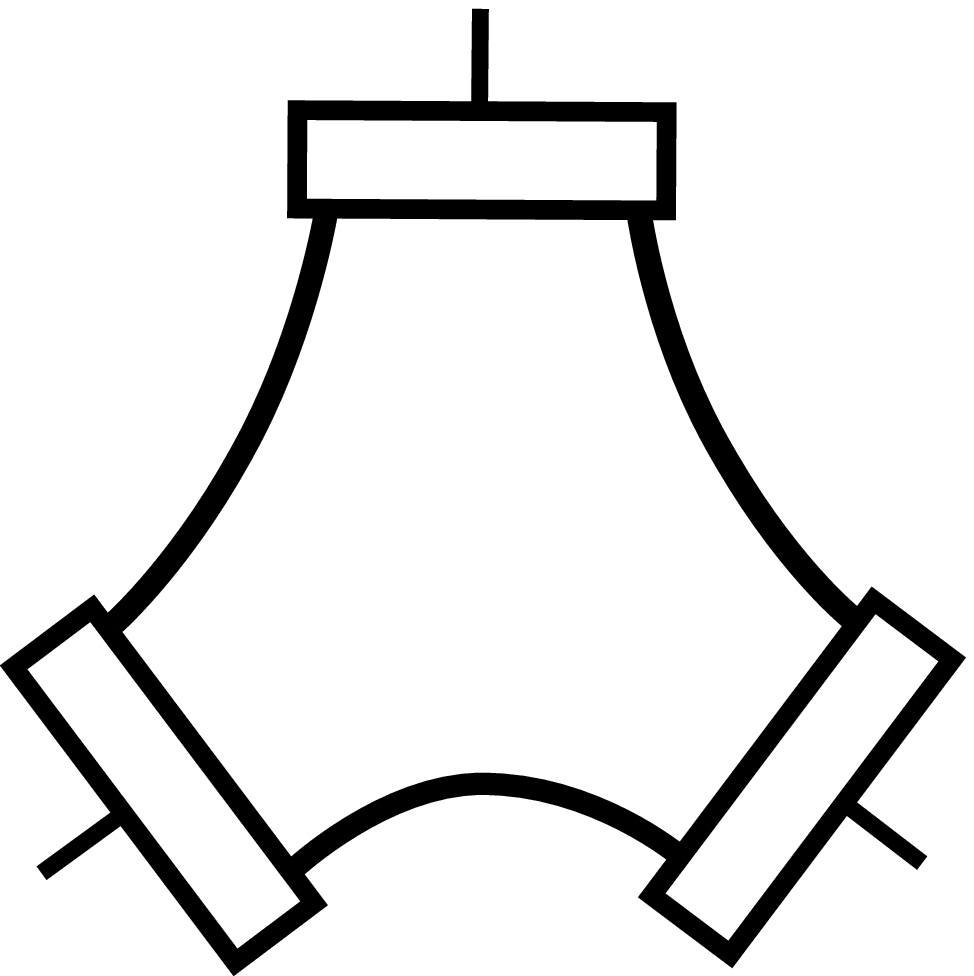}}
          \put(-50,95){$2n$}
          \put(-10,-2){$2n$}
         \put(-90,-2){$2n$}
          \end{minipage}
            \end{eqnarray}
where $S$ is the skein element obtained from $G$ by replacing every vertex by a circle colored $n$ and every edge by an idempotent that connects two circles. Write $\alpha^*_n$ to denote the skein element on the right hand side of (\ref{onemoretime}) then the result follows by noticing
\begin{eqnarray*}
 A(q)&\doteq_n& \alpha^*_n(\tau_{2n,2n,2n})/\tau^*_{2n,2n,2n}(\tau_{2n,2n,2n})\\&=& \alpha^*_n(\tau_{2n,2n,2n})/\theta(2n,2n,2n)
 \\&\doteq_n& \alpha^*_n(\tau_{2n,2n,2n})/(q^2;q)_{\infty}.
\end{eqnarray*}
(2)  Identity (\ref{equationtwo2}) is equivalent to 

 \begin{eqnarray}
    \begin{minipage}[h]{0.15\linewidth}
         \vspace{-10pt}
         \scalebox{0.27}{\includegraphics{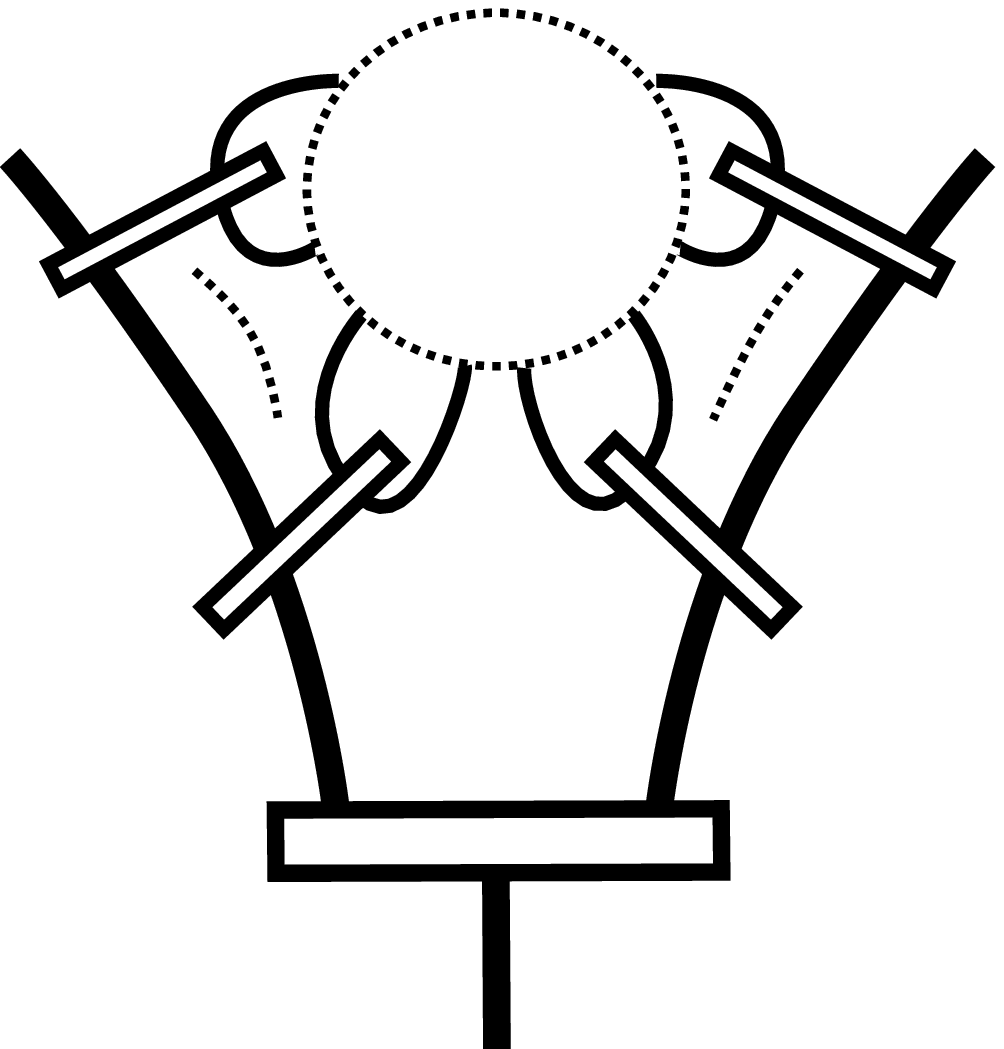}}
          \put(-42,+62){\small{$S$}}
           \end{minipage}&\doteq_n&B(q)
   \begin{minipage}[h]{0.21\linewidth}
        \vspace{0pt}
        \scalebox{0.27}{\includegraphics{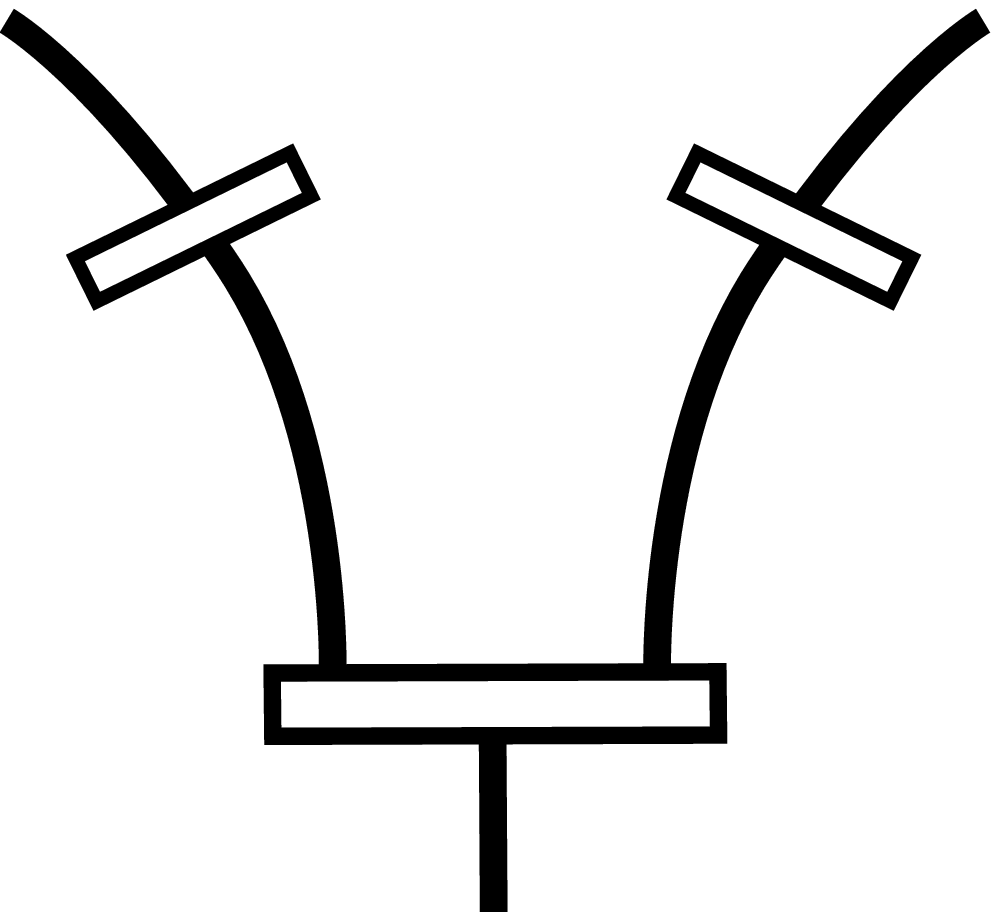}}
          \end{minipage}
            \end{eqnarray}
where $S$ is the skein element obtained from $G$ as explained in (1). Furthermore,  
\begin{eqnarray*}
 B(q)&\doteq_n& \beta^*_n(\tau_{n,n,2n})/\tau^*_{n,n,2n}(\tau_{2n,2n,2n})\\&=& \beta^*_n(\tau_{n,n,2n})/\theta(n,n,2n)
 \\&\doteq_n& \beta^*_n(\tau_{2n,2n,2n}).
\end{eqnarray*}
The result follows.
\end{proof}
As mentioned in the previous section, Armond and Dasbach \cite{Cody} showed that if $G_1$ and $G_2$ are reduced graphs then the product of the tails $T_{G_1}$ and $T_{G_2}$ is equal to the tail of the graph $G_1*G_2$ obtained from $G_1$ and $G_2$ by gluing one edge from $G_1$ and another edge from $G_2$. In other words the following identity holds
\begin{equation}
T_{G_1}T_{G_2}=T_{G_1*G_2}.
\end{equation}
Theorem \ref{newproduct} (2) merely a restatement of this result. On the other hand, Theorem \ref{newproduct} (1) implies immediately the following result.
\begin{corollary}
\label{mustafaproduct}
The tail of reduced graphs satisfies the following product: 
 \begin{eqnarray}
              T\Bigg( \begin{minipage}[h]{0.12\linewidth}
         \vspace{-6pt}
         \scalebox{0.38}{\includegraphics{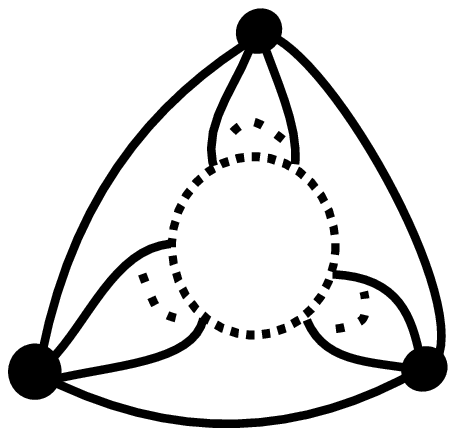}}
          \put(-26,+18){\small{$G_1$}}
           \end{minipage}\Bigg)
            T\Bigg( \begin{minipage}[h]{0.12\linewidth}
         \vspace{-6pt}
         \scalebox{0.38}{\includegraphics{graphid_12}}
          \put(-26,+18){\small{$G_2$}}
           \end{minipage}\Bigg)&\doteq_n&T\Bigg( \begin{minipage}[h]{0.10\linewidth}
        \vspace{0pt}
        \scalebox{0.36}{\includegraphics{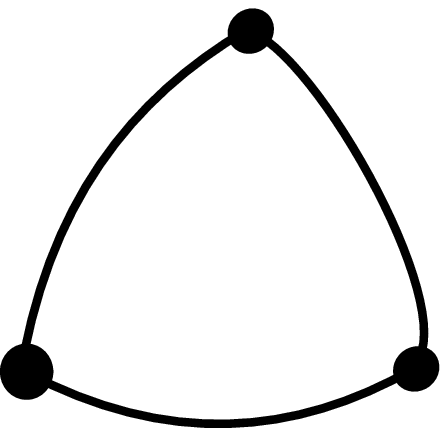}}
         \end{minipage}\Bigg)
  T\Bigg( \begin{minipage}[h]{0.21\linewidth}
        \vspace{0pt}
        \scalebox{0.36}{\includegraphics{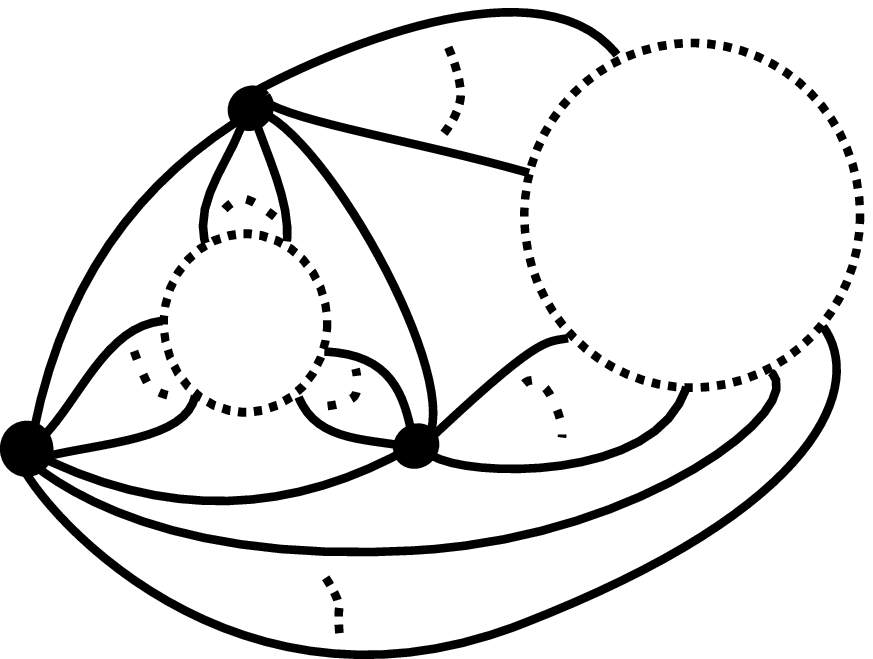}}
         \put(-20,+43){\small{$G_2$}}
         \put(-70,+32){\small{$G_1$}}
          \end{minipage}\Bigg)
            \end{eqnarray}
     \end{corollary}

Note that Theorem \ref{states} is a special case of (\ref{equationone1}) (1) and it can be stated as:
      \begin{eqnarray}
      \label{graa}T\Bigg(
    \begin{minipage}[h]{0.15\linewidth}
         \vspace{-6pt}
         \scalebox{0.30}{\includegraphics{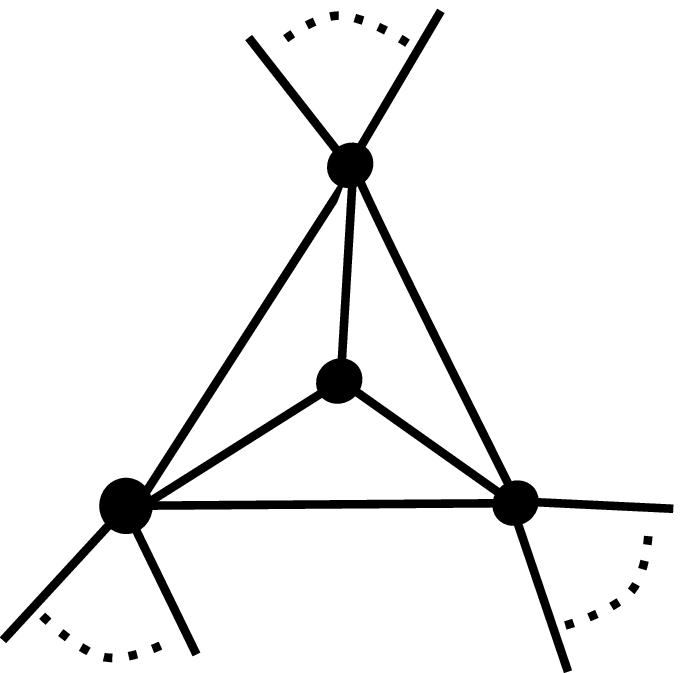}}
           \end{minipage}\Bigg)&\doteq_n&\Lambda(q)
   T\Bigg(\begin{minipage}[h]{0.14\linewidth}
        \vspace{-6pt}
        \scalebox{0.30}{\includegraphics{graphid_8}}
          \end{minipage}\Bigg)
            \end{eqnarray}
The following examples illustrate how one could apply the results obtained in section \ref{section4} to compute the tail of a reduced graph.
\begin{example}
\begin{eqnarray*}
    T\Bigg(\begin{minipage}[h]{0.16\linewidth}
         \vspace{-7pt}
         \scalebox{0.2}{\includegraphics{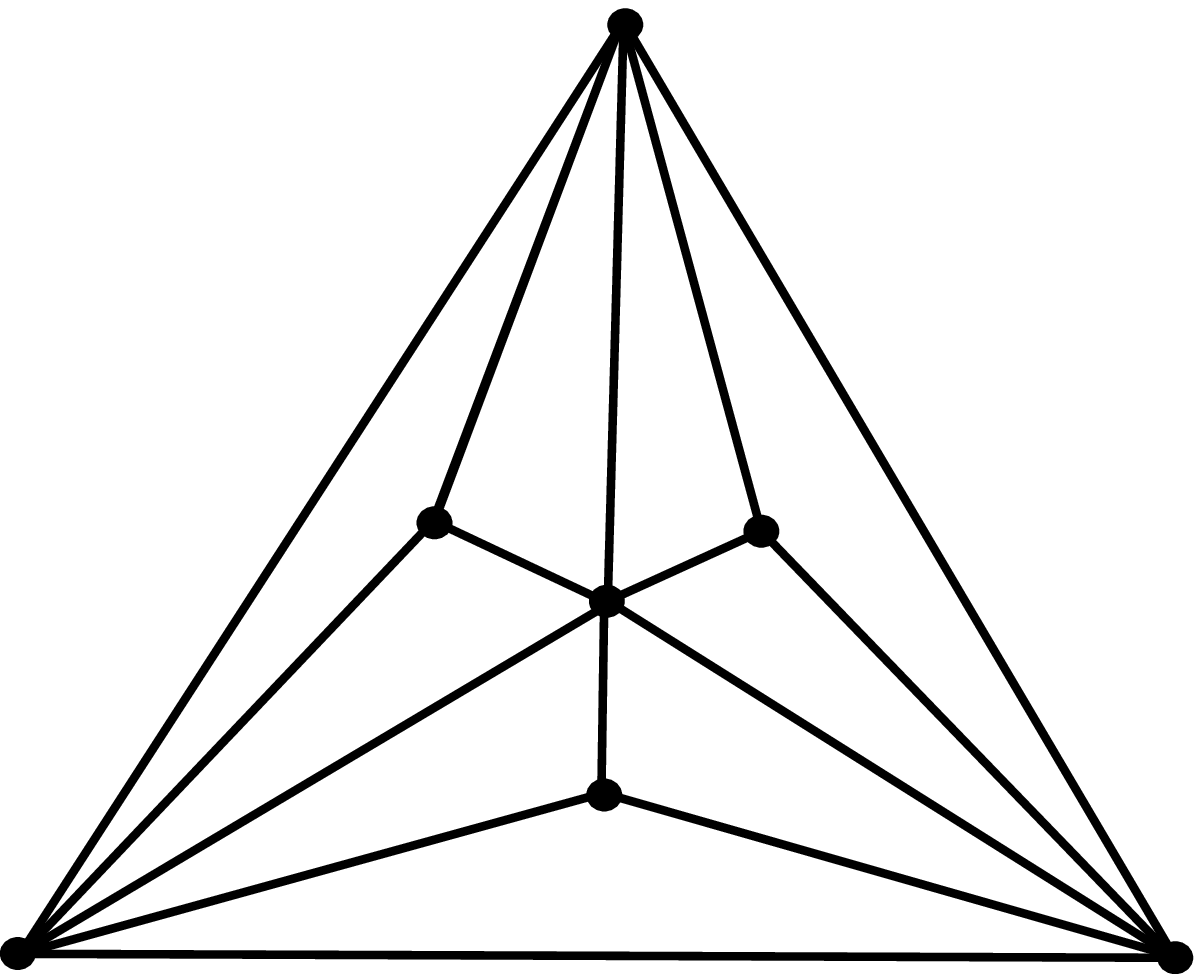}}
         \end{minipage}\Bigg)&\doteq_n&(\Lambda(q))^4
   T\Bigg(\begin{minipage}[h]{0.16\linewidth}
        \vspace{0pt}
        \scalebox{0.2}{\includegraphics{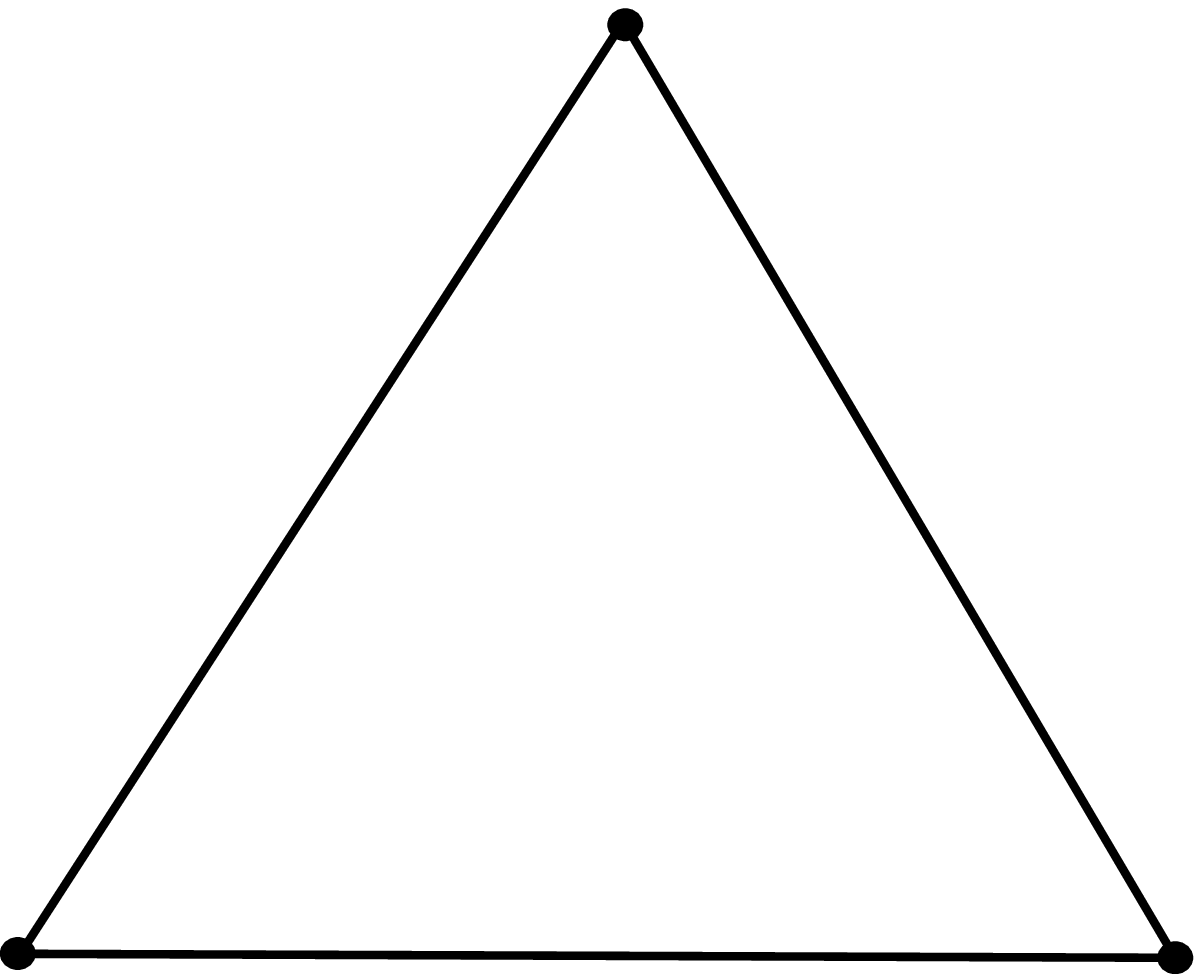}}
    
           \end{minipage}\Bigg)\doteq_n
           \begin{minipage}[h]{0.10\linewidth}
         \vspace{-7pt}
        $(\Lambda (q))^4(q;q)_{\infty}$.\end{minipage}. 
  \end{eqnarray*}Where in the first equality we used equation (\ref{graa}) and in the second equality we used the fact that the tail of a triangle is the same as the tail of $\Theta(2n,2n,2n)$ which is just $(q^2;q)_{\infty}$. Recall here that we normalize tail by dividing by $\Delta_n$. See remark \ref{normalize}.
\end{example}
\begin{example}
Let $G_m$ be the reduced graph in the Figure \ref{ba3sa}. 

\begin{figure}[H]
  \centering
   {\includegraphics[scale=0.14]{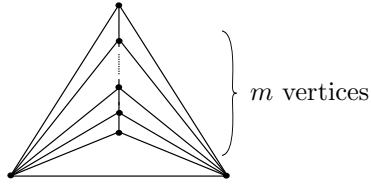}
    \put(5,30){\small{$m$ vertices}}
    \caption{The graph $G_m$}
  \label{ba3sa}}
\end{figure}
Then the tail of this graph can be computed as follows:
\begin{eqnarray*}
    T(G_m)&\doteq_n&(\Lambda(q))^m
   T\Bigg(\begin{minipage}[h]{0.16\linewidth}
        \vspace{0pt}
        \scalebox{0.2}{\includegraphics{bary1}}
    
           \end{minipage}\Bigg)\doteq_n
           \begin{minipage}[h]{0.10\linewidth}
         \vspace{-7pt}
        $(\Lambda (q))^m(q;q)_{\infty}$.\end{minipage}. 
  \end{eqnarray*}Here used again equation (\ref{graa}) in the first equality.
\end{example}
\begin{example}
\label{cccc}
Let $k\geq 1$ and $l\geq 0$. Let $G_{k,l}$ be the reduced graph in the Figure \ref{ba3sa2}. 

\begin{figure}[H]
  \centering
   {\includegraphics[scale=0.16]{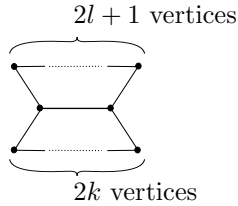}
    \put(-27,58){\small{$2l+1$ vertices}}
    \put(-27,-9){\small{$2k$ vertices}}
    \caption{The graph $G_{k,l}$}
  \label{ba3sa2}}
\end{figure}
Then using Theorem \ref{complicated111} one could see that
\begin{eqnarray*}
    T(G_{k,l})&\doteq_n&\Psi(q^{2k+1},q)
   T\Bigg(\begin{minipage}[h]{0.14\linewidth}
        \vspace{0pt}
        \scalebox{0.17}{\includegraphics{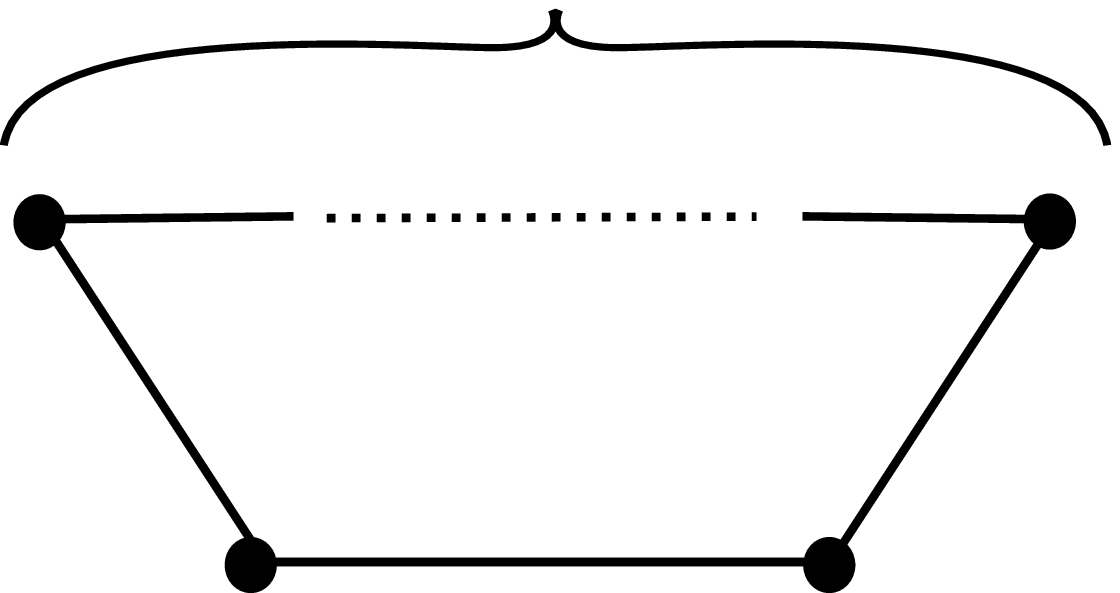}}
    	 \put(-40,33){\tiny{$2l+1$ vertices}}
           \end{minipage}\Bigg)\doteq_n
           \begin{minipage}[h]{0.10\linewidth}
         \vspace{-0pt}
        $\Psi(q^{2k+1},q)f(-q^{2l+2},q)$.\end{minipage} 
  \end{eqnarray*}
\end{example}
\begin{remark}
Example \ref{cccc} can be computed also using the techniques of Armond and Dasbach in \cite{Cody}.
\end{remark}

\subsection{The tail of the Colored Jones polynomial and Andrews-Gordon identities}
The fact that the tail of an alternating link $L$ is a well-defined $q$-power series invariant implies that any two expressions of the tail of $L$ are equal. This can be used to prove various $q$-identities and it was first utilized by Armond and Dasbach in \cite{Cody} where they showed that the Andrews-Gordon identity for the theta function can be proven using two methods to compute the of the tail of the $(2,2k+1)$ torus knots. In particular Armond and Dasbach use $R$-matrices and a combinatorial version of the quantum determinant formulation of Huynh and Le \cite{Thang1} developed by Armond \cite{Cody3} to compute the colored Jones polynomial of the $(2,2k+1)$ torus knot. These computations are then used to obtain two expressions of the tail associated with the $(2,2k+1)$ torus knot. The $q$-series they obtained are precisely the two sides of the Andrews-Gordon identity for the theta function. In this section we show that the skein theoretic techniques developed in this paper can be used to prove the following false theta function identity: 

\begin{eqnarray*}\sum\limits_{i=0}^{\infty} q^{ki^2+(k-1)i}-\sum\limits_{i=1}^{\infty} q^{k(i^2-i)+i}=(q,q)_{\infty}\sum\limits_{l_1=0}^{\infty}\sum\limits_{l_{2}=0}^{\infty}...\sum\limits_{l_{k-1}=0}^{\infty}\frac{q^{\sum\limits_{j=1}^{k-1}(i_j(i_j+1))}}{(q,q)^2_{l_{k-1}}\prod\limits_{j=1}^{k-2}(q,q)_{l_j}}
\end{eqnarray*}
with $k\geq $2 and $i_j=\sum\limits_{s=j}^{k-1}l_s$.\\

We also show that the same skein theoretic techniques can be applied to prove the Andrews-Gordon identity for the theta function \ref{and1}. Our techniques to prove the identity (\ref{and1}) has the advantage over the ones in \cite{Cody} in that our method restricts the tools used to prove this identity to skein theory.\\

Denote the torus knot $(2,f)$ in Figure \ref{TK} by $K_f$.
\begin{figure}[H]
  \centering
   {\includegraphics[scale=0.23]{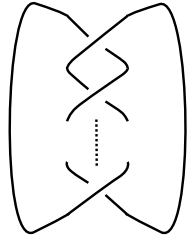}
    \caption{The $(2,f)$ torus knot}
  \label{TK}}
\end{figure}
\begin{theorem}
\label{usingskein}

\begin{enumerate}
\item For all $k\geq $2:
\begin{eqnarray*}\sum\limits_{i=0}^{\infty} q^{ki^2+(k-1)i}-\sum\limits_{i=1}^{\infty} q^{k(i^2-i)+i}=(q,q)_{\infty}\sum\limits_{l_1=0}^{\infty}\sum\limits_{l_{2}=0}^{\infty}...\sum\limits_{l_{k-1}=0}^{\infty}\frac{q^{\sum\limits_{j=1}^{k-1}(i_j(i_j+1))}}{(q,q)^2_{l_{k-1}}\prod\limits_{j=1}^{k-2}(q,q)_{l_j}}
\end{eqnarray*}
with $i_j=\sum\limits_{s=j}^{k-1}l_s$.
\item (The Andrews-Gordon identity for the theta function) For all $k\geq $1
\begin{eqnarray*}\sum\limits_{i=0}^{\infty}(-1)^i q^{k(i^2+i)}q^{i(i-1)/2}+\sum\limits_{i=1}^{\infty}(-1)^i q^{k(i^2-i)}q^{i(i+1)/2}=(q,q)_{\infty}\sum\limits_{l_1=0}^{\infty}\sum\limits_{l_{2}=0}^{\infty}...\sum\limits_{l_{k-1}=0}^{\infty}\frac{q^{\sum\limits_{j=1}^{k-1}(i_j(i_j+1))}}{\prod\limits_{j=1}^{k-1}(q,q)_{l_j}}
\end{eqnarray*}
with $i_j=\sum\limits_{s=j}^{k-1}l_s$.
\end{enumerate}

\end{theorem}
\begin{proof}
\begin{enumerate}
\item
 Using linear skein theory Kauffman bracket one can easily compute the colored Jones polynomial of $K_f$.  See \cite{Kauffman1} or \cite{Lickorish1} for more details about skein theory.
\begin{equation*}
\tilde{J}_{n,K_f}(q)/\Delta_n(q)=\frac{1}{\Delta_n(q)}\sum\limits_{i=0}^{n}(-1)^{f(n-i)} q^{f(2 i + 2 i^2 - 2 n - n^2)/4}\Delta_{2i}(q).
\end{equation*}
Hence,
\begin{equation}
\label{jonestorus}
\tilde{J}_{n,K_f}(q)/\Delta_n(q)\doteq_n \sum\limits_{i=0}^{n}(-1)^{fi} q^{f(i + i^2)/2}(q^{-i} - q^{1 + i}).
\end{equation}
If $f=2k$, then we can rewrite the previous equation:
\begin{eqnarray*}
\label{ann2}
\sum\limits_{i=0}^{n} q^{2k(i + i^2)/2}(q^{-i} - q^{1 + i})&=&\sum\limits_{i=0}^{n} q^{-i + i k + i^2 k}-\sum\limits_{i=0}^{n} q^{1 + i + i k + i^2 k}\\&=_n&\sum\limits_{i=0}^{\infty} q^{ki^2+(k-1)i}-\sum\limits_{i=1}^{\infty} q^{k(i^2-i)+i}
\end{eqnarray*}
Hence
\begin{equation}
\tilde{J}_{n,K_{2k}}(q)/\Delta_n(q)\doteq_n\Psi(q^{2k-1},q)
\end{equation}
On the other hand theorem \ref{cody thm} implies
    \begin{eqnarray*}\tilde{J}_{n,K_{2k}}(q)/\Delta_n(q)&\doteq_n&\frac{1}{\Delta_n(q)}
    \begin{minipage}[h]{0.18\linewidth}
         \vspace{0pt}
         \scalebox{0.30}{\includegraphics{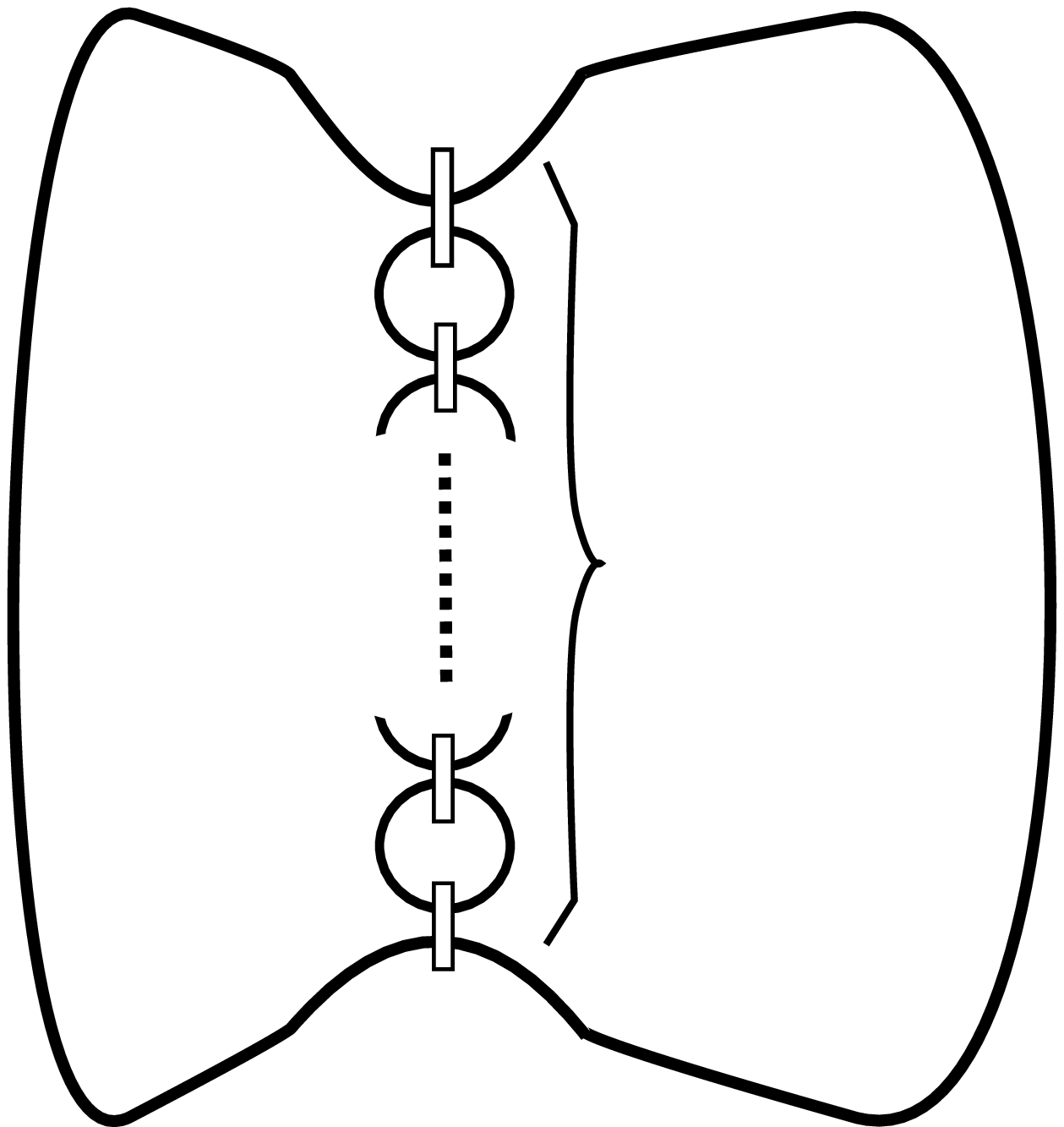}}
          \put(-37,+67){\footnotesize{$2k-1$}}
          \put(-40,+57){\footnotesize{ bubbles}}
           \end{minipage}
            \end{eqnarray*}
 and the tail of the skein element in the previous equation can be computed from corollary \ref{torususe} and we can obtain
 \begin{eqnarray}
\label{ann1} 
 \tilde{J}_{n,K_{2k}}(q)/\Delta_n(q)&\doteq_n&(q,q)_n\sum\limits_{l_1=0}^{n}\sum\limits_{l_{2}=0}^{n}...\sum\limits_{l_{k-1}=0}^{n}\frac{q^{\sum\limits_{j=1}^{k-1}(i_j(i_j+1))}}{(q,q)^2_{l_{k}}\prod\limits_{j=1}^{k-2}(q,q)_{l_j}}
            \end{eqnarray}
with $i_j=\sum\limits_{s=j}^{k-1}l_s$.            
Equations (\ref{ann2}) and (\ref{ann1}) yield the result.
\item
Substituting ($f=2k+1$) in (\ref{jonestorus}), we obtain
\begin{eqnarray*}
\label{jonestorus2}
\tilde{J}_{n,K_{2k+1}}(q)/\Delta_n(q)&\doteq_n& \sum\limits_{i=0}^{n}(-1)^{i} q^{(2k+1)(i + i^2)/2}(q^{-i} - q^{1 + i})\\&=&\sum\limits_{i=0}^{n}(-1)^{i} q^{-i/2 + i^2/2 + i k + i^2 k}-\sum\limits_{i=0}^{n}(-1)^{i} q^{1 + 3i/2 + i^2/2 + i k + i^2 k}\\&=&\sum\limits_{i=0}^{n}(-1)^{i} q^{-i/2 + i^2/2 + i k + i^2 k}-\sum\limits_{i=1}^{n}(-1)^{i} q^{i/2 + i^2/2 - i k + i^2 k}
\end{eqnarray*}
Hence
\begin{equation}
\label{final1}
\tilde{J}_{n,K_{2k+1}}(q)/\Delta_n(q)\doteq_nf(-q^{2k},-q)
\end{equation}
Theorem \ref{cody thm} implies
    \begin{eqnarray*}\tilde{J}_{n,K_{2k+1}}(q)/\Delta_n(q)&\doteq_n&\frac{1}{\Delta_n(q)}
    \begin{minipage}[h]{0.18\linewidth}
         \vspace{0pt}
         \scalebox{0.30}{\includegraphics{torusknotBstate}}
            \put(-35,+65){\footnotesize{$2k$}}
          \put(-43,+55){\footnotesize{ bubbles}}
           \end{minipage}
            \end{eqnarray*}
However corollary \ref{torususe} implies
\begin{eqnarray}
\label{final2}
\frac{1}{\Delta_n(q)}
\begin{minipage}[h]{0.26\linewidth}
         \vspace{0pt}
         \scalebox{0.30}{\includegraphics{torusknotBstate}}
            \put(-35,+65){\footnotesize{$2k$}}
          \put(-43,+55){\footnotesize{ bubbles}}
           \end{minipage}\doteq_n(q,q)_{\infty}\sum\limits_{l_1=0}^{\infty}\sum\limits_{l_{2}=0}^{\infty}...\sum\limits_{l_{k-1}=0}^{\infty}\frac{q^{\sum\limits_{j=1}^{k-1}(i_j(i_j+1))}}{\prod\limits_{j=1}^{k-1}(q,q)_{l_j}}
\end{eqnarray}
with $i_j=\sum\limits_{s=j}^{k-1}l_s$.
Equations (\ref{final1}) and (\ref{final2}) yield the result.
\end{enumerate}
\end{proof}

\end{document}